\documentclass[11pt]{article}
\usepackage{graphicx}
\usepackage[T1]{fontenc}
\usepackage[utf8]{inputenc}
\usepackage{lmodern}
\usepackage[dvipsnames]{xcolor}
\usepackage[section]{placeins}
\usepackage{multirow}

\usepackage{amsmath}
\usepackage{amsfonts}
\usepackage{amsthm}
\usepackage{amssymb}
\usepackage{comment}
\usepackage{pgfplots}
\usepackage{color}
\usepackage{xcolor}
\usepackage[normalem]{ulem}
\usepackage{caption}
\usepackage{subcaption}
\usepackage{url}
\usepackage{rotating}
\usepackage{setspace}
\usepackage{mathtools}
\mathtoolsset{showonlyrefs}
\usepackage{longtable}
\usepackage{float}
\usepackage{color,soul}
\graphicspath{{./images/}}

\usepackage{tikz}

\usepackage{booktabs}
\usepackage[scale=2]{ccicons}

\usepgfplotslibrary{dateplot}

\usepackage{xspace}

\usepackage{pdfpages} 

\newtheorem{thm}{Theorem}[section]
\newtheorem{lemma}[thm]{Lemma}

\theoremstyle{definition}

\usepackage[colorlinks=true,linkcolor=blue,citecolor=blue,urlcolor=black]{hyperref}
\usepackage[numbers,square]{natbib} 

\usepackage[margin=3cm]{geometry}
\setcitestyle{square}
\usepackage{titlesec}

\title{Travelling Waves 
in a Mathematical Model \\for 
Oncolytic Virotherapy}
\author{Negar Mohammadnejad \thanks{Department of Mathematical and Statistical Sciences, University of Alberta, Edmonton, AB, Canada, T6G 2R3, \texttt{negar3@ualberta.ca}} , Thomas Hillen \thanks{Department of Mathematical and Statistical Sciences, University of Alberta, Edmonton, AB, Canada, T6G 2R3,  \texttt{thillen@ualberta.ca} (Corresponding author)}}
\date{}
\usepackage{pgfplots}
\pgfplotsset{compat=1.18}
\begin{document}

\maketitle

\begin{abstract}
    Oncolytic virotherapy (OVT) is a promising cancer treatment strategy in which engineered viruses selectively infect and destroy tumor cells. Motivated by the biological mechanisms underlying viral spread and tumor invasion into the tissue, we analyze a non-cooperative reaction–diffusion model capturing the invasion of tumor tissue by oncolytic viruses. Using carefully constructed upper and lower solutions together with Schauder’s fixed point theorem, we establish the existence of positive travelling‐wave solutions. In particular, we identify a minimal wave speed value $\bar{c}$
  such that positive travelling waves exist for all $c \ge \bar{c}$. Our analysis also highlights parameter regions where the existence of travelling waves remains ambiguous, suggesting new mathematical questions about the propagation of viral treatments through tumor environments.
\end{abstract}
\section{Introduction} 

In 1991, the first genetically engineered oncolytic virus (OV)---a thymidine kinase–deficient herpes simplex virus type~1 (HSV-1) mutant---was developed and used to treat malignant glioma in nude mice~\cite{viro_intro}. Oncolytic virotherapy (OVT) has since emerged as a promising cancer treatment strategy in which oncolytic viruses are introduced into a patient to selectively infect and eliminate tumor cells \cite{LinDanni2023Ovbp}. 

Over the last two decades, mathematical models for oncolytic viruses were developed and analysed \cite{WodarzKamorva009,KomorovaWodarz2010,Tian,eftimie,BhattDarshakKartikey2022Mtsd,storey,Jenner2019,MorselliDavid2023AaCM,arwa,alturqwi}.
Besides clinically relevant conclusions on the administration of oncolytic virotherapy, it turned out that the basic oncolytic virus model has very interesting, and non-trivial, mathematical properties. The model is often formulated as a system of reaction-diffuion equations and it can lead to oscillations, spatial pattern formation, invasion fronts, spiral waves  and spatio-temporal chaotic behavior \cite{POOLADVAND2021108520,eftimie,Tejas2025}. 
The analysis of these patterns has mostly focussed on numerical simulations. In the present work, we rigorously establish the existence of travelling wave solutions for the basic oncolytic virus reaction--diffusion model, which describes the spatial interactions between uninfected cancer cells $C$, infected cancer cells $I$, and free virions $V$ on an unbounded spatial domain.

We notice that the three-component reaction-diffusion system is non-cooperative, hence the standard theory for monotone or cooperative systems (e.g.,~\cite{FangJian2009MWfP, LiangXing2007Asos, MaShiwang2001TWfD, WeinbergerHansF.2002Aold, WuJianhong2001TWFo, WangHaiyan2011SSaT_epidemy_travekingwave}) cannot be applied. Instead,  we construct appropriate upper and lower solutions and apply Schauder's fixed point theorem to show existence of travelling waves with speeds in an interval $[\bar c, \infty)$. 

Before presenting our analysis, we begin with a brief overview of existing oncolytic virotherapy (OVT) models, followed by a short introduction to travelling–wave methods. 
In the sections that follow, we introduce the oncolytic virotherapy model studied here in Section~\ref{sec:model} and provide a preliminary analysis. We then state the necessary conditions that allow us to construct the upper and lower solutions in preparation for the main existence result. The proof of the main theorem is completed in Section~\ref{sec:proof}. Finally, in Section~\ref{sec:numerics} we present numerical simulations of the full PDE system and compare them with our analytical findings.

\subsection{Oncolytic Virotherapy Modeling}
Early theoretical work on oncolytic virotherapy (OVT) was shaped by the models of Wodarz and Komarova \cite{Woodarz2002,WodarzKamorva009,KomorovaWodarz2010}, who established a general ODE framework capturing key features of tumor growth and viral spread. Their analyses revealed that OVT dynamics fall into fast- and slow-spread regimes, and that treatment outcomes depend sensitively on assumptions about infection kinetics. These studies provided a foundation for later spatial and PDE-based models of tumor–virus interactions.
Building on these ideas, Tian \cite{Tian} introduced a three-variable ordinary differential system tracking healthy cancer cells, infected cells, and free virus, showing how viral burst size can determine therapeutic success. More detailed spatial models soon followed. Pooladvand et al.\cite{POOLADVAND2021108520} developed a three-dimensional, radially symmetric PDE model for adenovirus therapy in solid tumors, allowing susceptible cells, infected cells, and free virus to diffuse. Their analysis showed that increasing viral infectivity alone does not guarantee tumor elimination—consistent with experimental observations that virotherapy rarely succeeds as a standalone treatment. Their work helped motivate further spatial modelling, including the diffusion-based system of \cite{arwa}, which examined temporal and spatial infection patterns in the absence of immune effects.
Additional computational studies have explored why OVT can fail. Bhatt et al. \cite{BhattDarshakKartikey2022Mtsd} combined immersed-boundary and Voronoi-based simulations to identify key limiting mechanisms: rapid loss of infected cells (leading to premature viral clearance), emergence of virus-resistant tumor cells, and insufficient viral spread rates. Morselli et al. \cite{MorselliDavid2023AaCM} used a stochastic agent-based model to examine how spatial constraints and microenvironmental factors shape viral propagation, comparing random versus pressure-driven viral movement and reproducing infection patterns seen in earlier biological and theoretical studies.
Although these agent-based and high-dimensional spatial models provide rich biological insight, their computational cost often limits extensive analytical investigation. For this reason, reaction–diffusion models remain valuable for uncovering the spatial structure of viral spread and identifying conditions favoring tumor control or eradication.
Baadbulla et al. \cite{arwa} extend the framework of Pooladvand et al. \cite{POOLADVAND2021108520}. Pooladvands original model—formulated in 3-D radial symmetry—displayed transitions to coexistence states and oscillatory dynamics through a Hopf bifurcation, matching oscillations previously reported in experimental systems. Baabdulla et al \cite{arwa} removed the radial symmetry and analyzed the full 2-D and 3-D diffusion problem. They computed invasion speeds numerically, and compared spatial infection profiles. They found  complex spatiotemporal behaviours that are often seen in excitable media \cite{Kuramoto,Sandstede}. Their analysis highlighted how spatial heterogeneity shapes viral spread and helps clarify the circumstances under which OVT can suppress—or nearly eradicate—tumor mass. 

In the absence of oncolytic viruses (OVs), tumors typically develop an immunosuppressive microenvironment that suppresses the host immune response. The introduction of OVs, however, induces pro-inflammatory signalling within the tumor and can substantially alter tumor--immune dynamics. Motivated by the work of Al-Tuwairqi et al.~\cite{alturqwi}, a reaction--diffusion model was proposed in \cite{NegarThomas1} to capture OVT--immune system interactions. This model extends that of Baabdulla et al.~\cite{arwa} by incorporating a fourth variable $Y$ representing the immune response.                                 Their analysis demonstrated the existence of stacked travelling waves and, through the computation of associated wave speeds, indicated that OVT as a monotherapy is unlikely to achieve complete tumor eradication. This was further explored in~\cite{NegarTHomas2}, where strategies to enhance the efficacy of OVT were investigated, with particular emphasis on the importance of combination treatments involving immunotherapies. These studies highlight that explicitly modelling the immune response is essential for understanding the full therapeutic potential of OVT.

Eftimie et al.~\cite{eftimie} developed an ODE model incorporating both lymphoid and peripheral immune compartments to describe the interactions among uninfected and infected tumor cells, memory and effector immune cells, and two viral types (Ad and VSV). Their model successfully reproduced experimental tumor growth and immune-response patterns, identifying conditions under which sustained tumor elimination is possible. Importantly, they showed that viral persistence alone cannot clear the tumor without a coordinated anti-tumor immune response.

Storey et al.~\cite{storey} further distinguished the roles of innate and adaptive immunity in OVT, noting that the innate response may act too rapidly and clear the virus before sufficient tumor infection occurs. To overcome this limitation, combination therapies pairing OVT with PD-1/PD-L1 checkpoint inhibition have been explored. Such approaches mitigate T-cell exhaustion, enhance immune-mediated tumor clearance, and reduce the viral infectivity threshold required for therapeutic success. Their findings also underscore the importance of treatment scheduling: administering a second viral dose too early can redirect the immune response toward viral, rather than tumor, elimination, thereby reducing overall therapeutic efficacy.

 \subsection{Travelling Waves} 

Modeling with partial differential equations (PDEs)—and reaction–diffusion equations in particular—has had a profound influence on the mathematical study of biological populations, with significant applications in cell biology \cite{SherrattJA1992Ewht_murray_cell_traveling, LagergrenJohnH.2020Bnng_Cell_travelingwave}, ecology \cite{SkellamJ.G.1991Rdit_ecology_travelingwave, LevinSimonA.2003TEAE_ecology_travelingwave}, and epidemiology \cite{ZhangTianran2014Eotw_wang_epidomy_travelingwave, WangHaiyan2011SSaT_epidemy_travekingwave}. Many works (e.g., \cite{murray, Britton, volpretVolpretVolpret, smoller_alma991043704819909116}) have investigated travelling waves and spatial dynamics for $n$-dimensional $(n\ge 1)$ reaction–diffusion systems of the form
\begin{align}\label{base_diffusion}
    \frac{\partial u_i}{\partial t}
    = D_i \Delta u_i + f_i(u_1,\ldots,u_n),
    \qquad t \ge 0,\ x \in \mathbb{R},
\end{align}
for $1\le i\le n$. Here $u_i(x,t)$ denotes the density of species of type $i=1,\dots,n$, $D_i$ denotes their diffusion coefficient, and $f_i(u_1,\dots,u_n)$ describes their birth, death, and other interactions.  
Such models give rise to travelling–wave solutions that represent invasion fronts. The most prominent model arises for one species $n=1$ and is known as the Fisher–KPP model \cite{murray, Britton}. It is a classical and widely studied prototype for biological invasions and describes the spatiotemporal dynamics of a population whose individuals migrate via linear diffusion and proliferate according to logistic-type  growth.

In cancer biology, reaction–diffusion frameworks have been especially useful in describing tumor invasion, particularly when the spread of malignant cells is coupled to degradation of surrounding tissue \cite{El-HachemMaud2021Twao_travelingwave}.  

The system \eqref{base_diffusion} is called \emph{cooperative} if each $f_i(u)$ is non-decreasing in all components of $u$, with the possible exception of the $i$th one. Cooperative systems have been extensively studied, and a general travelling–wave theory is available in this setting \cite{FangJian2009MWfP, WeinbergerHansF.2002Aold, LiangXing2007Asos}.

For \emph{non-cooperative} systems—characterized by kinetic Jacobians that include negative off-diagonal entries—this general theory does not apply. The model studied in this work (system \eqref{eqn:mainPDE}) is non-cooperative and the above mentioned theory is not available. Instead we follow the framework of \cite{ThiessenRyan2026Twst} to construct upper and lower solutions that connect correctly to the left boundary and remain uniformly bounded throughout the domain. This framework then enables the application of Schauder’s fixed-point theorem, yielding the existence of travelling–wave solution for speeds in an interval $[\bar c, \infty)$.

The non-monotonicity of our system introduces significant technical challenges and restricts us to parameter regimes where existence can be rigorously established. When determining the minimal wave speed, we encounter parameter regions in which the existence—or nonexistence—of travelling waves cannot be conclusively determined. Such indeterminate parameter zones are common in travelling–wave analysis and motivate future work aimed at resolving these ``foggy'' regimes where classical arguments do not decisively settle the question of the wave's existence.

\section{The mathematical model}\label{sec:model}
In this section, we will study the A. Baabdulla \cite{arwa} reaction diffusion
model on a rectangular or smooth domain $\Omega \subset \mathbb{R}^n$.\\
The proposed model from \cite{arwa} is given by 
\begin{align}
\label{eqn:main_PDE}
\notag
\frac{dC}{dt}=&D \Delta C 
+ r\,C\left(1-\frac{C+I}{L}\right)
- \beta C V,\\
\frac{dI}{dt}=&D \Delta I
+ \beta C V 
- \eta I,\\ \notag
\frac{dV}{dt}=&D_V \Delta V 
+ \eta b\, I
- \omega V.
\end{align}

In this formulation, \(C\) denotes the population of uninfected cancer cells, \(I\) the population of infected cancer cells, and \(V\) the concentration of free oncolytic virus.  

We consider system \eqref{eqn:main_PDE} together with the initial conditions
\begin{align*}
    C(x,0)=C_0>0, \qquad 
    I(x,0)=I_0=0, \qquad 
    V(x,0)=V_0>0,
\end{align*}
and homogeneous Neumann boundary conditions on $\partial \Omega$,
\begin{align*}
    n \cdot \nabla C = 
    n \cdot \nabla I = 
    n \cdot \nabla V = 0 
    \qquad \forall\, x \in \partial \Omega.
\end{align*}

In the system of equations~\eqref{eqn:main_PDE}, the diffusion coefficients of the susceptible and infected tumor cell populations are denoted by $D$, while the free virus diffuses with coefficient $D_V$. The operator $\Delta$ denotes the Laplacian, i.e.\ the sum of all second–order spatial derivatives.

In the first equation, the second term represents the logistic proliferation of the uninfected cancer cells in the absence of viral therapy, while the third term, $-\beta C V$, models the loss of susceptible tumor cells due to infection by free virus. Free virions interact with uninfected cancer cells and transfer them into the infected class. Consequently, the same interaction appears with opposite sign in the second equation as the term $+\beta C V$, which describes the influx of cells into the infected population. Thus, the second equation accounts for the accumulation of infected tumor cells through infection, while the term $-\eta I$ represents their removal due to lysis.

In the third equation, the term $\eta b I$ incorporates the burst size $b$, describing the number of viral particles released upon lysis of infected cells, and the final term, $-\omega V$, corresponds to the natural clearance rate of free virus.

To simplify our analysis, we apply the non-dimensionalization method by considering
\begin{align*}
    \hat{C}=\frac{C}{L}, \quad \hat{I}=\frac{I}{L}, \quad\hat{V}=\frac{\beta V}{r},\quad \hat{t}=r t,\quad \hat x=  x\sqrt{ \frac{r}{D_V}} \  \text{, and} \quad \hat D= \frac{D}{D_V}.
\end{align*} 
After dropping the hats, we get the non-dimensional system:

 \begin{align}
\label{eqn:immune}
\notag
\frac{dC}{dt}=&D \Delta C+{C}\left(1-C-I\right)- {C}{V},\\ \notag
\frac{dI}{dt}=& D \Delta I+CV- a I,\\
\frac{dV}{dt}=& \Delta V+\theta I-\gamma V,
\end{align}

where 
\begin{align*}
 a= \frac{\eta}{r},\quad \theta= \frac{\eta b \beta L }{r^2},\quad \gamma= \frac{\omega}{r}.\end{align*}


\subsection{Analysis of the model}
We initially examine the aforementioned model in the spatially homogeneous scenario,
where all spatial dependencies are disregarded. Consequently, the system governing
$C(t), I(t), V (t)$ simplifies to:

 \begin{align}
\label{eqn:base model}
\notag
\frac{dC}{dt}=&{C}\left(1-C-I\right)- {C}{V},\\ 
\frac{dI}{dt}=& CV- a I, \\ \notag
\frac{dV}{dt}=& \theta I-\gamma V.
\end{align}

We obtain the following equilibrium points for system \eqref{eqn:base model}
\begin{align*}
    E_0=&(0,0,0),\\
    E_1=&(1,0,0),\\ E_2=&(C^*,I^*,V^*)=\left(\frac{a \gamma }{\theta}, \frac{\gamma( \theta-a \gamma )}{\theta(\gamma +\theta)}, \frac{\theta-a \gamma}{\gamma+\theta}\right).
\end{align*}

The analysis of the equilibrium points $E_0$, $E_1$, and $E_2$ was carried out in detail in \cite{arwa}. We show a bifurcation diagram in Figure \ref{fig:bifurcation} with $\theta$ as bifurcation parameter. For low values of $\theta$, the cancer only state $E_1$ is stable, and by increasing $\theta$, it undergoes a transcritical bifurcation at $\theta_t = \gamma a$, and the coexistence state, $E_2$, becomes biologically relevant and stable. Upon further increase of $\theta$, we obtain a Hopf bifurcation at $\theta_H$ with periodic orbits arising for $\theta>\theta_H$. For our further analysis, we assume the coexistence equilibrium exists, hence we assume throughout: 
\begin{equation}\label{theta>agamma}
\theta > \gamma a.
\end{equation}
\color{black}

\begin{figure}
    \centering
    \includegraphics[width=0.5\linewidth]{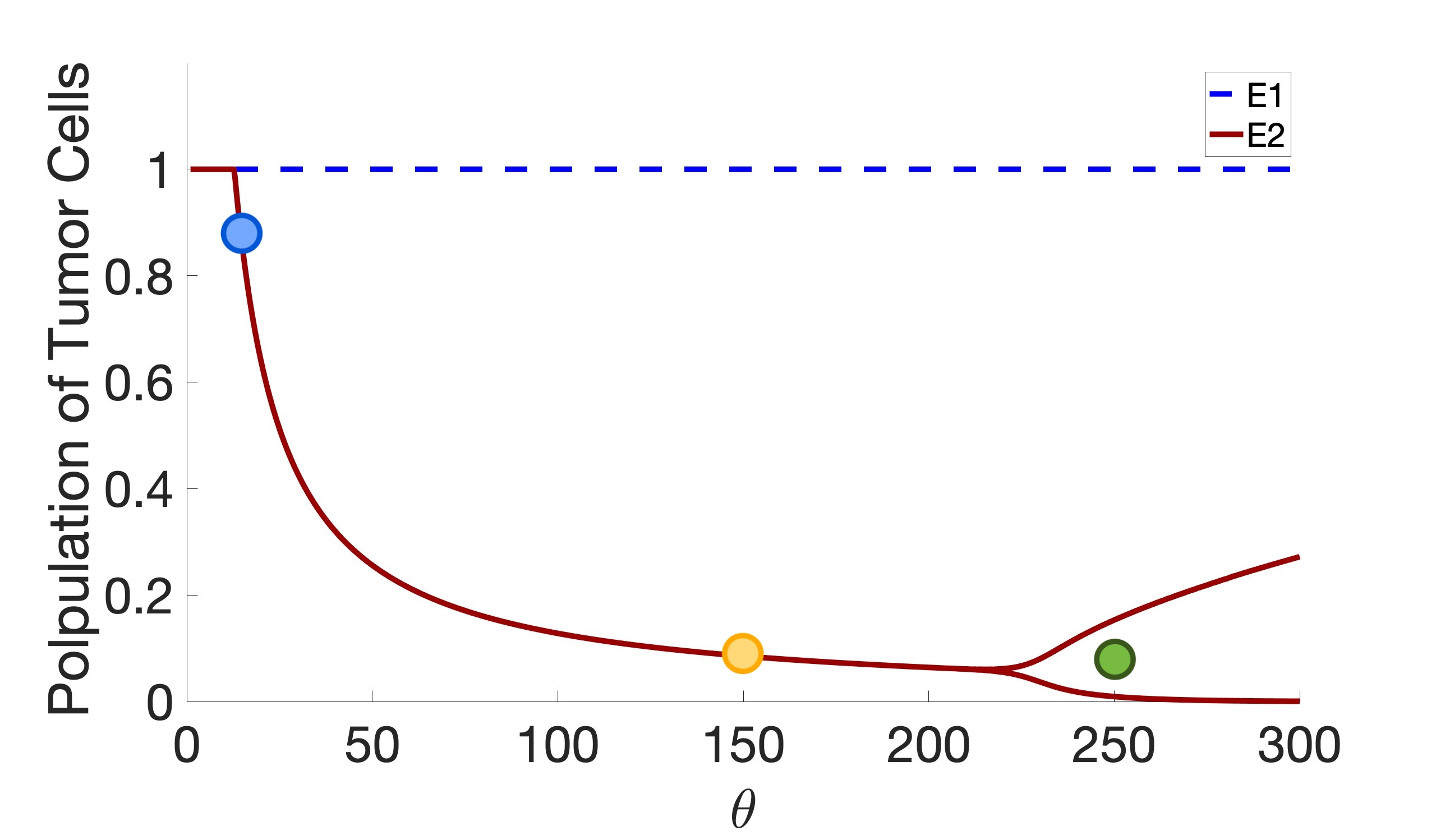}
 \caption{
Bifurcation diagram of the model~\eqref{eqn:base model} with bifurcation parameter $\theta$. 
The equilibrium $E_{1}$ loses stability at $\theta \approx 12.7$, while $E_{2}$ becomes 
biologically relevant and remains stable up to the Hopf bifurcation at 
$\theta_{H} \approx 220$. Beyond this point, the system exhibits oscillatory dynamics. 
The parameter values used here are $a = 0.96$ and $\gamma = \frac{40}{3}$. 
The circles mark the $\theta$ values selected in Section~\ref{sec:numerics}
for numerical simulations.
}
    \label{fig:bifurcation}
\end{figure}

\section{Travelling wave existence}
In our main result, Theorem~\ref{mainTHeorem}, we establish conditions under which the system
\begin{align}
\label{eqn:mainPDE}
\notag
\frac{dC}{dt}=& D \Delta C + C(1 - C - I) - C V,\\
\frac{dI}{dt} =& D \Delta I + C V - a I,\\
\notag
\frac{dV}{dt} =& \Delta V + \theta I - \gamma V,
\end{align}
admits travelling wave solutions on $-\infty<x<\infty$.   

\subsection{The travelling wave problem}
Since our interest lies in describing the invasion of the viral treatment into an established tumor, it is convenient to shift the equilibrium corresponding to the untreated cancer to the origin. To do so, we introduce the change of variables
\[
B \coloneqq 1 - C.
\]
Under this transformation, the relevant equilibrium $E_1$ becomes $(0,0,0)$, allowing us to apply standard techniques for studying invasion into the ``disease-free'' state $(0,0,0)$. Substituting $C = 1 - B$ into \eqref{eqn:mainPDE} yields the equivalent system
\begin{align}
\label{eqn:mainPDE_B}
\notag
\frac{dB}{dt} &= D \Delta B + (B-1)(B-I) + (1-B)V,\\
\frac{dI}{dt} &= D \Delta I + (1-B)V - a I,\\
\notag
\frac{dV}{dt} &= \Delta V + \theta I - \gamma V.
\end{align}

To derive the travelling wave equations, we introduce the coordinates $\xi = x + ct$. A travelling wave solution has the form
\[
(B(x,t), I(x,t), V(x,t)) = (B(\xi), I(\xi), V(\xi)),
\]
and substituting this into \eqref{eqn:mainPDE_B} leads to the ODE system
\begin{align}
\label{ODE_B}
c B' &= D B'' - B + B^2 - B I - B V + I + V, \nonumber\\
c I' &= D I'' - a I + (1-B)V,\\
c V' &= V'' + \theta I - \gamma V. \nonumber
\end{align}

The boundary conditions require special attention. At the leading edge of the wave, corresponding to the pre-invasion cancer state, we impose
\[
(C(-\infty), I(-\infty), V(-\infty)) = (1,0,0).
\]
Behind the wave, i.e.\ as $\xi \to \infty$, several behaviours are possible. The solution may converge to the equilibrium $E_2$, but as illustrated in Figure~\ref{fig:bifurcation}, $E_2$ undergoes a Hopf bifurcation for sufficiently large~$\theta$. In such cases,  the wave approaches a periodic orbit surrounding $E_2$ rather than the equilibrium itself. Accordingly, we impose the general boundedness condition
\[
\|(C(\infty), I(\infty), V(\infty))\| < \infty.
\]
The associated boundary conditions, expressed in $(B,I,V)$ variables, are
\begin{equation}\label{boundarycondition}
(B(-\infty),\, I(-\infty),\, V(-\infty)) = (0,0,0),
\qquad \text{and} \qquad
\bigl\|(B(\infty),\, I(\infty),\, V(\infty))\bigr\| < \infty.
\end{equation}

From this point forward, we work exclusively with the transformed system \eqref{ODE_B}, as it simplifies the analytical treatment of the travelling wave problem.

 The system \eqref{ODE_B} is a second-order inhomogeneous ODE and we find an alternative formulation using the method of variation of parameters. This transformation has shown to be useful for travelling wave analysis in the work of Fang and Zhao \cite{FangJian2009MWfP}. We define $U\coloneqq(B,I,V)$, matrix $\mathcal{D}$, and the function $f$ as  \begin{align}\label{f_equation}f(U)= f(B,I,V)=\begin{pmatrix}
    f_1(B,I,V)\\  f_2(B,I,V)\\  f_3(B,I,V)
\end{pmatrix}= \begin{pmatrix}
    -B+B^2-BI-BV+I+V\\ -a I+(1-B)V\\ \theta I-\gamma V
\end{pmatrix}, \ \mathcal{D}\coloneqq\begin{pmatrix}
    D &0&0\\ 0& D&0 \\ 0&0& 1 
\end{pmatrix}, \end{align} 
so system \eqref{ODE_B} can be written as  
\begin{align} \label{eq:short1}
    \mathcal{D}U''-cU'=-f(U).\end{align}
    
Let \begin{align}
    F(U)=\begin{pmatrix}
    F_1(U)\\  F_2(U)\\  F_3(U)
\end{pmatrix}= \alpha U+f(U),
\end{align}  where $\alpha$ is a positive free parameter. This definition allows us to rewrite \eqref{eq:short1} as \begin{align}\label{eq:short2}
    \mathcal{D}U''-cU'- \alpha U=-F(U).\end{align} 
    
    We apply the variation of parameters such that 
    \begin{align} \label{c_operator}
        \nonumber B(\xi)&=\dfrac{1}{\sqrt{c^2+4 \alpha D}} \left( \int_{-\infty}^{\xi} e^{(\xi- \eta)\delta_1} F_1(B(\eta),I(\eta),V(\eta))\ d \eta +  \int_{\xi}^{\infty} e^{(\xi- \eta)\delta_2} F_1(B(\eta),I(\eta),V(\eta))\ d \eta \right),\end{align} 
        where 
     \begin{align}
    \delta_1=\dfrac{c- \sqrt{c^2+4 \alpha D}}{2 D},\quad \delta_2=\dfrac{c+ \sqrt{c^2+4 \alpha D}}{2 D}.\end{align}
     
     \begin{align}\label{I_operator} I(\xi)&=\dfrac{1}{\sqrt{c^2+4 \alpha D}} \left( \int_{-\infty}^{\xi} e^{(\xi- \eta)\sigma_1} F_2(B(\eta),I(\eta),V(\eta))\ d \eta +  \int_{\xi}^{\infty} e^{(\xi- \eta)\sigma_2} F_2(B(\eta),I(\eta),V(\eta))\ d \eta \right),\end{align} where \begin{align} \label{I_operator2}\sigma_1=\dfrac{c- \sqrt{c^2+4 \alpha D}}{2 D},  \quad \sigma_2=\dfrac{c+ \sqrt{c^2+4 \alpha D}}{2 D}. \end{align}\begin{align} \label{V_operator} V(\xi)&=\dfrac{1}{\sqrt{c^2+4 \alpha }} \left( \int_{-\infty}^{\xi} e^{(\xi- \eta)\zeta_1} F_3(B(\eta),I(\eta),V(\eta))\ d \eta +  \int_{\xi}^{\infty} e^{(\xi- \eta)\zeta_2} F_3(B(\eta),I(\eta),V(\eta))\ d \eta \right),
    \end{align} where 
    \begin{align} \label{V_operator2}\zeta_1=\dfrac{c- \sqrt{c^2+4 \alpha }}{2 }, \quad \zeta_2=\dfrac{c+ \sqrt{c^2+4 \alpha }}{2 } .\end{align}

These integral formulations will allow us to define an operator framework, the fixed points of which correspond to traveling wave solutions.

\subsection{Dispersion relation for the invasion wave speed}
    Before stating the main theorem, we first analyze the linearization of \eqref{ODE_B} in order to characterize the invasion speed and the exponential growth rate at the leading edge of the wave.  
This analysis leads to a set of conditions that must be satisfied for traveling wave solutions to exist.

We linearize the system \eqref{ODE_B} about the $(0,0,0)$ equilibrium \begin{align}
    c B'&= DB''-B+I+V,\\ cI'&= DI'' -a I+V, \\ cV'&= V'' +\theta I - \gamma V.   
\end{align}

At the leading edge of the wave, we assume exponential growth as $B= {\nu}_be^{\rho \xi}$, $I=\nu_i e^{\rho \xi}$, and $V= \nu_v e^{\rho \xi}$ for a constant vector $\nu=(\nu_b,\nu_i,\nu_v)$, and a growth rate $\rho>0$. Substituting the solution into the above equation yields the algebraic system \begin{align}
    c\rho \nu =\underbrace{\begin{pmatrix}
        D & 0& 0 \\ 0 & D &0 \\ 0 & 0 &1
    \end{pmatrix}}_{\eqqcolon \mathcal{D}} \rho^2 \nu + \underbrace{\begin{pmatrix}
        -1 &1&1\\ 0&-a &1\\ 0&\theta & -\gamma
    \end{pmatrix}}_{= f'(0,0,0)} \nu,
\end{align} where $f'(0,0,0)$ is the Jacobian of $f$ from \eqref{f_equation} at equilibrium $(0,0,0)$. Defining the matrix $A(\rho)= \mathcal D \rho^2+f'(0,0,0)$, we retrieve an eigenvalue problem 
\begin{align}
    (A(\rho)- c \rho \mathbb{I})\nu= 0, 
\end{align} with the eigenvalue \begin{align}
    \lambda= c \rho.
\end{align} 

We explicitly compute the eigenvalues and the eigenvectors of the matrix $A$. The matrix $A$ is given by \begin{align}A(\rho)=
    \begin{pmatrix}
        D \rho^2 -1 &+1 &+1 \\ 0 & D \rho^2 -a &1 \\ 0 & \theta & \rho^2 -\gamma
    \end{pmatrix}.
\end{align}

Thus, the eigenvalues are given by \begin{align}\label{lambdas}\notag \lambda_{1}&= D \rho^2-1, \\  \lambda_{2,3}&= \frac{\rho^2+D \rho^2-a- \gamma \pm \sqrt{(\rho^2+ D \rho^2-a- \gamma)^2- 4(D \rho^4- D \rho^2 \gamma- a \rho^2+ a \gamma- \theta)}}{2}\\ \notag &= \frac{\rho^2+D \rho^2-a- \gamma \pm \sqrt{(\rho^2- D \rho^2+a- \gamma)^2+ 4 \theta}}{2},
\end{align} and the eigenvectors

\begin{align}\label{nu_s}
    \nu_1= \begin{pmatrix}
        1\\0\\0
    \end{pmatrix}, \ \nu_2= \begin{pmatrix}
       \dfrac{(\lambda_2+1+a-D \rho^2)}{(\lambda_2+a-D \rho^2 )(\lambda_2+1-D\rho^2 )} \\ \dfrac{1}{\lambda_2+a-D \rho^2} \\ 1
    \end{pmatrix}, \ \nu_3= \begin{pmatrix}
       \dfrac{(\lambda_3+1+a-D \rho^2)}{(\lambda_3+a-D \rho^2 )(\lambda_3+1-D\rho^2 )} \\ \dfrac{1}{\lambda_3+a-D \rho^2} \\ 1
    \end{pmatrix}.
\end{align}

The eigenvalue $\lambda_{1}(\rho)$ in \eqref{lambdas} has the associated eigenvector 
$\nu_{1} = (1,0,0)$, which corresponds to a direction purely in the cancer component: it involves neither infected cells nor virus. In fact, $\lambda_{1}$ is determined entirely by the $B$–equation and represents the spatial decay rate of the stable manifold of the classical Fisher--KPP equation at the equilibrium $C=1$ (equivalently, $B=0$). It is well known that this mode cannot generate a bounded travelling wave solution, and we therefore exclude this case from further consideration (see, for example, \cite{Britton, murray}).

Since our objective is to analyse viral invasion into a cancer population, we restrict attention to the remaining two eigenvalues, $\lambda_{2}(\rho)$ and $\lambda_{3}(\rho)$. Because $\lambda_{2} > \lambda_{3}$ for all $\rho>0$, the relevant branch for constructing travelling–wave solutions is $\lambda_{2}$.

The wave speed associated with the eigenvalue $\lambda_{2}(\rho)$ is
\[
    S(\rho) = \frac{\lambda_{2}(\rho)}{\rho}
    = \frac{\rho^2 + D\rho^2 - a - \gamma}{2\rho}
    + \frac{\sqrt{(\rho^2 - D\rho^2 + a - \gamma)^2 + 4\theta}}{2\rho}.
\]
We then introduce the quantity
\begin{align} \label{barc definiton}
    c_m = \inf_{\rho>0} S(\rho),
\end{align}
which represents the minimum of this speed curve. 
    
    \begin{lemma}
        The speed curve $S(\rho)$ decreases as $\rho$ is close to zero and tends to infinity as $\rho \to 0$. We have $\lim_{\rho \to \infty}S(\rho)= +\infty$ and ${S}(\rho)$ has at least one global minimum. 
    \end{lemma} 
  \begin{proof}
      The limit $\displaystyle \lim_{\rho \to 0} S(\rho) = \infty$ follows by applying \textit{L'Hôpital's rule} to the defining expression of $S(\rho)$.
Moreover, since $\displaystyle \lim_{\rho \to \infty} S(\rho) = \infty$ and $S$ is continuous on $(0,\infty)$, it follows that
$S$ admits at least one global minimizer ${\rho}_m \in (0,\infty)$ such that
\begin{align}
    S({\rho}_m) = \inf_{\rho>0} S(\rho) =: {c}_m.
\end{align}
If multiple minimizers exist, we select the smallest $\rho$ and denote it again by
\begin{align}\label{rho_m_definiton}
    {\rho}_m = \min\{\rho_i : S(\rho_i) = {c}_m\}.
\end{align}
Hence, the minimal value ${c}_m$ is attained at $\rho ={\rho}_m$.
 \end{proof}
 In the following lemma, we show that $\lambda_2$ is positive. This will show that the function $S(\rho)$ is positive and the minimum wave speed has positive values.

\begin{lemma}
  Assume \eqref{theta>agamma}. The eigenvalue  $\lambda_2(\rho)$ is positive for $\rho>0$. 
\end{lemma}
\begin{proof}
    Using $\theta> a\gamma$, we prove that \begin{align}\label{equivalent_inequality}
            \lambda_2&=  \frac{\rho^2+D \rho^2-a- \gamma}{2} +\dfrac{\sqrt{(\rho^2- D \rho^2+a- \gamma)^2+ 4 \theta}}{2}\\& \geq  \frac{\rho^2+D \rho^2-a- \gamma}{2} +\dfrac{ \sqrt{(\rho^2- D \rho^2+a- \gamma)^2}}{2}.
        \end{align}
Thus, we need to prove 
\begin{align}
    \frac{\rho^2+D \rho^2-a- \gamma + \sqrt{(\rho^2- D \rho^2+a- \gamma)^2+4a \gamma}}{2}>0 \quad \text{for } \quad \rho>0.
\end{align}
We first note that when $D=1$, we have 
\begin{align}\label{D=1}
   \lambda_2> \frac{\rho^2+ \rho^2-a- \gamma + \sqrt{(a- \gamma)^2+4a \gamma}}{2}=\frac{2\rho^2-a- \gamma + \sqrt{(a+ \gamma)^2}}{2}= \rho^2>0.
\end{align}
For $D \ne 1$,
two possible cases exist that need to be handled separately.
\begin{enumerate}
    \item $\rho^2+D \rho^2-a- \gamma>0$, \item $\rho^2+D \rho^2-a- \gamma<0$.
\end{enumerate}

In the first case, both terms in \eqref{equivalent_inequality} are positive, which implies directly that $\lambda_2>0$.  
In the second case, noting that $D>0$, we can rewrite the inequality as  
\begin{align} \label{Golden_assumption}
    \rho^2 + D\rho^2 - a - \gamma < 0 
    \qquad\Longleftrightarrow\qquad 
    0 < D\rho^2 < a + \gamma - \rho^2.
\end{align}
Under the assumption \eqref{Golden_assumption}, we now determine conditions guaranteeing $\lambda_2>0$.

To ensure that 
\begin{align}
    \lambda_2
    =\frac{ \sqrt{(\rho^2 - D\rho^2 + a - \gamma)^2 + 4a\gamma}
        + (\rho^2 + D\rho^2 - a - \gamma)}{2}
    >0,
\end{align}
we require
\begin{align}
    \sqrt{(\rho^2 - D\rho^2 + a - \gamma)^2 + 4a\gamma}
    > -(\rho^2 + D\rho^2 - a - \gamma).
\end{align}
Squaring both sides yields
\begin{align}
    (\rho^2 - D\rho^2 + a - \gamma)^2 + 4a\gamma
    &> (\rho^2 + D\rho^2 - a - \gamma)^2.
\end{align}
Expanding and simplifying, we obtain the required condition
\begin{align} \label{lambda2_requirement}
    a + D\gamma - D\rho^2 > 0
    \qquad\Longleftrightarrow\qquad
    a + D(\gamma - \rho^2) > 0.
\end{align}

We now show that this inequality holds for all $\rho>0$ in the two cases $D<1$ and $D\ge 1$.

Since \eqref{Golden_assumption} implies $a+\gamma-\rho^2>0$, if $D<1$, then immediately
\begin{align}\label{Di<1}
    a + D(\gamma - \rho^2) > 0.
\end{align}

If $D\ge 1$, then from \eqref{Golden_assumption} we have
\[
    D\rho^2 < a+\gamma-\rho^2.
\]
Since $D\rho^2$ is smaller than $a+\gamma-\rho^2$, substituting the larger value into \eqref{lambda2_requirement} yields
\begin{align}
    0 < a + D\gamma - (a + \gamma - \rho^2)
    < a + D\gamma - D\rho^2.
\end{align}
Simplifying, we obtain
\begin{align}\label{Di>1}
    a + D\gamma - a - \gamma + \rho^2
    = \gamma(D-1) + \rho^2 > 0,
\end{align}
which is always positive for $D\ge1$ .

Thus, by \eqref{D=1}, \eqref{Di<1} and \eqref{Di>1}, we conclude that
\begin{align}
    \lambda_2 > 0 
    \qquad\text{for all } \rho>0.
\end{align}
\end{proof}

We also prove an estimate for $\lambda_2$ in the following Lemma, which will be used in the proof of the main theorem.

\begin{lemma} \label{estimate_lambda_lemma}
  Assume \eqref{theta>agamma}.  For all $\rho>0$, we have that $\lambda_2 > D \rho^2-a$.
    \end{lemma}
    \begin{proof} Using $\theta>a \gamma$, we obtain
        \begin{align}
            \lambda_2&=  \frac{\rho^2+D \rho^2-a- \gamma +\sqrt{(\rho^2- D \rho^2+a- \gamma)^2+ 4 \theta}}{2}\\& >  \frac{\rho^2+D \rho^2-a- \gamma + \sqrt{(\rho^2- D \rho^2+a- \gamma)^2}}{2}.
        \end{align}
    We consider two cases. If we have $\rho^2 - D  \rho^2+a- \gamma \geq 0$ then  
    \begin{align}
        \lambda_2 & \geq  \frac{\rho^2+D \rho^2-a- \gamma + {(\rho^2- D \rho^2+a- \gamma)}}{2} = \rho^2- \gamma \geq D \rho^2 -a.
    \end{align} 
    Secondly, if $\rho^2 - D  \rho^2+a- \gamma \leq0$, 
    \begin{align}
        \lambda_2 & \geq  \frac{\rho^2+D \rho^2-a- \gamma - {(\rho^2- D \rho^2+a- \gamma)}}{2} =  D \rho^2 -a.
    \end{align}
    \end{proof}

\subsection{Conditions for the construction of upper and lower solutions} 
We will use the eigenvector $\nu_2$ from \eqref{nu_s} to define the upper and lower solutions. For this construction to work, we must ensure that all three components of $\nu_2$ are positive.

 \begin{lemma}\label{lem:cases} Assume \eqref{theta>agamma}. Then all components of $\nu_2$ are positive in the following cases:
\begin{enumerate}
    \item $a \le 1$.
    \item $a>1$ and the curves 
    \[
        S(\rho) \qquad \text{and} \qquad H_1(\rho) = D\rho - \frac{1}{\rho}
    \]
    do not intersect.
    \item $a>1$ and the curves $S(\rho)$ and $H_1(\rho)$ intersect, and the smallest intersection point $\rho_*$ satisfies $\rho_* > \rho_m$.
    \item $a>1$ and the curves $S(\rho)$ and $H_1(\rho)$ intersect, and the smallest intersection point $\rho_*$  satisfies $\rho_* < \rho_m$.
\end{enumerate}
In cases (1)–(3) we have $\rho \le \rho_m$, whereas in case (4) we restrict to $\rho \le \rho_*$.

\end{lemma}
\begin{proof}
For the first component of $\nu_2$ in \eqref{nu_s} we have
\[
\frac{\lambda_2 + a + 1 - D\rho^{2}}{(\lambda_2 + a - D\rho^{2})(\lambda_2 + a - D\rho^{2})}.
\]
Lemma~\ref{estimate_lambda_lemma} guarantees that
\[
\lambda_2 + a - D\rho^{2} > 0,
\]
which in turn implies 
\[
\lambda_2 + a + 1 - D\rho^{2} > 0.
\]
Therefore, to ensure positivity of the first component, we further require
\begin{equation}\label{new_conditions}
    \lambda_2 + 1 - D\rho^{2} > 0.
\end{equation}

The second component of $\nu_2$ in \eqref{nu_s} is always positive by Lemma~\ref{estimate_lambda_lemma}.

To determine when condition \eqref{new_conditions} holds, we distinguish two cases depending on the value of $a$.

\medskip
\noindent \textbf{Case 1: $a \le 1$.}
In this case, we have
\[
\lambda_2 +1 - D\rho^2 \geq \lambda_2 +a - D\rho^2 >0, \]
and the positivity requirement is automatically satisfied.  
Figure~\ref{fig:panel-a} illustrates an example of this case by plotting the speed curve 
$S(\rho)=\frac{\lambda_2(\rho)}{\rho}$ together with the curves 
$H_1(\rho)=D\rho - \frac{1}{\rho}$ and $H_2(\rho)=D\rho - \frac{a}{\rho}$.  
When $a \le 1$, the graph of $H_2$ lies strictly above that of $H_1$ for all 
$\rho>0$, and therefore neither curve intersects $S(\rho)$. This confirms that the inequality $\lambda_2 > D\rho^{2} - 1$ 
holds, and hence the components of $\nu_2$ remain positive throughout 
this case.

\medskip
\noindent\textbf{Case 2: $a > 1$.}
When $a>1$, there may exist an intersection point between $H_{1}(\rho)$ and
$S(\rho)$. If no such intersection occurs, then no additional restriction is
required. Otherwise, suppose that an intersection exists, and define the
smallest intersection point $(\rho_{*},c_{*})$ by
\begin{align}\label{rho_*_definition}
    \rho_{*} := \{\rho>0 \,;\, H_{1}(\rho)=S(\rho)\},
    \qquad
    c_{*} := S(\rho_{*}).
\end{align}
The inequality
\[
\lambda_{2}+1-D\rho^{2} > 0
\]
may still hold if $\rho_{m}<\rho_{*}$. However, if the smallest intersection
point satisfies $\rho_{*}<\rho_{m}$, then the inequality
\[
\lambda_{2}+1-D\rho^{2} > 0
\]
fails for $\rho\in(\rho_{*},\rho_{m})$. In this case, the first component of
$\nu_{2}$ becomes negative. Consequently, the construction of an upper solution
is no longer valid, and the parameter regime $\rho\in(\rho_{*},\rho_{m})$
remains undecided.

\end{proof}

We obtain two additional restrictions arising from the construction of upper
and lower solutions, which will be used later in~\eqref{upper_requirement}.
Specifically, we require
\begin{align}\label{H3_condition}
    c \;\ge\; D\rho \eqqcolon H_{3}(\rho),
\end{align}
and
\begin{align}\label{H4_condition}
    c \;\le\; \frac{-a}{\rho} + \frac{\theta}{\gamma\rho} + D\rho
    \eqqcolon H_{4}(\rho).
\end{align}

If $H_{3}(\rho)$ intersects $S(\rho)$ at a point $(\check{\rho},\check{c})$,
defined by
\begin{align}\label{check_rho_definition}
    \check{\rho} := \{\rho>0 \,;\, H_{3}(\rho)=S(\rho)\},
    \qquad
    \check{c} := S(\check{\rho}),
\end{align}
then the admissible values of $\rho$ satisfy $\rho \le \check{\rho}$.

Similarly, if $H_{4}(\rho)$ intersects $S(\rho)$ at a point
$(\tilde{\rho},\tilde{c})$, defined by
\begin{align}\label{tilde_rho_definition}
    \tilde{\rho} := \{\rho>0 \,;\, H_{4}(\rho)=S(\rho)\},
    \qquad
    \tilde{c} := S(\tilde{\rho}),
\end{align}
then the admissible values of $\rho$ satisfy $\rho \le \tilde{\rho}$.

Combining all the above restrictions
\eqref{rho_*_definition}, \eqref{check_rho_definition}, and
\eqref{tilde_rho_definition}, the values of $\rho$ for which we can establish
the existence of positive travelling wave solutions satisfy
\[
0 < \rho \le \bar{\rho},
\]
where
\begin{align}\label{bold_rho}
    \bar{\rho}
    := \min\{\rho_m,\, \rho_{*},\, \check{\rho},\, \tilde{\rho}\},
    \qquad
    \bar{c} := S(\bar{\rho}).
\end{align}
Here $\rho_m$ is defined in~\eqref{rho_m_definiton}, $\rho_*$ in \eqref{rho_*_definition}, $\check \rho$ in \eqref{check_rho_definition}, and $\tilde \rho$ in \eqref{tilde_rho_definition}.

\medskip

In Figure~\ref{fig:two-panels}, we present four different cases illustrating the relative alignment of the critical points
$\rho_m$, $\rho_{*}$, $\tilde{\rho}$, and $\check{\rho}$.
Figures~\ref{fig:panel-a} and~\ref{fig:panel-c} display the curves
$S(\rho)$, $H_{1}(\rho)$, $H_{2}(\rho)$, $H_{3}(\rho)$, and $H_{4}(\rho)$
in the parameter regime $a>1$.
In Figure~\ref{fig:panel-a}, since $a>1$, the curve $H_{1}(\rho)$ lies above
$H_{2}(\rho)$. However, there is no intersection between $H_{1}(\rho)$ and
$S(\rho)$ for $\rho<\rho_m$. Similarly, $H_{3}(\rho)$ and $S(\rho)$ do not
intersect for $\rho<\rho_m$. The curve $H_{4}(\rho)$ intersects $S(\rho)$ for
$\rho>\rho_m$, and the point $(\rho_m,c_m)$ lies within the shaded region.
Consequently, in this case we have $\bar{\rho}=\rho_m$.
The value of $\gamma$ used in Figures~\ref{fig:panel-a} and~\ref{fig:panel-c}
is taken from~\cite{arwa} and equals $\gamma=40/3$. By contrast, Friedman
et al.~\cite{FriedmanAvner2006GVEo} reported a viral clearance rate of $\gamma=1.25$.
This smaller value of $\gamma$ yields a less restrictive shaded region, as
illustrated in Figures~\ref{fig:panel-b} and~\ref{fig:panel-d}.
Figure~\ref{fig:panel-b} depicts a possible, though biologically unrealistic,
parameter regime in which the curve $H_{1}(\rho)$ intersects $S(\rho)$, as does
$H_{3}(\rho)$. In this situation, the smallest of these intersection points
lying within the shaded region determines $\bar{\rho}$. 

Figures~\ref{fig:panel-b}
and~\ref{fig:panel-d} correspond to the case $a<1$. In both panels, the curve
$H_{2}(\rho)$ lies strictly above $H_{1}(\rho)$, and consequently there are no
intersections between $S(\rho)$ and $H_{1}(\rho)$.
The effect of varying $\gamma$ is particularly evident in these latter two
cases. As shown in Figure~\ref{fig:panel-b}, for larger values  of $\gamma$ the
region between $H_{4}(\rho)$ and $S(\rho)$ becomes more restrictive, and
$\tilde{\rho}$ occurs before $\rho_m$, so that $\bar{\rho}=\tilde{\rho}$.
In contrast, for smaller values of $\gamma$, $\tilde{\rho}$ occurs beyond
$\rho_m$, and hence $\bar{\rho}=\rho_m$.

\begin{figure}[ht]
    \centering
    
    \begin{subfigure}[t]{0.49\textwidth}
        \centering
        \includegraphics[width=\textwidth]{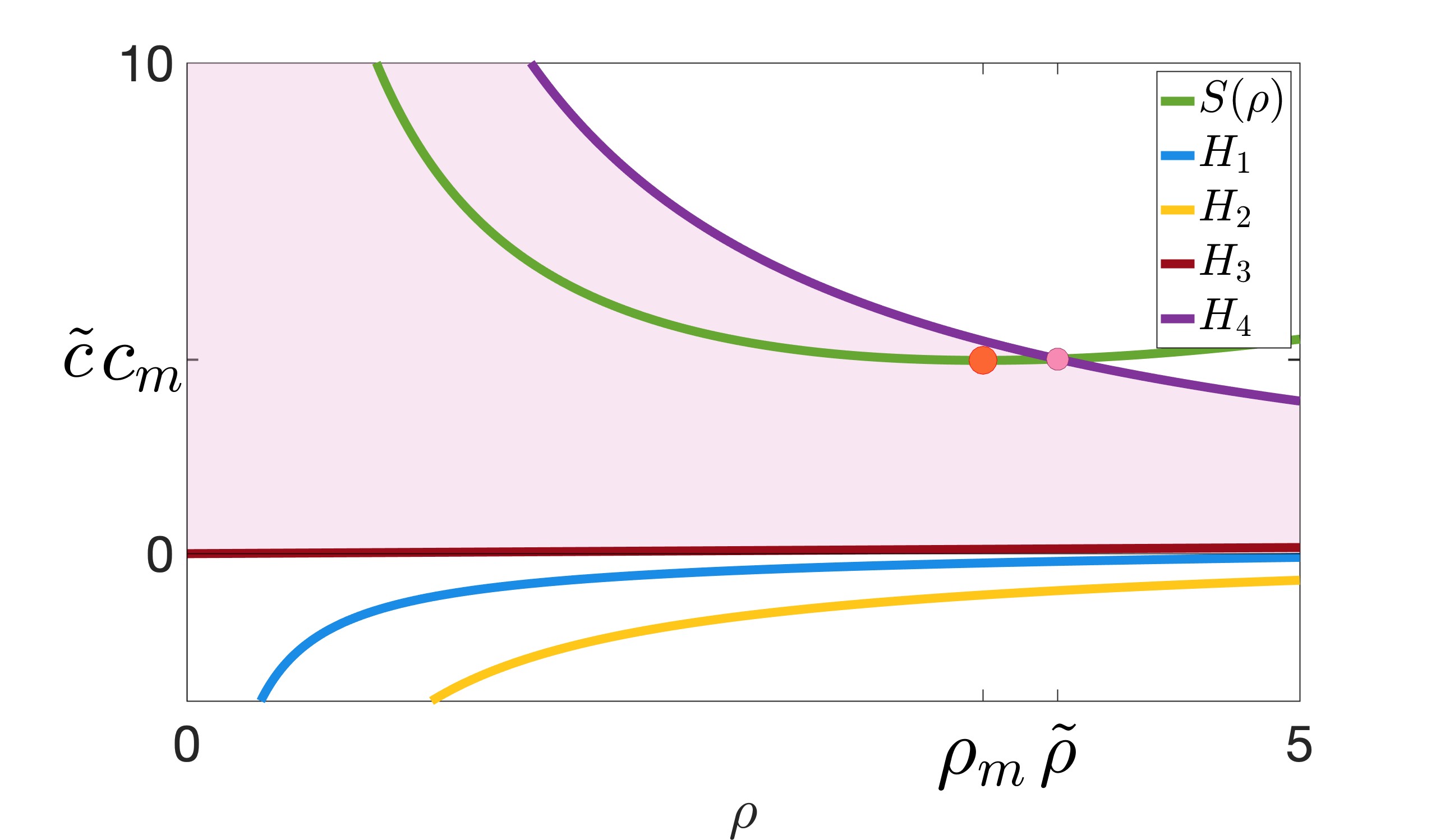}
        \caption{}
        \label{fig:panel-a}
    \end{subfigure}
    \hfill
    \begin{subfigure}[t]{0.49\textwidth}
        \centering
        \includegraphics[width=\textwidth]{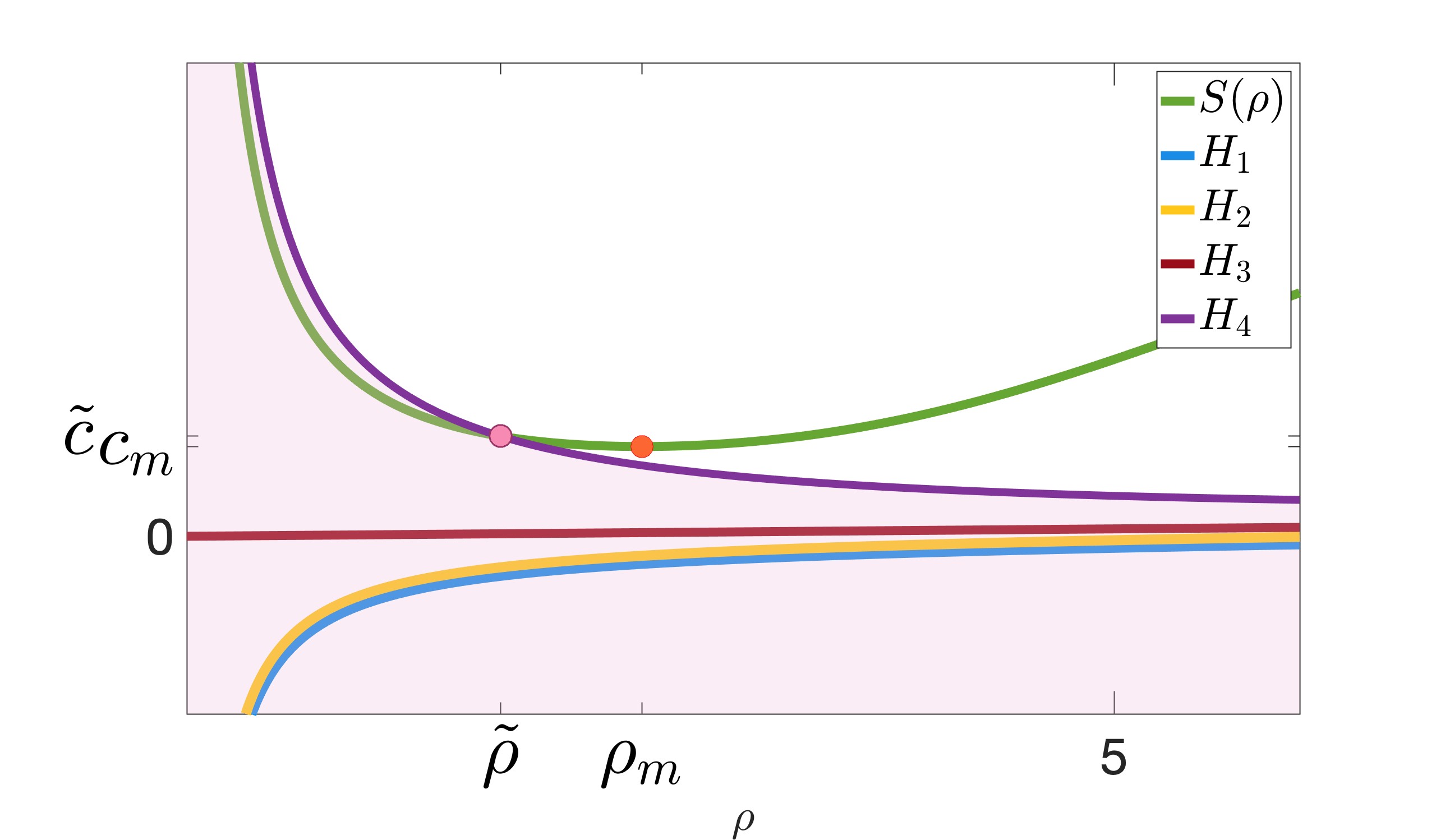}
        \caption{}
        \label{fig:panel-b}
    \end{subfigure}
    \begin{subfigure}[t]{0.49\textwidth}
        \centering
        \includegraphics[width=\textwidth]{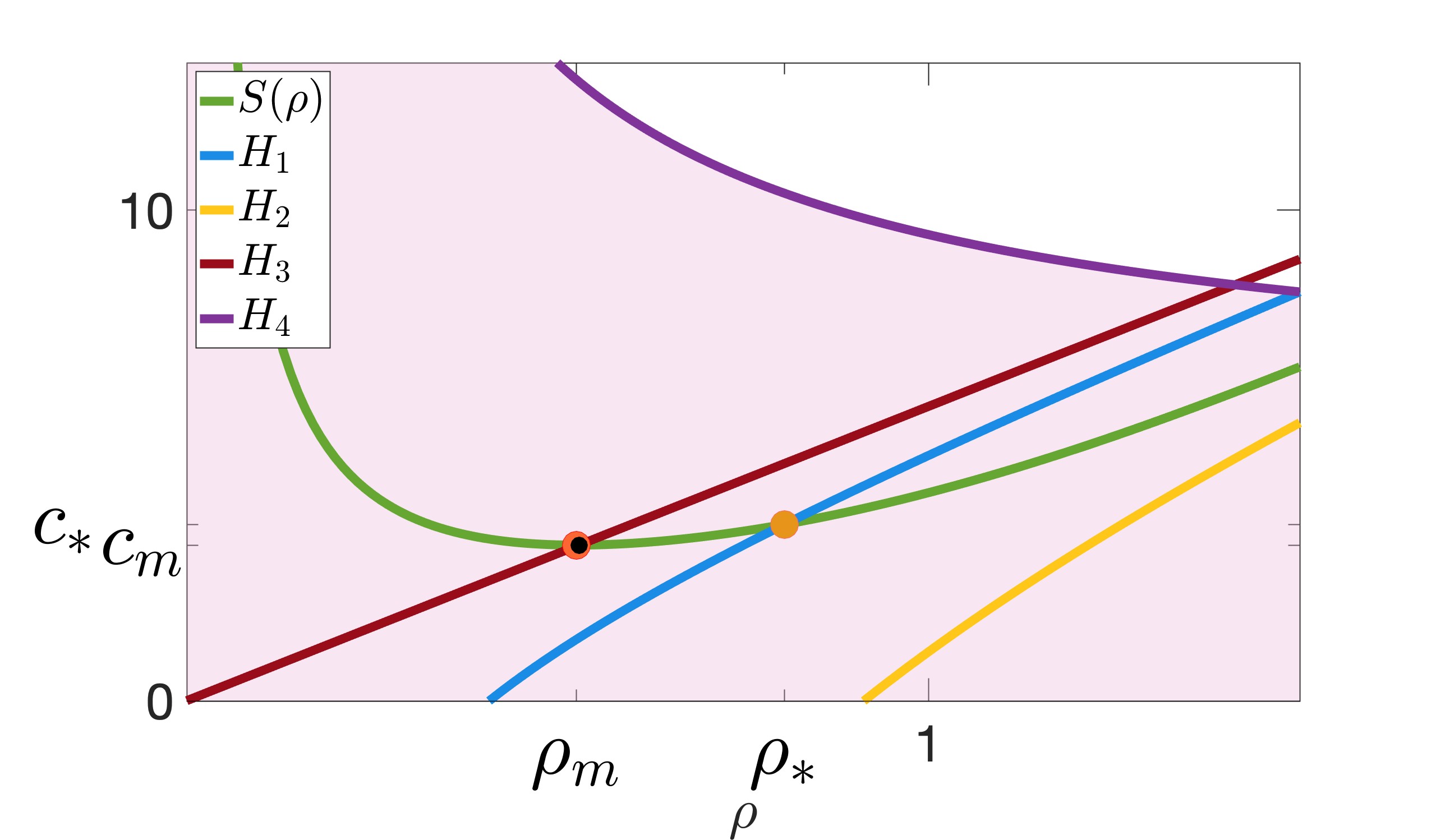}
        \caption{}
        \label{fig:panel-c}
    \end{subfigure}
    \begin{subfigure}[t]{0.49\textwidth}
        \centering
        \includegraphics[width=\textwidth]{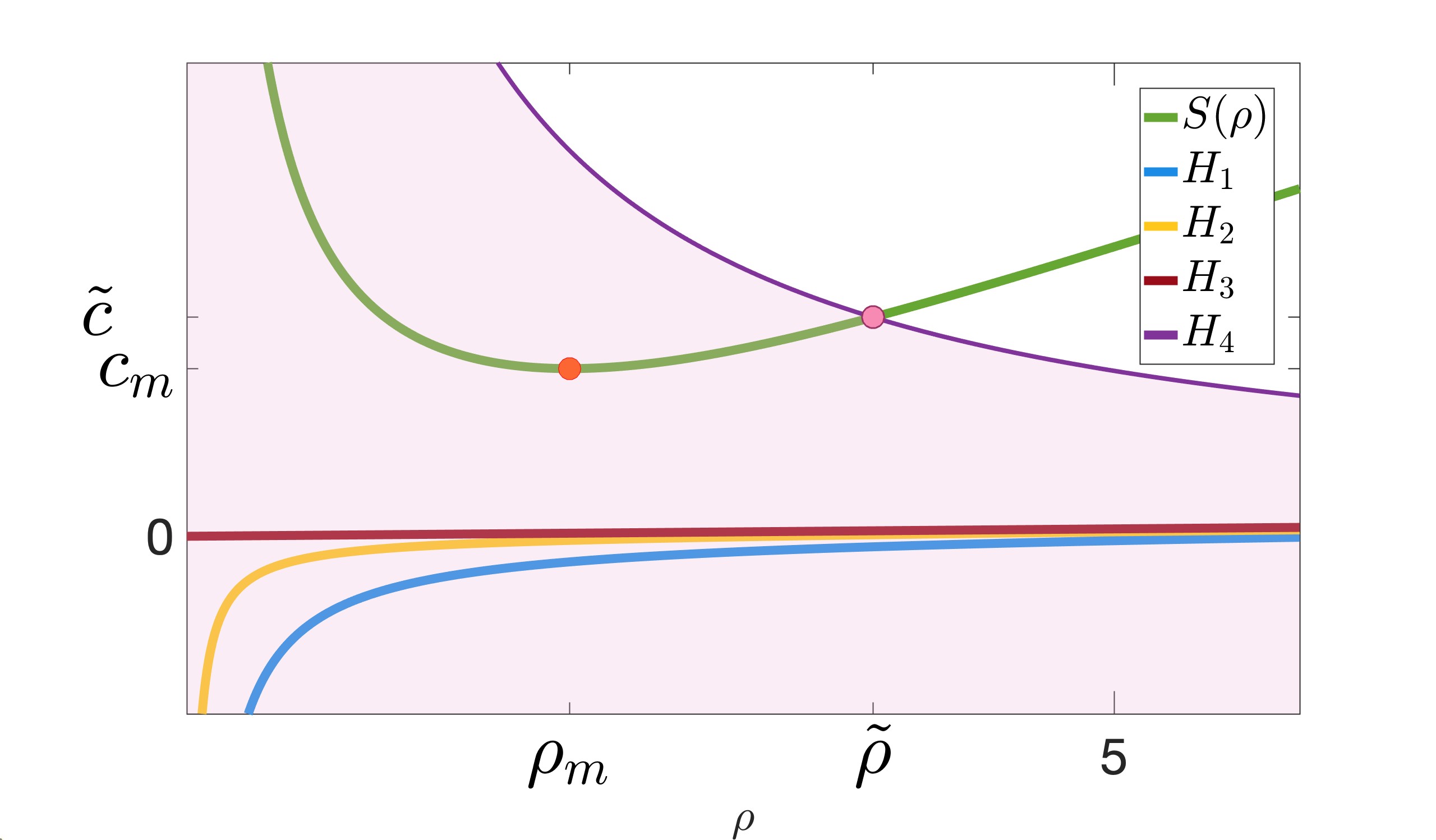}
        \caption{}
        \label{fig:panel-d}
    \end{subfigure}
  \caption{Numerical illustration of $S(\rho)$, and its minimum value indicated by the orange ball, and  $H_1(\rho),\ H_2(\rho), \ H_3(\rho)$, and $H_4(\rho)$ for different parameter values. The parameter values have been chosed to illustrate each of the regimes and are given by:  (a) $a= \frac{10}{3},\ \gamma= \frac{40}{3},\ \theta=250,\ D=0.025$ (b) $a=0.96,\ \gamma= \frac{40}{3}, \theta=50, D=0.025$, (c) $a=5 ,\ \gamma=2,\ \theta=17,\ D=6$ (d) $a= 0.96,\ \gamma= 1.25,\ \theta=17,\ D=0.025$. Note that case (c) may not arise in biologically relevant cases.} 
    \label{fig:two-panels}
\end{figure}

Before stating the main theorem, we analyze the asymptotic behaviour of the curves 
\(S(\rho)\) and \(H_4(\rho)\) as \(\rho\to 0^+\). As illustrated in Figure~\ref{fig:two-panels},  
whenever these two curves intersect at a point \((\tilde\rho,\tilde c)\), $\tilde \rho<{\rho}_m, $ no further 
intersections occur for \(0<\rho<\tilde\rho\).  
To justify this observation analytically, we examine the leading-order behaviour of  
both curves for small \(\rho\).
\begin{lemma}\label{lem:asymptotic}
\label{lem:H4>S_near_zero}
Under assumption \eqref{theta>agamma}, for $\rho$ sufficiently small we have
\[
H_4(\rho) > S(\rho).
\]
\end{lemma}
\begin{proof}
For \(\rho\) sufficiently small, we may neglect all terms of order \(\mathcal{O}(\rho)\) or higher.  
Under this approximation, \(H_4(\rho)\) reduces to  
\[
    H_4(\rho)\;\sim\;\frac{1}{\rho}\left(\frac{\theta}{\gamma}-a\right),
\]
while \(S(\rho)\) becomes
\[
    S(\rho)\;\sim\;
    \frac{1}{\rho}\left(
        -\frac{a}{2}-\frac{\gamma}{2}
        +\frac{1}{2}\sqrt{(a-\gamma)^2+4\theta}
    \right).
\]
Since both expressions scale like \(\mathcal{O}(1/\rho)\), the ordering of the two curves  
for small \(\rho\) is determined by comparing their coefficients.  
To ensure that \(H_4(\rho)\) lies strictly above \(S(\rho)\) as \(\rho\to 0^+\), we require
\begin{align}\label{ineq:H4>S}
    \frac{\theta}{\gamma}-a
    \;>\;
    -\frac{a}{2}-\frac{\gamma}{2}
    +\frac{1}{2}\sqrt{(a-\gamma)^2+4\theta}.
\end{align}

By simplifying~\eqref{ineq:H4>S} and squaring both sides, we obtain the equivalent condition
\[
    \left(\frac{\theta}{\gamma}-\frac{a}{2}+\frac{\gamma}{2}\right)^{2}
    >
    \frac{1}{4}\Big((a-\gamma)^{2}+4\theta\Big),
\]
which expands to
\[
    \frac{\theta^{2}}{\gamma^{2}}
    +\frac{a^{2}}{4}
    +\frac{\gamma^{2}}{4}
    -\frac{\theta a}{\gamma}
    +\theta
    -\frac{a\gamma}{2}
    >
    \frac{a^{2}}{4}
    +\frac{\gamma^{2}}{4}
    +\theta
    -\frac{a\gamma}{2}.
\]
Cancelling common terms yields the condition
\begin{align}\label{ineq:final_theta_gamma}
    \theta\left(\frac{\theta}{\gamma^{2}}-\frac{a}{\gamma}\right)>0.
\end{align}

Since our standing assumption is \(\theta>a\gamma\), the right-hand side of 
\eqref{ineq:final_theta_gamma} satisfies
\[
    \dfrac{\theta}{\gamma}\left(\frac{\theta}{\gamma}-{a}\right)
    >
     0.
\]
Thus condition~\eqref{ineq:final_theta_gamma} holds automatically.  
Consequently, we conclude that
\[
    H_4(\rho)\;>\;S(\rho)\qquad \text{for all sufficiently small } \rho>0,
\]
which confirms that the two curves cannot intersect again as \(\rho\to 0^+\).

\end{proof}
Note that $H_4(\rho)$ is convex for $\rho>0$, since
\[
H_4''(\rho)= \frac{2\left(\frac{\theta}{\gamma}-a\right)}{\rho^{3}} > 0 \qquad (\rho>0).
\]
Because of the algebraic complexity of the function $S(\rho)$, establishing its convexity on the interval $0<\rho<\rho_m$ is not straightforward. Nevertheless, by Lemma~\ref{lem:asymptotic}, we know that $H_4(\rho) > S(\rho)$ as $\rho \to 0^+$. Although our numerical simulations did not reveal any secondary intersections between the graphs of $H_4(\rho)$ and $S(\rho)$, we adopt a conservative formulation and define
\begin{equation}\label{tilderho}
\tilde{\rho}
    = \underset{0<\rho<\rho_m}{\min}\,\{\rho : H_4(\rho)=S(\rho)\},
\end{equation}
that is, $\tilde{\rho}$ is the smallest positive value of $\rho$ at which the two curves intersect.

\begin{thm}\label{mainTHeorem} Assume \eqref{theta>agamma} and \eqref{bold_rho}. Then for each $c \ge \bar{c}$ the system \eqref{eqn:immune} has a self-similar solution of the form $(C(\xi), I(\xi), V(\xi))  $ with $\xi=x+ct$. The solution connects to $(1,0,0)$ for $\xi \to - \infty$ and stays bounded for $\xi\to\infty$. 
 \end{thm} 

 \noindent{\bf Remarks:} 
 \begin{enumerate}
 \item Theorem \ref{mainTHeorem} is formulated in $(C,I,V)$, but we will use the formulation in $(B,I,V)$ for the proof. 
     \item 
 Theorem \eqref{mainTHeorem} is our main theorem, and we will use  Schauder's fixed point Theorem on an appropriate set $\Gamma$ where we will show this set is non-empty, closed, bounded, convex subset of some Banach space $\chi$. Once the upper ($\bar{B}, \bar{I}, \bar{V}$) and the lower solutions ($\underline{B},\underline{I},\underline{V}$) are defined, the set $\Gamma$ will be  
\begin{align}
    \Gamma \coloneqq \{(B,I,V) \in C(\mathbb{R}, \mathbb{R}^3): (\underline{B}(\xi), \underline{I}(\xi), \underline V(\xi))\leq (B(\xi),I(\xi),V(\xi))\le (\bar{B}(\xi),\bar {I}(\xi),\bar {V}(\xi)), \forall \xi \in \mathbb{R}\}.
\end{align}

\end{enumerate}

\subsection{Construction of upper and lower solution}For the construction of upper and lower solutions, note that for any
$c>\bar{c}$ there exists a $\rho\in(0,\bar{\rho})$ such that $S(\rho)=c$, where
$\bar{\rho}$ is defined in~\eqref{bold_rho}. Moreover, since $S(\rho)$ is
increasing on $(0,\bar{\rho}]$, we have
$S(\bar{\rho})=\bar{c}<c$. For $\epsilon>0$, define
\[
c_{\epsilon}:=S(\rho+\epsilon).
\]
If $\epsilon$ is chosen sufficiently small, then
\[
\bar{c} < c_{\epsilon} < c.
\]

We now assume that $\lambda_{2}(\rho)$ denotes the eigenvalue given in
\eqref{lambdas}, which is associated with the wave speed $c$ and the decay rate
$\rho$.

The wave–speed curve $S(\rho)$ is shown in Figure~\ref{fig:speed_curve} (left
panel). The pair $(\bar{\rho},\bar{c})$ is indicated on the curve by an orange
marker. We also introduce the points
\begin{align}\label{rho_rho_epsilon}
    \rho := \bar{\rho} - \tau,
    \qquad
    \rho_{\epsilon} := \rho + \epsilon .
\end{align}
Here $\tau>0$ is arbitrary, chosen so that $\bar{\rho}-\tau>0$, and
$\epsilon>0$ is sufficiently small to ensure that
\[
\rho < \rho_{\epsilon} < \bar{\rho}.
\]

\begin{figure}
    \centering \includegraphics[width=7.3cm]{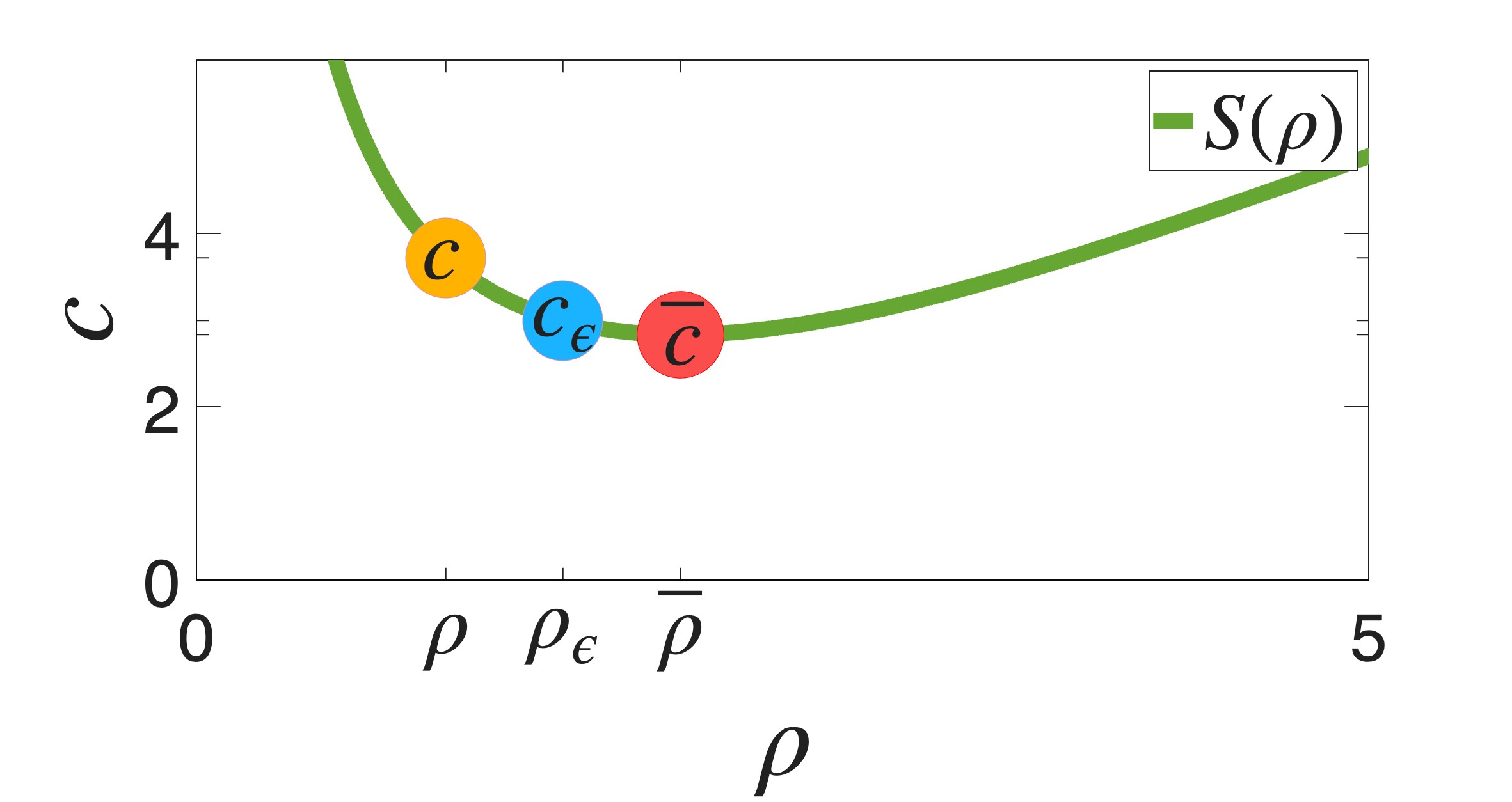}\hspace*{0.7cm}
    \includegraphics[width=6.8cm]{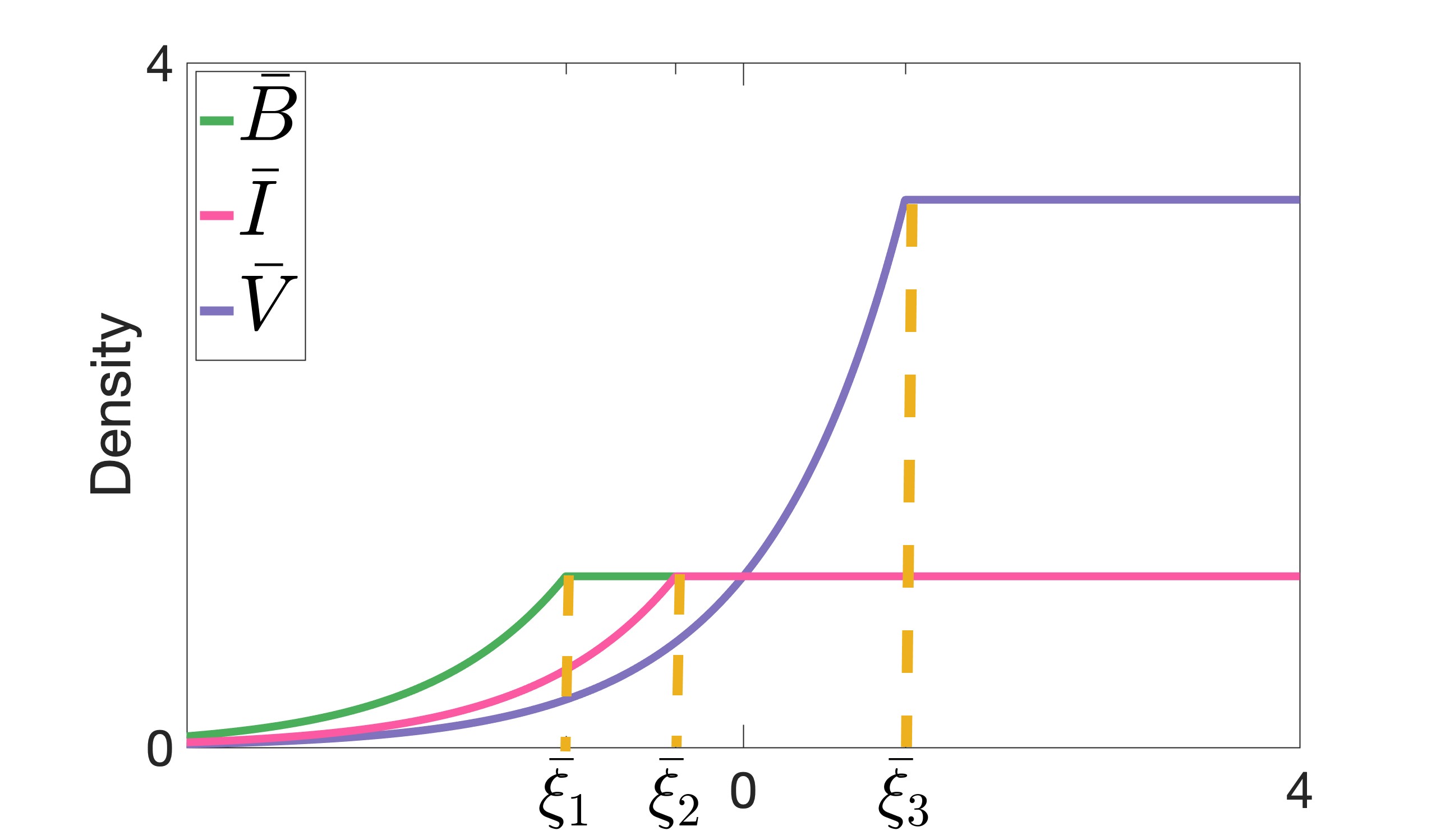}
    \caption{Left: Wave–speed function $S(\rho)$ together with the values $\bar{c}$, $c$,
and $c_{\epsilon}$. The point $(\bar{\rho},\bar{c})$ is marked in red. The points
$\rho=\bar{\rho}-\tau$ and $\rho_{\epsilon}=\rho+\epsilon$, with $\tau=1$ and
$\epsilon=0.5$, are also shown, along with their corresponding wave speeds
$(c,c_{\epsilon})$.  
Right: Sketch of the upper solutions and their transition points
$\bar{\xi}_{1}$, $\bar{\xi}_{2}$, and $\bar{\xi}_{3}$.
  }  
    \label{fig:speed_curve}
\end{figure}

We  now define our upper and lower solution functions.
The upper solutions $(\bar{B}, \bar{I},\bar{V})$ are given by
\begin{align} \label{upper_definitions} \nonumber
    \bar{B}(\xi)&= \min \left\{e^{\rho \xi} \dfrac{(c \rho+a+1-D \rho^2)}{(c \rho +a- D \rho^2)(c\rho +1- D\rho^2)},1\right\}, \\ \bar{I}(\xi)&= \min \left\{e^{\rho \xi}\dfrac{1}{c \rho +a - D \rho^2},1 \right\}, \\ \nonumber \bar{V}(\xi)&= \min\left\{e^{\rho \xi}, \frac{\theta}{\gamma}\right\},
\end{align}
and the lower solutions $(\underline{B},\underline{I},\underline{V})$ are given by
\begin{align} \label{lower_definitions} \nonumber
     \underline{B}(\xi)&= 0, \\ \underline{I}(\xi)&= \max \left\{\kappa e^{\rho \xi} \dfrac{1}{c \rho +a - D \rho^2}-e^{(\rho+ \epsilon)\xi}\dfrac{1}{c_\epsilon (\rho+\epsilon) +a - D (\rho+\epsilon)^2},0 \right\}, \\ \nonumber \underline{V}(\xi)&= \max \left\{\kappa e^{\rho \xi}-e^{(\rho+ \epsilon)\xi},0 \right\},
 \end{align}
where $0<\kappa<1$ will be specified later.

We denote by $\bar{\xi}_{1}$, $\bar{\xi}_{2}$, and $\bar{\xi}_{3}$ the points at which the minimum functions in the definition of the upper solutions~\eqref{upper_definitions} switch from their exponential form to their corresponding constant values.  And 
similarly, we use $\underline{\xi}_{2}$ and $\underline{\xi}_{3}$ to denote the points at which the maximum functions in the definition of the lower solutions~\eqref{lower_definitions} switch from their exponential form to their corresponding constant values. We show the upper solutions in Figure \ref{fig:speed_curve} on the right. 

 The transition points are given by
 \begin{align}
    \bar{\xi}_1&=\frac{1}{\rho} \ln {\frac{(D \rho^2 -a -c \rho)(D\rho^2 -a- c \rho)}{D \rho^2-a-1-c \rho} },\\\bar{\xi}_2&=\frac{1}{\rho} \ln{(c \rho +a - D \rho^2)}, \\ \bar{\xi}_3&=\frac{1}{\rho}\ln{\frac{\theta}{\gamma}}, \\ \underline{\xi_2}&= \frac{1} {\epsilon} \ln{\kappa\frac{c_\epsilon (\rho+\epsilon) +a - D (\rho+\epsilon)^2}{c \rho +a - D \rho^2}}, \\  \underline{\xi_3}&= \frac{1} {\epsilon} \ln{\kappa}.
\end{align}

Thus, we can rewrite the constructed upper solutions as \begin{align}\label{uppersolutions}
     \nonumber \bar{B}(\xi)& =
    \begin{cases}
     e^{\rho \xi} \dfrac{(c \rho+a+1-D \rho^2)}{(c \rho +a- D \rho^2)(c\rho +1- D\rho^2)},& \quad \xi \leq \bar{\xi_1}  \\ 
        1, & \quad \xi > \bar{\xi}_1
    \end{cases} \\\bar{I}(\xi) &=
    \begin{cases}
       e^{\rho \xi}\dfrac{1}{c \rho +a - D \rho^2}, & \quad \xi \leq \bar{\xi_2} \\ 
        1, & \quad \xi > \bar{\xi_2}
    \end{cases}\\ \nonumber\bar{V}(\xi)& =
    \begin{cases}
        e^{\rho \xi}, & \quad \xi \leq\bar{\xi_3} \\ 
         \dfrac{\theta}{\gamma},\, & \quad \xi >\bar{\xi_3},
    \end{cases}
\end{align}

and the lower solutions as
\begin{align}\label{lowersolutions}
  \nonumber  \underline{B}(\xi)& =
   0, \\ \underline{I}(\xi) &=
    \begin{cases}
       \kappa e^{\rho \xi} \dfrac{1}{c \rho +a - D \rho^2}-e^{(\rho+ \epsilon)\xi}\dfrac{1}{c_\epsilon (\rho+\epsilon)+a-D (\rho+\epsilon)^2 },  & \quad \xi \geq \underline{\xi_2} \\ 
        0, & \quad \xi < \underline{\xi_2}
    \end{cases}\\ \notag \underline{V}(\xi)& =
    \begin{cases}
        \kappa e^{\rho \xi}-e^{(\rho+ \epsilon)\xi}, & \quad \xi \leq \underline{\xi_3} \\ 
       0, & \quad \xi >\underline{\xi_3}. 
    \end{cases}
\end{align}

In our proofs, we make use of the following ordering 
\begin{align}
    \bar{\xi}_{1} \;\le\; \bar{\xi}_{2} \;\le\; \bar{\xi}_{3},
\end{align}
which is true due to our choice of $\bar\rho$ in (\ref{bold_rho}).

Indeed,  we rewrite the upper solutions as
\begin{align}\nonumber
    \bar{B}(\xi)=\min\left\{A_1 e^{\rho \xi},\, 1\right\}, 
    \qquad 
    \bar{I}(\xi)=\min\left\{A_2 e^{\rho \xi},\, 1\right\},
    \qquad 
    \bar{V}(\xi)=\min \left\{e^{\rho \xi},\, \dfrac{\theta}{\gamma}\right\}.
\end{align}

We first observe that since $a$ has a positive value, we have
\begin{align}
   A_1
   =\frac{c\rho + a + 1 - D\rho^{2}}{(c\rho + a - D\rho^{2})(c\rho + 1 - D\rho^{2})}
   \;\ge\;
   \frac{c\rho + 1 - D\rho^{2}}{(c\rho + a - D\rho^{2})(c\rho + 1 - D\rho^{2})}
   =\frac{1}{c\rho + a - D\rho^{2}}
   =A_2.
\end{align}
Since $A_1 \ge A_2$, it follows that
\begin{align}
    \bar{\xi}_1
    =\frac{1}{\rho}\ln\!\left(\frac{1}{A_1}\right)
    \;\le\;
    \frac{1}{\rho}\ln\!\left(\frac{1}{A_2}\right)
    =\bar{\xi}_2.
\end{align}

Next we determine the condition under which $\bar{\xi}_2 \le \bar{\xi}_3$.  
This requires
\begin{align}
    \bar{\xi}_2
    =\frac{1}{\rho}\ln\!\left(\frac{1}{A_2}\right)
    \;\le\;
    \frac{1}{\rho}\ln\!\left(\frac{\theta}{\gamma}\right)
    =\bar{\xi}_3.
\end{align}
Equivalently,
\begin{align}
    \frac{1}{A_2} \;\le\; \frac{\theta}{\gamma},
\end{align}
which reduces to
\begin{align}
    c\rho + a - D\rho^2 < \frac{\theta}{\gamma}.
\end{align}

This is precisely the assumption \eqref{H4_condition}.

\begin{lemma}
    For $c> \bar{c}$ and $\bar{B}, \bar{I}$ and $\bar{V}$ as given above \eqref{uppersolutions}, the following inequalities are satisfied. 
    \begin{align}\label{upper_condition_lemma}
      \nonumber  c \bar{B}'- D \bar{B}'' 
 &\geq- \bar{B}+\bar I +\bar{B}^2-\bar{B} \bar{I}+ \bar{V}-\bar{B} \bar{V}, \\  c \bar{I}'- D \bar{I}'' 
 &\geq -a\bar{I}+ (1-\bar{B})\bar{V}, \\ \nonumber c \bar{V}'-  \bar{V}'' 
 &\geq \theta \bar{I} - \gamma \bar{V}.
    \end{align}
\end{lemma}

\begin{proof} 
\textbf{Step 1.} We first show that under the condition $c > D\rho$, the upper solution 
$\bar{U} = (\bar{B}, \bar{I}, \bar{V})$ satisfies the linearization inequality
\begin{align}\label{linearization_inequality}
    cU' - \mathcal{D}U'' - f(U) \;\ge\; cU' - \mathcal{D}U'' - Uf'(0).
\end{align}

For the components $\bar{I}$ and $\bar{V}$ the inequality follows directly.  
For $\bar{I}$ we compute
\begin{align}
    c\bar{I}' - D\bar{I}'' - f(\bar{I})
    &= c\bar{I}' - D\bar{I}'' + a\bar{I} + (\bar{B}-1)\bar{V} \\
    &\ge c\bar{I}' - D\bar{I}'' + a\bar{I} - \bar{V}
      = c\bar{I}' - D\bar{I}'' - \bar{I}f'(0),
\end{align}
since $\bar{B} \ge 0$.  
For $\bar{V}$ we similarly obtain
\begin{align}
    c\bar{V}' - \bar{V}'' - f(\bar{V})
    = c\bar{V}' - \bar{V}'' + \gamma \bar{V} - \theta \bar{I}
    = c\bar{V}' - \bar{V}'' - \bar{V}f'(0).
\end{align}

For the component $\bar{B}$, we have
\begin{align}
    c\bar{B}' - D\bar{B}'' - f(\bar{B})
    &= c\bar{B}' - D\bar{B}'' 
    + \bar{B} - \bar{I} - \bar{B}^2 + \bar{B}\bar{I} - \bar{V} + \bar{B}\bar{V} \\
    &= c\bar{B}' - D\bar{B}'' 
       + \bar{B}\bigl(-\bar{B} + \bar{I} + \bar{V}\bigr)
       + (\bar{B}-\bar{I} - \bar{V}) \\
    &\ge c\bar{B}' - D\bar{B}'' + \bar{B} - \bar{I} - \bar{V}
      = c\bar{B}' - D\bar{B}'' - \bar{B}f'(0),
\end{align}
which holds provided that
\begin{align} \label{I+V>B}
    \bar{B} \le \bar{I} + \bar{V}.
\end{align}

For $\xi < \bar{\xi}_1$, substituting the expressions from~\eqref{uppersolutions} into~\eqref{I+V>B} gives
\begin{align}
    \bar{I} + \bar{V}
    = \frac{c\rho^{2}+a+1 - D\rho^2}{c\rho+a - D\rho^2}
    \;\ge\;
    \frac{(c\rho+a+1 - D\rho^2)}
         {(c\rho+a - D\rho^2)(c\rho+1 - D\rho^2)}=\bar{B}.
\end{align}

Since we have
\[
c\rho+a - D\rho > 0,
\qquad
c\rho+1 - D\rho^2 > 0,
\]
the inequality reduces to the requirement
\begin{align}\label{upper_requirement}
    0 \le c\rho - D\rho^{2},
\end{align}
which is equivalent to the condition $c \ge D\rho$ for $\rho>0$.

On the interval $\bar{\xi}_1 < \xi < \bar{\xi}_2$, only the condition at 
$\xi = \bar{\xi}_1$ needs to be verified, because $\bar{B}$ is identically 
equal to $1$ on this interval, while both $\bar{I}$ and $\bar{V}$ are 
increasing functions.  
At $\xi = \bar{\xi}_1$ we have

\begin{align}
   \bar{B}= e^{\rho \bar{\xi}_1}
    \left(
     \frac{(c\rho+a+1 - D\rho^2)}
         {(c\rho+a - D\rho^2)(c\rho+1 - D\rho^2)}
    \right)
    = 1,
\end{align}
which simplifies to
\begin{align}
    e^{\rho \bar{\xi}_1}
    \left(
     \frac{(c\rho+a+1 - D\rho^2)}
         {(c\rho+a - D\rho^2)}
    \right)
    = c\rho - D\rho^{2} + 1 \;\ge\; 1,
\end{align}
since $c > D\rho$.  
Thus,
\begin{align}
    \bar{I} + \bar{V}
    = e^{\rho \bar{\xi}_1}
    \left(
        \frac{(c\rho+a+1 - D\rho^2)}
         {(c\rho+a - D\rho^2)}
    \right)
    \ge 1.
\end{align}

For all $\xi > \bar{\xi}_2$, the inequality $\bar{I} + \bar{V} \ge \bar{B}$ holds trivially because  
$\bar{I} = 1$ and $\bar{V} \ge 0$.

Since all components $\bar{B},\ \bar{I},$ and $  \bar{V}$ satisfy the linearization inequality $cU' - \mathcal{D}U'' - f(U) \;\ge\; cU' - \mathcal{D}U'' - Uf'(0)$, instead of proving the inequalities \eqref{upper_condition_lemma}, we can equivalently prove
\begin{align}
    c \bar{B}' - D\bar{B}'' &\ge -\bar{B} + \bar{I} + \bar{V}, \\
    c \bar{I}' - D\bar{I}'' &\ge -a\bar{I} + \bar{V}, \\
    c \bar{V}' - \bar{V}'' &\ge \theta \bar{I} - \gamma \bar{V}.
\end{align}
Thus, in more complicated cases, it is sufficient to verify the linearized inequality
\[
cU' - \mathcal{D}U'' - Uf'(0) \;\ge\; 0,
\]
instead of the full nonlinear inequality
\[
cU' - \mathcal{D}U'' - f(U) \;\ge\; 0.
\]
In the following steps, we prove the inequalities \eqref{upper_condition_lemma} for all possible positions of $\xi$ within the ordering \[\bar{\xi}_1\le \bar{\xi}_2 \le \bar{\xi}_3.\]
\textbf{Step 2.} We first consider the region $\xi < \bar{\xi}_1$, where the upper solutions are given by 
\[
\bar{B}(\xi)= e^{\rho \xi}\,
\frac{(c\rho+a+1 - D\rho^2)}
         {(c\rho+a - D\rho^2)(c\rho+1 - D\rho^2)}, 
\qquad
\bar{I}(\xi)= e^{\rho \xi}\,\frac{1}{c\rho+a-D\rho^2}, 
\qquad
\bar{V}(\xi)= e^{\rho \xi}.
\]

Substituting these into the linearized equation
\[
cU' - D U'' - U f'(0),
\]
we compute each component. For $\bar{B}$, we obtain
\begin{align}
c\bar{B}' - D\bar{B}'' + \bar{B} - \bar{I} - \bar{V}
= e^{\rho \xi}\left[
\frac{(c\rho +1- D\rho^{2})(c\rho+a+1-D\rho^2)}{(c\rho+a-D\rho^2)(c\rho+1-D\rho^2)}
-\frac{1}{c\rho+a-D\rho^2} - 1
\right]
=0.
\end{align}

For $\bar{I}$, we have
\begin{align}
c\bar{I}' - D\bar{I}'' + a\bar{I} - \bar{V}
= e^{\rho \xi}\left[
\frac{D\rho^{2}-a-c\rho}{D\rho^{2}-a-c\rho} - 1
\right]
=0.
\end{align}

For $\bar{V}$, substitution yields
\begin{align}
c\bar{V}' - \bar{V}'' - \theta \bar{I} + \gamma\bar{V}
= e^{\rho \xi}\left[
(c\rho - \rho^{2} + \gamma)(D\rho^{2}-a-c\rho) + \theta
\right]
=0.
\end{align}

The last equality follows from the relation  
\begin{align}\label{det_identitty_lemma_upper_proof}
0 = \det(A(\rho) - c\rho\,\mathbb{I})
     = (D\rho^{2}-1)\left[(c\rho - \rho^{2} - \gamma)(D\rho^{2}-a-c\rho) + \theta\right],
\end{align}
since $c\rho$ is the eigenvalue of $A(\rho)$.

\textbf{Step 3.} Next, consider $\bar{\xi}_1 \le \xi < \bar{\xi}_2 < \bar{\xi}_3$, where  
\[
\bar{B}=1, 
\qquad 
\bar{I}= e^{\rho \xi}\frac{1}{c\rho+a-D\rho^2},
\qquad 
\bar{V}= e^{\rho \xi}.
\]

For $\bar{B}$, we have
\begin{align}
c\bar{B}'-D\bar{B}''+\bar{B}-\bar{B}^{2}+\bar{B}\bar{I}+\bar{B}\bar{V}-\bar{I}-\bar{V}
= 1 - 1 + \bar{I} + \bar{V} - \bar{I} - \bar{V}
=0.
\end{align}

For $\bar{I}$, we obtain
\begin{align}
c\bar{I}' - D\bar{I}'' + a\bar{I} - \bar{V}
= e^{\rho \xi}\left[
(c\rho+a - D\rho^{2} )\frac{1}{c\rho+a-D\rho^2} - 1
\right]
=0.
\end{align}

For $\bar{V}$, we compute
\begin{align}
c\bar{V}' - \bar{V}'' - \theta \bar{I} + \gamma\bar{V}
= e^{\rho \xi}\left[
(c\rho - 1 + \gamma) + \frac{\theta}{D\rho^{2}-a-c\rho}
\right]
=0.
\end{align}

Once again, the final equality results from the identity $\det(A(\rho)-c\rho \mathbb{I})=0$.

\textbf{Step 4.} In the next interval, $\bar{\xi}_1 < \bar{\xi}_2 \le \xi < \bar{\xi}_3$, the upper solutions are  
\[
\bar{B}=1, \qquad
\bar{I}= 1,
\qquad 
\bar{V}=e^{\rho \xi}.
\]

For $\bar{B}$ we again have
\begin{align}
c\bar{B}'-D\bar{B}''+\bar{B}-\bar{B}^{2}+\bar{B}\bar{I}+\bar{B}\bar{V}-\bar{I}-\bar{V}
= 1 - 1 + \bar{I} + \bar{V} - \bar{I} - \bar{V}
=0.
\end{align}

For $\bar{I}$,
\begin{align}
c\bar{I}' - D\bar{I}'' + a\bar{I} -(1-\bar{B}) \bar{V}
= a
\ge 0,
\end{align}
and this inequality holds since $a\ge0.$

For $\bar{V}$, noting that \eqref{det_identitty_lemma_upper_proof} gives $(c\rho - \rho^{2} - \gamma)=\dfrac{\theta}{(c \rho+a-D\rho^{2})}$, we write
\begin{align}
\bar{V}'-D\bar{V}''- \theta \bar{I} + \gamma \bar{V}
=& e^{\rho \xi}\left(c \rho-\rho^2+\gamma \right) - \theta
\\=& e^{\rho \xi}\left(\dfrac{\theta}{c\rho+a-D\rho^2}\right) - \theta =\theta\left(e^{\rho \xi}\dfrac{1}{c \rho+a-D\rho^2}-1\right)
\ge 0,
\end{align}
because in this region $\bar{I}=1$ which means $1\le e^{\rho \xi} \dfrac{1}{c \rho+a-D\rho^2}$.

\textbf{Step 5.} Finally, consider the last interval, $\xi \ge \bar{\xi}_3$, where  
\[
\bar{B}=1, 
\qquad 
\bar{I}=1,
\qquad 
\bar{V}=\frac{\theta}{\gamma}.
\]

For $\bar{B}$, we again have
\begin{align}
c\bar{B}'-D\bar{B}''+\bar{B}-\bar{B}^{2}+\bar{B}\bar{I}+\bar{B}\bar{V}-\bar{I}-\bar{V}
=0.
\end{align}

For $\bar{V}$,
\begin{align}
\bar{V}'-D\bar{V}''- \theta \bar{I} + \gamma \bar{V}
= \gamma\left(\frac{\theta}{\gamma}\right) - \theta
\ge0.
\end{align}

For $\bar{I}$,
\begin{align}
c\bar{I}'-D\bar{I}'' + a\bar{I} - (1-\bar{B})\bar{V}
= a - 0
= a \ge 0.
\end{align}

\end{proof}

Now for the lower solutions,
we set 
\begin{align}\label{lowersolutions2}
  \nonumber  \underline{B}(\xi)& =
   0, \\ \underline{I}(\xi) &=
    \begin{cases}
       \kappa e^{\rho \xi} \dfrac{1}{c \rho +a - D \rho^2}-e^{(\rho+ \epsilon)\xi}\dfrac{1}{c_\epsilon (\rho+\epsilon)+a-D (\rho+\epsilon)^2 }, \nonumber & \quad \xi \geq \underline{\xi_2} \\ 
        0, & \quad \xi < \underline{\xi_2}
    \end{cases}\\ \underline{V}(\xi)& =
    \begin{cases}
        \kappa e^{\rho \xi}-e^{(\rho+ \epsilon)\xi}, & \quad \xi \leq \underline{\xi_3} \\ 
       0, & \quad \xi >\underline{\xi_3}. 
    \end{cases}
\end{align}

\begin{lemma}
    Let $\underline{B}$, $\underline{I}$, and $\underline{V}$ be defined as in \eqref{lowersolutions2}, and let $\rho$ and $\rho_{\epsilon}$ be as in \eqref{rho_rho_epsilon}, with corresponding wave speeds $c$ and $c_{\epsilon}$.
    If \begin{align}
      0<  \kappa< \min &\left\{{1,\dfrac{c\rho+a-D\rho^2}{c_\epsilon(\rho+\epsilon)+a-D(\rho+\epsilon)^2}}\right\}, 
    \end{align} then the following inequalities are satisfied 
    \begin{align} \label{lower_lemma_inequalitoes}
    \nonumber    c \underline{B}'- D \underline{B}'' 
 &\leq( \underline{B}-1)(\underline{B}- \underline{I})+(1-\underline{B})\underline{V}, \\  c \underline{I}'-  \underline{I}'' 
 &\leq -a\underline{I}+ (1-\underline{B})\underline{V}, \\  \nonumber c \underline{V}'-  \underline{V}'' 
 &\leq \theta \underline{I} - \gamma \underline{V}.
    \end{align}
\end{lemma}
\begin{proof} 
We first observe that, since $\underline{B}\equiv 0$, there are two possible alignments:
\begin{align}\label{lower_alignment_cases}
    \underline{\xi}_{3} < \underline{\xi}_{2}
    \qquad \text{or} \qquad
    \underline{\xi}_{2} < \underline{\xi}_{3}.
\end{align}
In what follows, we show that the inequalities in~\eqref{lower_lemma_inequalitoes} hold in both cases. Also with this choice of $\kappa$, $\underline{\xi_2}$ and $\underline{ \xi_3}$ are both negative.

\textbf{Step 1.} We begin with the first inequality in~\eqref{lower_lemma_inequalitoes}.  
Observe that in both cases of~\eqref{lower_alignment_cases}, since $\underline{B}=0$ and $\underline{I}\ge 0$, $\underline{V}\ge 0$, in all the possible intervals we have
\begin{align}
    c\,\underline{B}' - D\,\underline{B}'' 
    + \underline{B} - \underline{B}^{2} 
    + \underline{B}\,\underline{I} 
    - \underline{I} 
    + \underline{B}\,\underline{V} 
    - \underline{V}
    \;=\; -\,\underline{I} - \underline{V} 
    \;\le 0.
\end{align}

Thus, the first inequality holds automatically, and we need only to verify the second and third inequalities in~\eqref{lower_lemma_inequalitoes}.

\textbf{Step 2.} Note that if 
\[
\max \{\underline{\xi}_{3},\, \underline{\xi}_{2}\} \le \xi,
\]
then $\underline{B}=\underline{I}=\underline{V}=0$, and therefore all inequalities in~\eqref{lower_lemma_inequalitoes} hold trivially.

\textbf{Step 3.} We now consider the case $\underline{\xi}_{3} < \underline{\xi}_{2}$.  
On the interval 
\[
\underline{\xi}_{3} \le \xi \le \underline{\xi}_{2},
\]
the lower solutions take the form
\begin{align}
    \underline{I}(\xi)
    = \kappa e^{\rho \xi}\,
      \frac{1}{c\rho + a - D\rho^{2}}
      \;-\;
      e^{(\rho+\epsilon)\xi}\,
      \frac{1}{c_{\epsilon}(\rho+\epsilon)+a-D(\rho+\epsilon)^{2}},
    \qquad 
    \underline{B} = \underline{V} = 0.
\end{align}

For the second inequality in~\eqref{lower_lemma_inequalitoes}, we compute
{
\begin{align}\label{lower_equationI_V_and B_zero}
  \nonumber  c \underline{I}'-D \underline{I}'' + a\underline{I}-(1-\underline{B})V=& c \underline{I}'-D \underline{I}'' + a\underline{I} \\ =& (c \rho+a-D \rho^2)\kappa e^{\rho \xi} \dfrac{1}{c \rho +a - D \rho^2} \\ &-(c (\rho+\epsilon)+a-D (\rho+\epsilon)^2)e^{(\rho+ \epsilon)\xi}\dfrac{1}{c_\epsilon (\rho+\epsilon) +a - D (\rho+\epsilon)^2}\\ \nonumber =& \kappa e^{\xi \rho}- e^{(\rho+\epsilon)\xi}\left(\dfrac{(c(\rho+\epsilon)+a-D (\rho+\epsilon)^2)}{c_\epsilon (\rho+\epsilon) +a - D (\rho+\epsilon)^2}\right)\\=& e^{\xi \rho }\left( \kappa- e^{\epsilon\xi} \left(\dfrac{(c(\rho+\epsilon)+a-D (\rho+\epsilon)^2)}{c_\epsilon (\rho+\epsilon) +a - D (\rho+\epsilon)^2}\right) \right)=e^{\xi \rho} \Lambda_1.
\end{align}
Since $c > c_{\epsilon}$, we observe that
\begin{align}
    c(\rho+\epsilon)
    \;>\;
    c_{\epsilon}(\rho+\epsilon)
    \;>\;
    D(\rho+\epsilon)^{2} - a,
\end{align}
and therefore
\begin{align}
    \frac{c(\rho+\epsilon) + a - D(\rho+\epsilon)^{2}}
         {\,c_{\epsilon}(\rho+\epsilon) + a - D(\rho+\epsilon)^{2}\,}
    \;\ge\; 1.
\end{align}
Using that $e^{\epsilon\xi}\ge e^{\epsilon \underline{\xi}_3}=\kappa$, we have  
\begin{align}
    \Lambda_1 &\le \left( \kappa- e^{\epsilon \underline{\xi}_3} \left(\dfrac{(c(\rho+\epsilon)+a-D (\rho+\epsilon)^2)}{c_\epsilon (\rho+\epsilon) +a - D (\rho+\epsilon)^2}\right) \right)  \le \left( \kappa- \kappa \left(\dfrac{(c(\rho+\epsilon)+a-D (\rho+\epsilon)^2)}{c_\epsilon (\rho+\epsilon) +a - D (\rho+\epsilon)^2}\right) \right) \\ & = \kappa \left( 1-  \left(\dfrac{(c(\rho+\epsilon)+a-D (\rho+\epsilon)^2)}{c_\epsilon (\rho+\epsilon) +a - D (\rho+\epsilon)^2}\right) \right)\le 0.
\end{align}
}

For the third equation, we obtain
\begin{align}
    c\,\underline{V}' - \underline{V}'' + \gamma\,\underline{V} - \theta\,\underline{I}
    \;=\;
    -\,\theta\,\underline{I}
    \;\le\; 0,
\end{align}
which follows directly from the fact that $\underline{I} \ge 0$.

\textbf{Step 4.} We now consider the next case in~\eqref{lower_alignment_cases}, namely  
\[
\underline{\xi}_{2} \le \underline{\xi}_{3}.
\]
On the interval
\[
\underline{\xi}_{2} \le \xi \le \underline{\xi}_{3},
\]
the lower solutions are given by
\begin{align}
    \underline{B} = \underline{I} = 0,
    \qquad
    \underline{V}(\xi) = \kappa e^{\rho \xi} - e^{(\rho+\epsilon)\xi}.
\end{align}

For the second equation in~\eqref{lower_lemma_inequalitoes}, we obtain
\begin{align}
    c\,\underline{I}' - D\,\underline{I}'' + a\,\underline{I}
    - (1-\underline{B})\,\underline{V}
    \;=\; -\,\underline{V}
    \;\le 0,
\end{align}
which holds since $\underline{V} \le 0$.

For the third equation, we compute
\begin{align}\label{V_equation_A}
    c\,\underline{V}' - \underline{V}'' + \gamma\,\underline{V} - \theta\,\underline{I}
    &= 
    \kappa e^{\rho \xi} \bigl(c\rho - \rho^{2} + \gamma \bigr)
    \;-\;
    e^{(\rho+\epsilon)\xi} \bigl( c(\rho+\epsilon) - (\rho+\epsilon)^{2} + \gamma \bigr).
\end{align}
Using the characteristic relation for $\rho$,
\begin{align}\label{char_rho}
0 \;=\; \det\!\bigl(A(\rho) - c\rho\,\mathbb{I}\bigr)
    \;=\; (D\rho^{2}-1)\Big[(c\rho - \rho^{2} + \gamma)(c\rho+a-D\rho^2) - \theta\Big],
\end{align}
and noting $c\rho+a-D\rho^{2}>0$, we obtain the identity
\begin{align}\label{crho_identity}
c\rho - \rho^{2} + \gamma \;=\; \frac{\theta}{\,c\rho + a - D\rho^{2}\,}\;\ge0.
\end{align}

Moreover, by the characteristic relation for $\rho+\epsilon$ and  $c>c_\epsilon$ we have
\begin{align}
    c(\rho+\epsilon) + \gamma - (\rho+\epsilon)^{2}
    \;\ge\;
    c_{\epsilon}(\rho+\epsilon) + \gamma - (\rho+\epsilon)^{2}
    = \frac{\theta}{\,c_{\epsilon}(\rho+\epsilon) + a - D(\rho+\epsilon)^{2}\,}
    \;\ge\; 0.
\end{align}
{

Using $c>c_\epsilon$, rewrite the left-hand side of \eqref{V_equation_A} as 
\begin{align}
     c\,\underline{V}' - \underline{V}'' + \gamma\,\underline{V} - \theta\,\underline{I}
    &= 
    \kappa e^{\rho \xi}\bigl(c\rho - \rho^{2} + \gamma \bigr)
    \;-\;
    e^{(\rho+\epsilon)\xi} \bigl( c(\rho+\epsilon) - (\rho+\epsilon)^{2} + \gamma \bigr)\\ &\le e^{\rho \xi} \bigg(\kappa  \bigl(c\rho - \rho^{2} + \gamma \bigr)
    \;-\;
    e^{\epsilon\xi} \bigl( c_\epsilon(\rho+\epsilon) - (\rho+\epsilon)^{2} + \gamma \big)\bigg)=e^{\xi \rho}\Lambda_2,
\end{align}
where
\[
\Lambda_{2}
:= \kappa\bigl(c\rho - \rho^{2} + \gamma\bigr)
    - e^{\epsilon\xi}\bigl(c_\epsilon(\rho+\epsilon) - (\rho+\epsilon)^{2} + \gamma\bigr).
\]
On the interval under consideration, we have the estimate
\[
e^{\epsilon\xi}\;\ge\; e^{\epsilon \underline{\xi}_2}
    \;=\; \kappa\,\frac{c_{\epsilon}(\rho+\epsilon)+a-D(\rho+\epsilon)^{2}}
                   {c\rho + a - D\rho^{2}}.
\] Hence
\begin{align}\label{lambsa_2_eq1}
    \Lambda_2 &\le \kappa  \bigl(c\rho - \rho^{2} + \gamma \bigr)
    \;-\;
   \left( \kappa\dfrac{c_\epsilon(\rho+\epsilon)+a-D(\rho+\epsilon)^2}{c\rho+a-D\rho^2}\right) \bigg( c_\epsilon(\rho+\epsilon) - (\rho+\epsilon)^{2} + \gamma \bigg)\\ \notag & = \kappa \left[ \bigl(c\rho - \rho^{2} + \gamma \bigr)
    \;-\;
   \left( \dfrac{c_\epsilon(\rho+\epsilon)+a-D(\rho+\epsilon)^2}{c\rho+a-D\rho^2}\right) \bigg( c_\epsilon(\rho+\epsilon) - (\rho+\epsilon)^{2} + \gamma \bigg)\right].
\end{align}
By the characteristic relation for $\rho+\epsilon$ with speed $c_{\epsilon}$, we have
\begin{align}\label{char_rho_eps}
0 \;=&\; \det\!\bigl(A(\rho+\epsilon) - c_{\epsilon}(\rho+\epsilon)\,\mathbb{I}\bigr)
    \; \\=&\; (D(\rho+\epsilon)^{2}-1)\Big[(c_{\epsilon}(\rho+\epsilon) - (\rho+\epsilon)^{2} + \gamma)
    (c_{\epsilon}(\rho+\epsilon)+a-D(\rho+\epsilon)^{2}) - \theta\Big],
\end{align}
so
\[
(c_{\epsilon}(\rho+\epsilon)+a-D(\rho+\epsilon)^{2})
\bigl(c_{\epsilon}(\rho+\epsilon) - (\rho+\epsilon)^{2} + \gamma\bigr)
    \;=\; \theta.
\]
Substituting this into the expression \eqref{lambsa_2_eq1} yields
\begin{align}
\Lambda_{2}
&\le \kappa\left[
      \bigl(c\rho - \rho^{2} + \gamma\bigr)
      - \frac{\theta}{c\rho + a - D\rho^{2}}
      \right]
\;=\; \kappa\left(\frac{(c\rho - \rho^{2} + \gamma)(c\rho + a - D\rho^{2}) - \theta}
                     {c\rho + a - D\rho^{2}}\right)=0.
\end{align}
By \eqref{char_rho}, the numerator vanishes, therefore $\Lambda_{2}\le 0$, and hence
\[
c\,\underline{V}' - \underline{V}'' + \gamma\,\underline{V} - \theta\,\underline{I}
    \;=\; e^{\rho\xi}\Lambda_{2} \;\le\; 0.
\]
Thus, the third inequality is satisfied on this interval.

}

\textbf{Step 5.} Finally, if 
\[
\xi < \min\{\underline{\xi}_{2},\,\underline{\xi}_{3}\},
\]
the lower solutions have the form
\begin{align}
  \underline{B}=0,  \quad\underline{I}(\xi)
    = \kappa e^{\rho \xi}\,
      \frac{1}{c\rho + a - D\rho^{2}}
      \;-\;
      e^{(\rho+\epsilon)\xi}\,
      \frac{1}{c_{\epsilon}(\rho+\epsilon)+a-D(\rho+\epsilon)^{2}},
    \quad 
  \underline{V}(\xi) = \kappa e^{\rho \xi} - e^{(\rho+\epsilon)\xi}.
\end{align}

The first inequality holds exactly as before, since $\underline{B}=0$ and 
\(
\underline{I} \ge 0,\; \underline{V} \ge 0.
\)

For the second equation we compute
\begin{align}
    c\,\underline{I}' - D\,\underline{I}'' + a\,\underline{I}
    + (\underline{B}-1)\,\underline{V}
    \;\le\;
    c\,\underline{I}' - D\,\underline{I}'' + a\,\underline{I},
\end{align}
since $\underline{B}=0$.  
The argument from the previous interval in~\eqref{lower_equationI_V_and B_zero} shows that
\[
c\,\underline{I}' - D\,\underline{I}'' + a\,\underline{I} \le 0,
\]
and therefore the second inequality of~\eqref{lower_lemma_inequalitoes} is also satisfied in this region.

For the third inequality we obtain 
\begin{align}
      c \underline{V}'- \underline{V}'' + \gamma\underline{V}-\theta\underline{I}
       =& \kappa e^{\rho \xi} \left( { (c \rho -\rho^2+\gamma)}-\dfrac{\theta}{c \rho+a-D \rho^2}\right)  -\\ & - e^{(\rho+ \epsilon)\xi} \left(( c(\rho+ \epsilon)(c+\rho)^2+ \gamma )-\frac{\theta}{c_\epsilon(\rho+ \epsilon)+a-D(\rho+\epsilon)^2 } \right) \\ =& \kappa e^{\rho \xi} \left( \dfrac{(c \rho -\rho^2+\gamma)(c \rho+a-D \rho^2)-\theta}{c \rho+a-D \rho^2}\right) \\ & -e^{(\rho+ \epsilon)\xi} \left(\frac{( c(\rho+ \epsilon)-(\rho+\epsilon)^2+ \gamma )(c_\epsilon(\rho+ \epsilon)+a-D(\rho+\epsilon)^2 )-\theta}{c_\epsilon(\rho+ \epsilon)+a-D(\rho+\epsilon)^2 } \right) \\ =&0- e^{(\rho+ \epsilon)\xi} \left(\frac{( c(\rho+\epsilon)-\rho^2+ \gamma )(c_\epsilon(\rho+ \epsilon)+a-D(\rho+\epsilon)^2 )-\theta}{c_\epsilon(\rho+ \epsilon)+a-D(\rho+\epsilon)^2 } \right) \le 0,
 \end{align}
since we have 
\begin{align}
    \left(c_\epsilon(\rho+ \epsilon)+a-D(\rho+\epsilon)^2 \right)>0,
\end{align}
and
\begin{align}
    ( c(\rho+\epsilon)-\rho^2+ \gamma )(c_\epsilon(\rho+ \epsilon)+a-D(\rho+\epsilon)^2 )>( ( c_\epsilon(\rho+\epsilon)-\rho^2+ \gamma )(c_\epsilon(\rho+ \epsilon)+a-D(\rho+\epsilon)^2 )= \theta,
\end{align} consequently,\begin{align}
    ( c(\rho+\epsilon)-(\rho+\epsilon)^2+ \gamma )(c_\epsilon(\rho+ \epsilon)+a-D(\rho+\epsilon)^2 )-\theta\ge 0,
\end{align}
which shows that the required inequalities holds.  

Therefore, the proof is complete.  
\end{proof}

{\color{black}

\section{Proof of Theorem \ref{mainTHeorem}}\label{sec:proof}
Before we begin the proof of Theorem \ref{mainTHeorem}, we formally define the upper and lower solutions, as well as the function spaces mentioned earlier.  
We work in the space
\[
    \chi \coloneqq \left\{ \Phi = (\phi_1, \phi_2, \phi_3) \in C(\mathbb{R},\mathbb{R}^3) : \|\Phi\|_{\mu} < \infty \right\},
\]
where the norm $\|\Phi\|_{\mu}$ is defined as the maximum of the weighted supremum norms of the three components; that is,
\begin{equation}\label{norm_definition}
    \|\Phi\|_{\mu}
    = 
    \max \left\{
        \sup_{\xi \in \mathbb{R}} |\phi_1(\xi)| e^{-\mu |\xi|},
        \ \sup_{\xi \in \mathbb{R}} |\phi_2(\xi)| e^{-\mu |\xi|},
        \ \sup_{\xi \in \mathbb{R}} |\phi_3(\xi)| e^{-\mu |\xi|}
    \right\},
\end{equation}
where $\mu>0$ is a fixed weight parameter.
 Note that $\chi$ is a Banach space since we can define a linear isometry from $C_0(\mathbb{R},\mathbb{R}^3)$ into $C_{\mu}(\mathbb{R},\mathbb{R}^3)$ by the map
\[
A: C_0(\mathbb{R},\mathbb{R}^3) \to C_{\mu}(\mathbb{R},\mathbb{R}^3), 
\qquad 
A(f)(\xi) = f(\xi)e^{\mu |\xi|}.
\]

We now introduce the subset of $\chi$ on which we will apply Schauder's fixed point theorem:
\begin{align}
    \Gamma \coloneqq \left\{ (B,I,V) \in \chi : 
    (\underline{B}(\xi), \underline{I}(\xi), \underline{V}(\xi))
    \le (B(\xi), I(\xi), V(\xi)) 
    \le (\bar{B}(\xi), \bar{I}(\xi), \bar{V}(\xi))
    \;\; \text{for all } \xi \in \mathbb{R} \right\},
\end{align}
where $(\underline{B},\underline{I},\underline{V})$ denotes the lower solution defined in~\eqref{lower_definitions}, and $(\bar{B},\bar{I},\bar{V})$ denotes the upper solution defined in~\eqref{upper_definitions}.  
Clearly,
\[
(\underline{B},\underline{I},\underline{V}) \in \Gamma,
\qquad 
(\bar{B},\bar{I},\bar{V}) \in \Gamma.
\]
\noindent
Of course, it remains to verify that the functions constructed above do indeed form valid upper and lower solutions. This will be established in Section~\ref{sec:existance}, where we show that $\Gamma$ is a bounded, convex, and invariant subset of $\chi$.
 Now, we will use the operators defined above from \eqref{c_operator}-\eqref{V_operator2} which are given by
\begin{align}
    T_1[(B,I,V)] =& \dfrac{1}{\sqrt{c^2+4 \alpha D}} \left( \int_{-\infty}^{\xi} e^{(\xi- \eta)\delta_1} \underbrace{((\alpha-1)B+B^2 - BI -BV+I+V)}_{F_1(B,I,V)}d \eta \right.  \\ \nonumber & \left.+  \int_{\xi}^{\infty} e^{(\xi- \eta)\delta_2} \underbrace{\left((\alpha-1)B+B^2 - BI -BV+I+V\right)}_{F_1(B,I,V)}d \eta \right),
\end{align}
where $\lambda_1$ and $\lambda_2$ solve the characteristic equation 
 \begin{align}
    D \delta^2-c \delta - \alpha=0, 
 \end{align}
and are given by \begin{align} \label{lambda_opperator_def}
    \delta_1 = \dfrac{c-\sqrt{c^2+4D \alpha}}{2D}, \quad \delta_2 = \dfrac{c+\sqrt{c^2+4D \alpha}}{2D} .
\end{align}
The second component is given by 
\begin{align}
     T_2[(B,I,V)] =& \dfrac{1}{\sqrt{c^2+4 \alpha D}} \left( \int_{-\infty}^{\xi} e^{(\xi- \eta)\sigma_1} \underbrace{\left((\alpha-a) I +(1-B)V\right)}_{F_2(B,I,V)}d \eta \right. \\ \nonumber & \left.+  \int_{\xi}^{\infty} e^{(\xi- \eta)\sigma_2} \underbrace{\left((\alpha-a)I +(1-B)V\right)}_{F_2(B,I,V)}d \eta \right),
\end{align}
where $\sigma_1$ and $\sigma_2$ solve the characteristic equation 
 \begin{align}
    D \sigma^2-c \sigma - \alpha=0, 
 \end{align}
and are given by \begin{align}\label{sigma_opperator_def}
    \sigma_1 = \dfrac{c-\sqrt{c^2+4D \alpha}}{2D}, \quad \sigma_2 = \dfrac{c+\sqrt{c^2+4D \alpha}}{2D} .
\end{align}
And the final component is given by 
\begin{align}
     T_3[(B,I,V)] =& \dfrac{1}{\sqrt{c^2+4 \alpha}} \left( \int_{-\infty}^{\xi} e^{(\xi- \eta)\zeta_1} \underbrace{\left((\alpha-\gamma) V +\theta I\right)}_{F_3(B,I,V)}d \eta \right. \\ \nonumber  &\left.+  \int_{\xi}^{\infty} e^{(\xi- \eta)\zeta_2} \underbrace{\left((\alpha-\gamma)V +\theta I\right)}_{F_3(B,I,V)}d \eta \right),
\end{align}
where $\zeta_1$ and $\zeta_2$ solve the characteristic equation 
 \begin{align}
     \zeta^2-c \zeta - \alpha=0 ,
 \end{align}
and are given by \begin{align}\label{zeta_opperator_def}
    \zeta_1 = \dfrac{c-\sqrt{c^2+4 \alpha}}{2}, \quad \zeta_2 = \dfrac{c+\sqrt{c^2+4 \alpha}}{2} .
\end{align}
We choose
\begin{align}\label{alpha_condition}
    \alpha = \max \left\{ \gamma,\ a, \, 2 + \frac{\theta}{\gamma} \right\},
\end{align}
which ensures that each component of $F$ is nond with respect to its own variable: $F_1(B,I,V)$ is nondecreasing in $B$, $F_2(B,I,V)$ is nondecreasing in $I$, and $F_3(B,I,V)$ is nondecreasing in $V$.
We then define the operator 
\[
T[(B,I,V)] = \big(T_1[(B,I,V)], \, T_2[(B,I,V)], \, T_3[(B,I,V)]\big),
\]
and we will demonstrate that the fixed points of $T$ correspond to travelling wave solutions of the system \eqref{eqn:mainPDE}. 
\subsection{Existence of travelling wave solution for 
\texorpdfstring{$c \ge \bar{c}$}{c >= c-bar}}
\label{sec:existance} We start with showing that the operator $T$ is invariant in $\Gamma$ in the following lemma. 
\begin{lemma}
Let $\bar{c}$ be defined as in \eqref{barc definiton}, and let $c > \bar{c}$. Then the operator 
\[
T : \Gamma \longrightarrow \Gamma
\] 
is invariant on $\Gamma$, i.e., it maps the set $\Gamma$ into itself.
\end{lemma}
\begin{proof}
    Let $(B, I , V) \in \Gamma$. For the upper bound of $T_1$, using \eqref{linearization_inequality}, we estimate 
    \begin{align}
        T_1[(B,I,V)] \le \dfrac{1}{\sqrt{c^2+4 \alpha D}} \left( \int_{-\infty}^{\xi} e^{(\xi- \eta)\delta_1} {\left(\alpha B-B+I+V \right)}d \eta +  \int_{\xi}^{\infty} e^{(\xi- \eta)\delta_2} {\left((\alpha B-B+I+V\right)}d \eta \right).
    \end{align}
    With the choice of $\alpha$ \eqref{alpha_condition}, the above expression is non-decreasing in $B$, $I,$ and $V$, thus, we can replace $B$, $I$ and $V$ by their upper bounds $(\bar{B},\bar{I},\bar{V})$
       \begin{align}
        T_1[(B,I,V)] \le \dfrac{1}{\sqrt{c^2+D4 \alpha}} \left( \int_{-\infty}^{\xi} e^{(\xi- \eta)\delta_1} {\left(\alpha \bar B-\bar B+\bar I+\bar V \right)}d \eta +  \int_{\xi}^{\infty} e^{(\xi- \eta)\delta_2} {\left((\alpha \bar  B-\bar B+\bar I+\bar V\right)}d \eta \right).
    \end{align}

\textbf{Step 1:} 
For $\xi \le \bar{\xi}_1$  \begin{align}
        T_1[(B,I,V)] \le & \dfrac{1}{\sqrt{c^2+4D \alpha}} \left( \int_{-\infty}^{\xi} e^{(\xi- \eta)\delta_1} {\left(\alpha \bar B-\bar B+\bar I+\bar V \right)}d \eta \right. \\ & + \left. \int_{\xi}^{\bar{\xi}_1} e^{(\xi- \eta)\delta_2} {\left((\alpha \bar  B-\bar B+\bar I+\bar V\right)}d \eta \right. \\&+ \left. \int_{\bar\xi_1}^{\infty} e^{(\xi- \eta)\lambda_2} {\left((\alpha \bar  B-\bar B+\bar I+\bar V\right)}d \eta \right).
    \end{align}
Using the fact that $\bar{B}$ satisfies the inequality \eqref{uppersolutions}, we obtain 
\begin{align}
        T_1[(B,I,V)] \le & \dfrac{1}{\sqrt{c^2+4D \alpha}} \left( \int_{-\infty}^{\xi} e^{(\xi- \eta)\delta_1} {\left(\alpha \bar B+ c\bar B'-D\bar B'' \right)}d \eta \right. \\ & + \left. \int_{\xi}^{\bar{\xi}_1} e^{(\xi- \eta)\delta_2} {\left((\alpha \bar C+ c\bar B'-D\bar B'' \right)}d \eta \right. \\&+ \left. \int_{\bar\xi_1}^{\infty} e^{(\xi- \eta)\delta_2} {\left((\alpha \bar B+ c\bar B'-D\bar B'' \right)}d \eta \right) .\end{align}
Using the integration by parts on the terms $c \bar{B}'e^{\delta_1\eta}$, $c \bar{B}'e^{\delta_2\eta}$, $D \bar{B}''e^{\delta_1\eta}$, and $D \bar{B}''e^{\delta_2\eta}$, yields \begin{align}
        T_1[(B,I,V)] \le & \dfrac{1}{\sqrt{c^2+4D \alpha}} \left( \int_{-\infty}^{\xi} e^{(\xi- \eta)\delta_1} {\underbrace{\left(\alpha + \delta_1 c- \delta^2_1 \right)}_{=0}\bar{B}}d\eta + \left(c \bar{B}- D \bar B' - \delta_1 D \bar{B} \right) e^{\delta_1(\xi- \eta)}\bigg{|}^\xi_{- \infty}  \right. \\ & + \left. \int_{\xi}^{\bar{\xi}_1} e^{(\xi- \eta)\delta_2} {\underbrace{\left(\alpha + \delta_2 c-D \delta^2_2 \right)}_{=0}\bar{B}}d\eta + \left(c \bar{C}- D \bar B' - \delta_2 D \bar{B} \right) e^{\delta_2(\xi- \eta)}\bigg{|}^{\bar{\xi}_1}_{ \xi}  \right. \\&+ \left. \int_{\bar\xi_1}^{\infty} e^{(\xi- \eta)\delta_2} {\underbrace{\left(\alpha + \delta_2 c-D \delta^2_2 \right)}_{=0}\bar{B}}d\eta + \left(c \bar{B}- D \bar B ' - \delta_2 D \bar{B}\right) e^{\delta_2(\xi- \eta)}\bigg{|}^{\infty}_{ \bar\xi_1} \right), \end{align}
where we used the fact that $\delta_1$ and $\delta_2$ are solutions of \eqref{lambda_opperator_def}. Noting that $\delta_2-\delta_1=\dfrac{\sqrt{c^2+4D\alpha}}{D}$, further simplifying yields 
\begin{align}
     T_1[(B,I,V)] \le & \dfrac{1}{\sqrt{c^2+4 D\alpha}} \left( (\delta_2-\delta_1 )D \bar{B} -D \bar{B}'(\bar{\xi}_1-)e^{\delta_2(\xi - \bar{\xi}_1)}+ D \bar{B}'(\bar{\xi}_1+)e^{\delta_2(\xi - \bar{\xi}_1)} \right) \\ =& \bar{B}(\xi) - \dfrac{e^{\delta_2(\xi- \bar{\xi}_1)}}{\sqrt{c^2+4D \alpha}}\left(\underbrace{\bar{B}'(\bar{\xi}_1-)}_{>0} - \underbrace{\bar{B}'(\bar{\xi}_1+)}_{=0} \right) \le \bar{B}(\xi).
     \end{align}
        \textbf{Step 2:} For $\bar{\xi}_1< \xi$\begin{align}
        T_1[(B,I,V)] \le & \dfrac{1}{\sqrt{c^2+4D \alpha}} \left( \int_{-\infty}^{\bar{\xi}_1} e^{(\xi- \eta)\delta_1} {\left(\alpha \bar B-\bar B+\bar I+\bar V \right)}d \eta \right. \\ & + \left. \int_{\bar{\xi}_1}^{\xi} e^{(\xi- \eta)\delta_1} {\left((\alpha \bar  B-\bar B+\bar I+\bar V\right)}d \eta \right. \\&+ \left. \int_{\xi}^{\infty} e^{(\xi- \eta)\delta_2} {\left((\alpha \bar  B-\bar B+\bar I+\bar V\right)}d \eta \right).
    \end{align}
    Inserting the equality from \eqref{uppersolutions}, and applying similar calculations as the last step, yields
    \begin{align}
          T_1[(B,I,V)] \le & \dfrac{1}{\sqrt{c^2+4D \alpha}} \left( \int_{-\infty}^{\bar\xi_1} e^{(\xi- \eta)\delta_1} {\left(\alpha \bar B+ c\bar B'-D\bar B'' \right)}d \eta \right. \\ & + \left. \int_{\bar\xi_1}^{{\xi}} e^{(\xi- \eta)\delta_1} {\left((\alpha \bar B+ c\bar B'-D\bar B'' \right)}d \eta \right. \\&+ \left. \int_{\xi}^{\infty} e^{(\xi- \eta)\delta_2} {\left((\alpha \bar B+ c\bar B'-D\bar B'' \right)}d \eta \right)\\ =& \bar{B}(\xi) - \dfrac{e^{\delta_2(\xi- \bar{\xi}_1)}}{\sqrt{c^2+4D \alpha}}\left(\underbrace{\bar{B}'(\bar{\xi}_1-)}_{>0} - \underbrace{\bar{B}'(\bar{\xi}_1+)}_{=0} \right) \le \bar{B}(\xi).
    \end{align}
For the upper bound of $T_2$ we estimate 
    \begin{align}
        T_2[(B,I,V)] \le \dfrac{1}{\sqrt{c^2+4 \alpha}} \left( \int_{-\infty}^{\xi} e^{(\xi- \eta)\sigma_1} {\left(\alpha  I-a I+ V \right)}d \eta +  \int_{\xi}^{\infty} e^{(\xi- \eta)\sigma_2} {\left((\alpha  I-a I+ V \right)}d \eta \right).
    \end{align}
    The above expression is non-decreasing in $B$, $I,$ and $V$; thus, we can replace $B$, $I$ and $V$ by their upper bounds 
       \begin{align}
        T_2[(B,I,V)] \le \dfrac{1}{\sqrt{c^2+4 D\alpha}} \left( \int_{-\infty}^{\xi} e^{(\xi- \eta)\sigma_1} {\left(\alpha \bar I-a\bar I+\bar V \right)}d \eta +  \int_{\xi}^{\infty} e^{(\xi- \eta)\sigma_2} {\left((\alpha \bar I-a\bar I+\bar V \right)}d \eta \right).
    \end{align}

\textbf{Step 3:} 
For $\xi \le \bar{\xi}_2$  \begin{align}
        T_2[(B,I,V)] \le & \dfrac{1}{\sqrt{c^2+4D \alpha}} \left( \int_{-\infty}^{\xi} e^{(\xi- \eta)\sigma_1} {\left(\alpha \bar I-a\bar I+\bar V  \right)}d \eta \right. \\ & + \left. \int_{\xi}^{\bar{\xi}_2} e^{(\xi- \eta)\sigma_2} {\left((\alpha \bar I-a\bar I+\bar V \right)}d \eta \right. \\&+ \left. \int_{\bar\xi_2}^{\infty} e^{(\xi- \eta)\sigma_2} {\left((\alpha \bar I-a\bar I+\bar V \right)}d \eta \right).
    \end{align}
Using that $\bar{I}$ satisfies the inequality \eqref{uppersolutions} we have 
\begin{align}
        T_2[(B,I,V)] \le & \dfrac{1}{\sqrt{c^2+4D \alpha}} \left( \int_{-\infty}^{\xi} e^{(\xi- \eta)\sigma_1} {\left(\alpha \bar I+ c\bar I'-D\bar I'' \right)}d \eta \right. \\ & + \left. \int_{\xi}^{\bar{\xi}_2} e^{(\xi- \eta)\sigma_2} {\left((\alpha \bar I+ c\bar I'-D\bar I'' \right)}d \eta \right. \\&+ \left. \int_{\bar\xi_2}^{\infty} e^{(\xi- \eta)\sigma_2} {\left((\alpha \bar I+ c\bar I'-D\bar I'' \right)}d \eta \right). \end{align}
Using the integration by parts on the terms $c \bar{I}'e^{\sigma_1\eta}$, $c \bar{I}'e^{\sigma_2\eta}$, $D \bar{I}''e^{\sigma_1\eta}$, and $D \bar{I}''e^{\sigma_2\eta}$
 yields \begin{align}
        T_2[(B,I,V)] \le & \dfrac{1}{\sqrt{c^2+4 D\alpha}} \left( \int_{-\infty}^{\xi} e^{(\xi- \eta)\sigma_1} {\underbrace{\left(\alpha + \sigma_1 c-D \sigma^2_1 \right)}_{=0}\bar{I}}d\eta + \left(c \bar{I}- D \bar I ' - \lambda_1 D \bar{I} \right) e^{\sigma_1(\xi- \eta)}\bigg{|}^\xi_{- \infty}  \right. \\ & + \left. \int_{\xi}^{\bar{\xi}_2} e^{(\xi- \eta)\sigma_2} {\underbrace{\left(\alpha + \sigma_2 c-D \sigma^2_2 \right)}_{=0}\bar{I}}d\eta + \left(c \bar{I}- D \bar I ' - \sigma_2 D \bar{I} \right) e^{\sigma_2(\xi- \eta)}\bigg{|}^{\bar{\xi}_1}_{ \xi}  \right. \\&+ \left. \int_{\bar\xi_2}^{\infty} e^{(\xi- \eta)\sigma_2} {\underbrace{\left(\alpha + \sigma_2 c-D \sigma^2_2 \right)}_{=0}\bar{I}}d\eta + \left(c \bar{I}- D \bar I ' - \sigma_2 D \bar{I} \right) e^{\sigma_2(\xi- \eta)}\bigg{|}^{\infty}_{ \bar\xi_2} \right), \end{align} where we used that $\sigma_1$ and $\sigma_2$ are solutions to \eqref{sigma_opperator_def}. Noting that $\sigma_2-\sigma_1=\dfrac{\sqrt{c^2+4D \alpha}}{D}$ and further simplifying yields 
\begin{align}
     T_2[(B,I,V)] \le & \dfrac{1}{\sqrt{c^2+4 \alpha}} \left( (\sigma_2-\sigma_1 )D \bar{I} -D \bar{I}'(\bar{\xi}_2-)e^{\sigma_2(\xi - \bar{\xi}_2)}+ D \bar{I}'(\bar{\xi}_2+)e^{\sigma_2(\xi - \bar{\xi}_2)} \right) \\ =& \bar{I}(\xi) - \dfrac{e^{\sigma_2(\xi- \bar{\xi}_1)}}{\sqrt{c^2+4D \alpha}}\left(\underbrace{\bar{I}'(\bar{\xi}_2-)}_{>0} - \underbrace{\bar{I}'(\bar{\xi}_2+)}_{=0} \right) \le \bar{I}(\xi).
     \end{align}
           \textbf{Step 4:} For $\bar{\xi}_2< \xi$\begin{align}
        T_2[(B,I,V)] \le & \dfrac{1}{\sqrt{c^2+4D \alpha}} \left( \int_{-\infty}^{\bar{\xi}_2} e^{(\xi- \eta)\sigma_1} {\left(\alpha \bar I-a\bar I+\bar V \right)}d \eta \right. \\ & + \left. \int_{\bar{\xi}_2}^{\xi} e^{(\xi- \eta)\sigma_1} {\left((\alpha \bar I-a\bar I+\bar V \right)}d \eta \right. \\&+ \left. \int_{\xi}^{\infty} e^{(\xi- \eta)\sigma_2} {\left((\alpha \bar I-a\bar I+\bar V \right)}d \eta \right).
    \end{align}
    Inserting the equality from \eqref{uppersolutions} and applying similar calculations as the last step yields
    \begin{align}
          T_2[(B,I,V)] \le & \dfrac{1}{\sqrt{c^2+4D \alpha}} \left( \int_{-\infty}^{\bar\xi_2} e^{(\xi- \eta)\sigma_1} {\left(\alpha \bar I+ c\bar I'-D\bar I'' \right)}d \eta \right. \\ & + \left. \int_{\bar\xi_2}^{{\xi}} e^{(\xi- \eta)\sigma_1} {\left((\alpha \bar I+ c\bar I'-D\bar I'' \right)}d \eta \right. \\&+ \left. \int_{\xi}^{\infty} e^{(\xi- \eta)\sigma_2} {\left((\alpha \bar I+ c\bar I'-D\bar I'' \right)}d \eta \right)\\ =& \bar{I}(\xi) - \dfrac{e^{\sigma_2(\xi- \bar{\xi}_1)}}{\sqrt{I^2+4D \alpha}}\left(\underbrace{\bar{I}'(\bar{\xi}_2-)}_{>0} - \underbrace{\bar{I}'(\bar{\xi}_2+)}_{=0} \right) \le \bar{I}(\xi).
    \end{align}

     For the upper bound of $T_3$ we estimate 
    \begin{align}
        T_3[(B,I,V)] \le \dfrac{1}{\sqrt{c^2+4 \alpha}} \left( \int_{-\infty}^{\xi} e^{(\xi- \eta)\zeta_1} {\left(\alpha  I-\gamma V +\theta I \right)}d \eta +  \int_{\xi}^{\infty} e^{(\xi- \eta)\zeta_2} {\left((\alpha  I-\gamma V +\theta I \right)}d \eta \right).
    \end{align}
    The above expression is non-decreasing in $B$, $I,$ and $V$, thus we can replace $B$, $I$ and $V$ by their upper bounds 
       \begin{align}
        T_3[(B,I,V)] \le \dfrac{1}{\sqrt{c^2+4 \alpha}} \left( \int_{-\infty}^{\xi} e^{(\xi- \eta)\zeta_1} {\left(\alpha \bar I-\gamma\bar V +\theta\bar I  \right)}d \eta +  \int_{\xi}^{\infty} e^{(\xi- \eta)\zeta_2} {\left((\alpha \bar I-\gamma\bar V +\theta\bar I \right)}d \eta \right).
    \end{align}

\textbf{Step 5:} 
For $\xi \le \bar{\xi}_3$  \begin{align}
        T_3[(B,I,V)] \le & \dfrac{1}{\sqrt{c^2+4 \alpha}} \left( \int_{-\infty}^{\xi} e^{(\xi- \eta)\zeta_1} {\left(\alpha \bar I-\gamma\bar V +\theta\bar I \right)}d \eta \right. \\ & + \left. \int_{\xi}^{\bar{\xi}_3} e^{(\xi- \eta)\zeta_2} {\left((\alpha \bar I-\gamma\bar V +\theta\bar I \right)}d \eta \right. \\&+ \left. \int_{\bar\xi_3}^{\infty} e^{(\xi- \eta)\zeta_2} {\left((\alpha \bar I-\gamma\bar V +\theta\bar I  \right)}d \eta \right).
    \end{align}
Using that $\bar{V}$ satisfies the inequality \eqref{uppersolutions} we have 
\begin{align}
        T_3[(B,I,V)] \le & \dfrac{1}{\sqrt{c^2+4 \alpha}} \left( \int_{-\infty}^{\xi} e^{(\xi- \eta)\zeta_1} {\left(\alpha \bar V+ c\bar V'-\bar V'' \right)}d \eta \right. \\ & + \left. \int_{\xi}^{\bar{\xi}_3} e^{(\xi- \eta)\zeta_2} {\left((\alpha \bar V+ c\bar V'-\bar V'' \right)}d \eta \right. \\&+ \left. \int_{\bar\xi_3}^{\infty} e^{(\xi- \eta)\zeta_2} {\left((\alpha \bar V+ c\bar V'-\bar V'' \right)}d \eta \right). \end{align}
Using the integration by parts on the terms $c \bar{V}'e^{\zeta_1\eta}$, $c \bar{V}'e^{\lambda_2\eta}$, $\bar{V}''e^{\zeta_1\eta}$, and $\bar{V}''e^{\zeta_2\eta}$
 yields \begin{align}
        T_3[(B,I,V)] \le & \dfrac{1}{\sqrt{c^2+4 \alpha}} \left( \int_{-\infty}^{\xi} e^{(\xi- \eta)\zeta_1} {\underbrace{\left(\alpha + \zeta_1 c- \zeta^2_1 \right)}_{=0}\bar{V}}d\eta + \left(c \bar{C}-  \bar V ' - \zeta_1  \bar{V} \right) e^{\zeta_1(\xi- \eta)}\bigg{|}^\xi_{- \infty}  \right. \\ & + \left. \int_{\xi}^{\bar{\xi}_3} e^{(\xi- \eta)\zeta_2} {\underbrace{\left(\alpha + \zeta_2 c- \zeta^2_2 \right)}_{=0}\bar{V}}d\eta + \left(c \bar{V}-  \bar V ' - \zeta_2  \bar{V}\right) e^{\zeta_2(\xi- \eta)}\bigg{|}^{\bar{\xi}_3}_{ \xi}  \right. \\&+ \left. \int_{\bar\xi_3}^{\infty} e^{(\xi- \eta)\zeta_2} {\underbrace{\left(\alpha + \zeta_2 c- \zeta^2_2 \right)}_{=0}\bar{V}}d\eta + \left(c \bar{V}-  \bar V ' - \zeta_2  \bar{V} \right) e^{\zeta_2(\xi- \eta)}\bigg{|}^{\infty}_{ \bar\xi_3} \right) \end{align}

where we used that $\zeta_1$ and $\zeta_2$ are solutions to \eqref{zeta_opperator_def}. Nothing that $\zeta_2-\zeta_1=\dfrac{\sqrt{c^2+4\alpha}}{2}$, and further simplifying yields 
\begin{align}
     T_3[(B,I,V)] \le & \dfrac{1}{\sqrt{c^2+4 \alpha}} \left( (\zeta_2-\zeta_1 ) \bar{V} - \bar{V}'(\bar{\xi}_3-)e^{\zeta_2(\xi - \bar{\xi}_3)}+  \bar{V}'(\bar{\xi}_3+)e^{\zeta_2(\xi - \bar{\xi}_3)} \right) \\ =& \bar{V}(\xi) -\dfrac{e^{\zeta_2(\xi- \bar{\xi}_1)}}{\sqrt{V^2+4\alpha}}\left(\underbrace{\bar{V}'(\bar{\xi}_3-)}_{>0} - \underbrace{\bar{V}'(\bar{\xi}_3+)}_{=0} \right) \le \bar{V}(\xi).
     \end{align}
                \textbf{Step 6:} For $\bar{\xi}_3< \xi$\begin{align}
        T_3[(B,I,V)] \le & \dfrac{1}{\sqrt{c^2+4 \alpha}} \left( \int_{-\infty}^{\bar{\xi}_3} e^{(\xi- \eta)\zeta_1} {\left(\alpha \bar V+\theta\bar I-\gamma\bar V \right)}d \eta \right. \\ & + \left. \int_{\bar{\xi}_3}^{\xi} e^{(\xi- \eta)\zeta_1} {\left((\alpha \bar V+\theta\bar I-\gamma\bar V \right)}d \eta \right. \\&+ \left. \int_{\xi}^{\infty} e^{(\xi- \eta)\zeta_2} {\left((\alpha \bar V+\theta\bar I-\gamma\bar V \right)}d \eta \right).
    \end{align}
    Inserting the equality from \eqref{uppersolutions} and applying similar calculations as the last step, yields
    \begin{align}
          T_3[(B,I,V)] \le & \dfrac{1}{\sqrt{c^2+4 \alpha}} \left( \int_{-\infty}^{\bar\xi_3} e^{(\xi- \eta)\zeta_1} {\left(\alpha \bar V+ c\bar V'-\bar V'' \right)}d \eta \right. \\ & + \left. \int_{\bar\xi_3}^{{\xi}} e^{(\xi- \eta)\zeta_1} {\left((\alpha \bar V+ c\bar V' \bar V'' \right)}d \eta \right. \\&+ \left. \int_{\xi}^{\infty} e^{(\xi- \eta)\zeta_2} {\left((\alpha \bar V+ c\bar V'-\bar V'' \right)}d \eta \right)\\ =& \bar{V}(\xi) - \dfrac{e^{\zeta_2(\xi- \bar{\xi}_3)}}{\sqrt{V^2+4 \alpha}}\left(\underbrace{\bar{V}'(\bar{\xi}_3-)}_{>0} - \underbrace{\bar{V}'(\bar{\xi}_3+)}_{=0} \right) \le \bar{V}(\xi).
    \end{align}
\noindent\textbf{Step 7:} For the lower bound of $T_1$ we again note that the condition \eqref{alpha_condition} makes $T_1$ non-decreasing in $B$, thus  
 \begin{align}
        T_1[(B,I,V)] = &\dfrac{1}{\sqrt{c^2+4 \alpha D}} \left( \int_{-\infty}^{\xi} e^{(\xi- \eta)\delta_1} {\left(\alpha {B}-{B}+{B^2}-{BI}-{BV}+I+V \right)}d \eta \right.\\ & \left. +\int_{\xi}^{\infty} e^{(\xi- \eta)\delta_2} {\left((\alpha {B}-{B}+{B^2}-{BI}-{BV}+I+V\right)}d \eta \right)\\ 
      \geq &\dfrac{1}{\sqrt{c^2+D4 \alpha}} \left( \int_{-\infty}^{\xi} e^{(\xi- \eta)\delta_1} {\left(\alpha \underline{B}-\underline{B}+\underline{B}^2-\underline{BI}-\underline{BV} +\underline{I}+\underline{V}\right)}d \eta \right. \\ &\left.  +\int_{\xi}^{\infty} e^{(\xi- \eta)\delta_2} {\left((\alpha \underline{B}-\underline{B}+\underline{B}^2-\underline{BI}-\underline{BV} +\underline{I}+\underline{V}\right)}d \eta \right).
    \end{align}
Inserting the inequality \eqref{lowersolutions} yields 
  \begin{align}
          T_1[(B,I,V)] \geq & \dfrac{1}{\sqrt{c^2+4D \alpha}} \left( \int_{-\infty}^{\xi} e^{(\xi- \eta)\delta_1} {\left(\alpha \underline{B}+ c\underline{B}'-D\underline{B}'' \right)}d \eta \right. \\ & + \left. \int_{\xi}^{\infty} e^{(\xi- \eta)\delta_2} {\left((\alpha \underline{B}+ c\underline{B}'-D\underline{B}'' \right)}d \eta \right.\\ =& 0\ge\underline{B}(\xi).
    \end{align}

\noindent\textbf{Step 8:} For the lower bound of $T_2$, let $\xi < \underline{\xi_2}$, using the inequality \eqref{lower_lemma_inequalitoes} yields
\begin{align}
        T_2[(B,I,V)] = &\dfrac{1}{\sqrt{c^2+4 \alpha D}} \left( \int_{-\infty}^{\xi} e^{(\xi- \eta)\sigma_1} {\left(\alpha I-aI+(1-{B})V \right)}d \eta \right. \\ & \left.+  \int_{\xi}^{\infty} e^{(\xi- \eta)\sigma_2} {\left((\alpha I-aI+(1-{B})V \right)}d \eta \right)\\ 
      \geq& \dfrac{1}{\sqrt{c^2+D4 \alpha}} \left( \int_{-\infty}^{\xi} e^{(\xi- \eta)\sigma_1} {\left(\alpha \underline{I}-a\underline{I}+(1-\underline{B})\underline{V} \right)}d \eta \right. \\ &\left.+  \int_{\xi}^{\underline{\xi}_2} e^{(\xi- \eta)\sigma_2} {\left((\alpha \underline{I}-a\underline{I}+(1-\underline{B})\underline{V}\right)}d \eta \right)\\ 
      \geq &\dfrac{1}{\sqrt{c^2+D4 \alpha}} \left( \int_{-\infty}^{\xi} e^{(\xi- \eta)\sigma_1} {\left(\alpha \underline{I}+c\underline{I}'-D\underline{I}'' \right)}d \eta \right. \\ &\left. +  \int_{\xi}^{\underline{\xi}_2} e^{(\xi- \eta)\sigma_2} {\left((\alpha \underline{I}+c\underline{I}'-D\underline{I}''\right)}d \eta \right)  \\  =& \underline{I}+ \dfrac{e^{\sigma_2(\xi-\xi_2)}}{\sqrt{c^2+D4 \alpha}}\left(\underbrace{c\underline{I}(\underline{\xi}_2)}_{=0}-D \underbrace{\underline{I}'(\underline{\xi}_2-)}_{<0} - \sigma_2 \underbrace {D {\underline{I}(\underline{\xi}_2)}}_{=0} \right) \ge \underline{I}.
    \end{align}
\textbf{Step 9: }Now let $\xi \ge \underline{\xi_2}$ \begin{align}
      T_2[(B,I,V)] \geq& \dfrac{1}{\sqrt{c^2+D4 \alpha}} \left( \int_{-\infty}^{\xi} e^{(\xi- \eta)\sigma_1} {\left(\alpha \underline{I}-a\underline{I}+(1-\underline{B})\underline{V} \right)}d \eta \right. \\ & \left.+  \int_{\xi}^{\infty} e^{(\xi- \eta)\sigma_2} {\left((\alpha \underline{I}-a\underline{I}+(1-\underline{B})\underline{V}\right)}d \eta \right)\\ \geq&\dfrac{1}{\sqrt{c^2+D4 \alpha}} \left( \int_{-\infty}^{\underline{\xi}_1} e^{(\xi- \eta)\sigma_1} {\left(\alpha \underline{I}-a\underline{I}+(1-\underline{B})\underline{V} \right)}d \eta \right)\\ =& \dfrac{1}{\sqrt{c^2+D4 \alpha}}\left(c \underline{I}(\underbrace{\underline{\xi}_2)}_{=0}- D \underbrace{\underline{I}'(\underline{\xi}_2)}_{<0}-\sigma_1D\underbrace{\underline{I}(\underline{\xi}_2)}_{=0}\right) \geq 0 = \underline{I}(\xi).
\end{align}
\textbf{Step 10:} For the lower bound of $T_3$, let $\xi \le \underline{\xi_3}$, we can use the relation in \eqref{lowersolutions} and obtain
\begin{align}
       & T_3[(B,I,V)] = \dfrac{1}{\sqrt{c^2+4 \alpha }} \left( \int_{-\infty}^{\xi} e^{(\xi- \eta)\zeta_1} {\left(\alpha V+\theta I- \gamma V \right)}d \eta +  \int_{\xi}^{\infty} e^{(\xi- \eta)\zeta_2} {\left((\alpha V+\theta I- \gamma V\right)}d \eta \right)\\ 
     & \geq \dfrac{1}{\sqrt{c^2+4 \alpha}} \left( \int_{-\infty}^{\xi} e^{(\xi- \eta)\zeta_1} {\left(\alpha \underline{V}+ \theta \underline{I}- \gamma \underline{V} \right)}d \eta +  \int_{\xi}^{\underline{\xi}_3} e^{(\xi- \eta)\zeta_2} {\left((\alpha \underline{V}+ \theta \underline{I}- \gamma \underline{V}\right)}d \eta \right)\\ 
     & \geq \dfrac{1}{\sqrt{c^2+4 \alpha}} \left( \int_{-\infty}^{\xi} e^{(\xi- \eta)\zeta_1} {\left(\alpha \underline{V}+c\underline{V}'-\underline{V}'' \right)}d \eta +  \int_{\xi}^{\underline{\xi}_3} e^{(\xi- \eta)\zeta_2} {\left((\alpha \underline{V}+c\underline{V}'-\underline{V}''\right)}d \eta \right)  \\ & = \underline{V}+ \dfrac{e^{\zeta_2(\xi-\xi_3)}}{\sqrt{c^2+4 \alpha}}\left( \underbrace{c\underline{V}(\underline{\xi}_3)}_{=0}- \underbrace{\underline{V}'(\underline{\xi}_3-)}_{<0} - \zeta_2 \underbrace { {\underline{V}(\underline{\xi}_3)}}_{=0} \right) \ge \underline{V}(\xi).
    \end{align}
\textbf{Step 11: }Now let $\xi> \underline{\xi_3}$
\begin{align}
      T_3[(B,I,V)] &\geq \dfrac{1}{\sqrt{c^2+4 \alpha}} \left( \int_{-\infty}^{\xi} e^{(\xi- \eta)\zeta_1} {\left(\alpha \underline{V}+ \theta \underline{I}- \gamma \underline{V} \right)}d \eta +  \int_{\xi}^{\infty} e^{(\xi- \eta)\zeta_2} {\left((\alpha \underline{V}+ \theta \underline{I}- \gamma \underline{V}\right)}d \eta \right)\\& \geq \dfrac{1}{\sqrt{c^2+4 \alpha}} \left( \int_{-\infty}^{\underline{\xi}_3} e^{(\xi- \eta)\zeta_1} {\left(\alpha \underline{V}+c\underline{V}'-\underline{V}'' \right)}d \eta \right) \\ &\ge \dfrac{e^{\zeta_2(\xi-\xi_3)}}{\sqrt{c^2+4 \alpha}}\left( \underbrace{c\underline{V}(\underline{\xi}_3)}_{=0}- \underbrace{\underline{V}'(\underline{\xi}_3-)}_{<0} - \zeta_2 \underbrace { {\underline{V}(\underline{\xi}_3)}}_{=0} \right) \ge 0=\underline{V}.
\end{align}

Thus $\forall (B,I, V) \in \Gamma, \ T[(B,I,V)] \in \Gamma$.
\end{proof}
\color{black}
Next, we show that $T$ is continuous on $\Gamma$ with respect to the norm $\|.\|_\mu $ defined in \eqref{norm_definition}.
\begin{lemma}
    Suppose $c > \bar{c}$. Then for sufficiently small $\mu $, the operator $T: \Gamma  \to \Gamma$ is continuous with respect to the weighted norm $\|.\|_\mu$. 
    \end{lemma}
    \begin{proof}
        Assume $\Phi= (\phi_1, \phi_2, \phi_3)$, $\Psi=(\psi_1, \psi_2, \psi_3) \in \Gamma$ and $\Phi \neq \Psi$. Consider \begin{align}
            &| F_1(\Phi)- F_1(\Psi)| = |\alpha (\phi_1-\psi_1)- (\phi_1-\psi_1)+(\phi_1^2- \psi_1^2) - \phi_1\phi_2+ \psi_1\psi_2- \phi_1\phi_3+ \psi_1\psi_3| \\ & |(\phi_1-\psi_1)(\alpha-1+(\phi_1+\psi_1))- \phi_1\phi_2+ \phi_1\psi_2-\phi_1\psi_2+\psi_1\psi_2- \phi_1\phi_3+ \phi_1\psi_3-\phi_1\psi_3+ \psi_1\psi_3 | \\&= |(\phi_1-\psi_1)(\alpha-1+(\phi_1+\psi_1)- \psi_2-\psi_3)+(\phi_2-\psi_2)(- \phi_1)+ (\phi_3-\psi_3)(-\phi_1)| \\ & \le |\alpha-1+(\phi_1+\psi_1)-\psi_2-\psi_3||\phi_1-\psi_1|+|\phi_1||\phi_2-\psi_2|+|\phi_1||\phi_3-\psi_3| .
        \end{align} 
Since $\Phi$ and $\Psi$ belong to $\Gamma$, all of their components are bounded. Consequently, there exists a constant $\widehat C_1>0$ such that
$|F_1(\Phi)- F_1(\Psi)|e^{-\mu|\xi|} \le \widehat C_1\|\Phi-\Psi\|_\mu$. We have \begin{align}
    |T_1[\Phi]-T_1[\Psi]||e^{- \mu |\xi|}|=& \dfrac{e^{-\mu |\xi|}}{\sqrt{c^2+ 4 D \alpha}}\left|\int_{- \infty}^{\xi}e^{\delta_1(\xi-\eta)}\left(F_1(\Phi)- F_1(\Psi)\right) d \eta \right. \\ & \left. 
 + \int_{\xi}^{\infty }e^{\delta_2(\xi-\eta)}\left(F_1(\Phi)- F_1(\Psi)\right) d \eta  \right| \\ \le& \dfrac{e^{-\mu |\xi|}}{\sqrt{c^2+ 4 D \alpha}}\left(\int_{- \infty}^{\xi}e^{\delta_1(\xi-\eta)+ \mu|\eta|}\left|F_1(\Phi)- F_1(\Psi)\right|e^{-\mu |\eta|} d \eta \right. \\ & \left.
 + \int_{\xi}^{\infty }e^{\delta_2(\xi-\eta)+ \mu|\eta|}\left|F_1(\Phi)- F_1(\Psi)\right|e^{-\mu |\eta|} d \eta  \right) \\ \leq& \dfrac{\widehat{C}_1 e^{-\mu |\xi|}}{\sqrt{c^2+ 4 D \alpha}}\left(\int_{- \infty}^{\xi}e^{\delta_1(\xi-\eta)+ \mu|\eta|} d \eta 
 + \int_{\xi}^{\infty }e^{\delta_2(\xi-\eta)+ \mu|\eta|} d \eta  \right)\|\Phi-\psi\|_\mu.
\end{align}

The integrals are bounded for 
$0 < \mu < \min\{-\delta_1,\, \delta_2\},$
which we assume henceforth. We estimate the integrals separately over the regions 
$\xi < 0$ and $\xi \ge 0$. If $\xi < 0$, then 
\begin{align}
     &  |T_1[\Phi]-T_1[\Psi]||e^{- \mu |\xi|}| \\ &\le\dfrac{\widehat{C}_1 e^{\mu \xi}}{\sqrt{c^2+ 4 D \alpha}}\left( e^{\delta_1 \xi}\int_{- \infty}^{\xi}e^{-(\delta_1+ \mu)\eta} d \eta 
 + e^{\delta_2 \xi}\int_{\xi}^{0 }e^{-(\lambda_2+ \mu)\eta} d \eta + e^{\delta_2 \xi}\int_{0}^{\infty }e^{(\mu-\delta_2)\eta} d \eta\right)\|\Phi-\psi\|_\mu    \\ &= \dfrac{\widehat{C}_1}{\sqrt{c^2+4 D \alpha}} \left( - \dfrac{1}{\mu + \delta_1}+ \dfrac{1-e^{{(\delta_2+\mu)\xi}}}{\mu+ \delta_2}+ \dfrac{e^{(\delta_2+\mu)\xi}}{\delta_2- \mu} \right) \|\Phi-\Psi\|_\mu \\ &\le   \dfrac{\widehat{C}_1}{\sqrt{c^2+4 D \alpha}} \left(\underbrace{ - \dfrac{1}{\mu + \delta_1}+ \dfrac{1}{\mu+ \delta_2}+ \dfrac{1}{\delta_2- \mu}}_{\tilde C_{\xi <0}} \right) \|\Phi-\Psi\|_\mu.
\end{align}
Now, if $\xi \geq 0$, we have \begin{align}
     &  |T_1[\Phi]-T_1[\Psi]||e^{- \mu |\xi|}| \\&\le\dfrac{\widehat{C}_1 e^{-\mu \xi}}{\sqrt{c^2+ 4 D \alpha}}\left( e^{\delta_1 \xi}\int_{- \infty}^{0}e^{-(\delta_1+ \mu)\eta} d \eta 
 + e^{\delta_1 \xi}\int_{0}^{\xi }e^{(\mu-\delta_1)+ \eta} d \eta + e^{\delta_2 \xi}\int_{\xi}^{\infty }e^{(\mu-\delta_2)\eta} d \eta\right)\|\Phi-\psi\|_\mu    \\ &= \dfrac{\widehat{C}_1}{\sqrt{c^2+4 D \alpha}} \left( - \dfrac{e^{(\delta_1-\mu)\xi}}{\mu + \delta_1}+ \dfrac{1-e^{{(\delta_1-\mu)\xi}}}{\mu- \delta_1}+ \dfrac{1}{\delta_2- \mu} \right) \|\Phi-\Psi\|_\mu \\ &\le   \dfrac{\widehat{C}_1}{\sqrt{c^2+4 D \alpha}} \left(\underbrace{ - \dfrac{1}{\mu + \delta_1}+ \dfrac{1}{\mu- \delta_1}+ \dfrac{1}{\delta_2- \mu}}_{\tilde C_{\xi >0}} \right) \|\Phi-\Psi\|_\mu.
\end{align}

Let $C_1=\dfrac{\widehat{C}_1}{\sqrt{c^2+4 D \alpha}}\max\{\tilde C_{\xi>0}, \tilde C_{\xi<0}\}$, thus $|T_1[\Phi]-T_1[\Psi]||e^{- \mu |\xi|}| \le C_1 \|\Phi-\Psi\|_\mu$.

We now establish an analogous estimate for the second component $T_2$. To begin, we examine the integrand
\begin{align}
   \nonumber |F_2(\Phi)- F_2(\Psi)|&=|\alpha(\phi_1-\psi_2)-(\phi_1\phi_3-\psi_1\psi_3)-a(\phi_2-\psi_2)+(\phi_3-\psi_3)|\\ \nonumber & =|(\alpha-a)(\phi_2-\psi_2)-\phi_2\psi_3+\phi_1\psi_3-\phi_1\psi_3+\psi_1\psi_3+(\phi_3-\psi_3)|\\&=|(\phi_1-\psi_1)(-\psi_3)+(\phi_2-\psi_2)(\alpha-a)+(\phi_3-\psi_3)(1-\phi_1)|\\ &\le |\phi_1-\psi_1||\psi_3|+|\phi_2-\psi_2||\alpha-a|+|\phi_3-\psi_3||1-\phi_1|.
\end{align}
        Thus, there exists a constant $\widehat C_2$ such that
\[
\bigl|F_2(\Phi)-F_2(\Psi)\bigr|\, e^{-\mu|\xi|}
   \;\le\; \widehat C_2\,\|\Phi-\Psi\|_\mu .
\]
Turning now to the operator $T_2$, we obtain
\begin{align}
    |T_2[\Phi]-T_2[\Psi]||e^{- \mu |\xi|}|=& \dfrac{e^{-\mu |\xi|}}{\sqrt{c^2+ 4 D \alpha}}\left|\int_{- \infty}^{\xi}e^{\sigma_1(\xi-\eta)}\left(F_2(\Phi)- F_2(\Psi)\right) d \eta 
 \right. \\ & \left. + \int_{\xi}^{\infty }e^{\sigma_2(\xi-\eta)}\left(F_2(\Phi)- F_2(\Psi)\right) d \eta  \right| \\ \le &\dfrac{e^{-\mu |\xi|}}{\sqrt{c^2+ 4 D \alpha}}\left(\int_{- \infty}^{\xi}e^{\sigma_1(\xi-\eta)+ \mu|\eta|}\left|F_2(\Phi)- F_2(\Psi)\right|e^{-\mu |\eta|} d \eta 
 \right. \\ & \left.+ \int_{\xi}^{\infty }e^{\sigma_2(\xi-\eta)+ \mu|\eta|}\left|F_2(\Phi)- F_2(\Psi)\right|e^{-\mu |\eta|} d \eta  \right) \\ \leq &\dfrac{{\widehat C_2} e^{-\mu |\xi|}}{\sqrt{c^2+ 4 D \alpha}}\left(\int_{- \infty}^{\xi}e^{\sigma_1(\xi-\eta)+ \mu|\eta|} d \eta 
 + \int_{\xi}^{\infty }e^{\sigma_2(\xi-\eta)+ \mu|\eta|} d \eta  \right)\|\Phi-\psi\|_\mu.
\end{align}

The integrals are bounded for 
\[
0 < \mu < \min\{-\sigma_{1},\, \sigma_{2}\},
\]
which we assume from now on.  
We estimate the integrals separately on the regions $\xi < 0$ and $\xi \ge 0$.  
If $\xi < 0$, then we have
\begin{align}
     &  |T_2[\Phi]-T_2[\Psi]||e^{- \mu |\xi|}| \\ &\le\dfrac{\widehat{C}_2 e^{\mu \xi}}{\sqrt{c^2+ 4 D \alpha}}\left( e^{\sigma_1 \xi}\int_{- \infty}^{\xi}e^{-(\sigma_1+ \mu)\eta} d \eta 
 + e^{\sigma_2 \xi}\int_{\xi}^{0 }e^{-(\sigma_2+ \mu)\eta|} d \eta + e^{\sigma_2 \xi}\int_{0}^{\infty }e^{(\mu-\sigma_2)\eta} d \eta\right)\|\Phi-\psi\|_\mu    \\ &= \dfrac{\widehat{C}_2}{\sqrt{c^2+4 D \alpha}} \left( - \dfrac{1}{\mu + \sigma_1}+ \dfrac{1-e^{{(\lambda_2+\mu)\xi}}}{\mu+ \sigma_2}+ \dfrac{e^{(\sigma_2+\mu)\xi}}{\sigma_2- \mu} \right) \|\Phi-\Psi\|_\mu \\ &\le   \dfrac{\widehat{C}_2}{\sqrt{c^2+4 D \alpha}} \left(\underbrace{ - \dfrac{1}{\mu + \sigma_1}+ \dfrac{1}{\mu+ \sigma_2}+ \dfrac{1}{\sigma_2- \mu}}_{\tilde C_{\xi <0}} \right) \|\Phi-\Psi\|_\mu.
\end{align}
Now, for the other case $\xi \geq 0$, it gives \begin{align}
     &  |T_2[\Phi]-T_2[\Psi]||e^{- \mu |\xi|}| \\ &\le\dfrac{\widehat{C}_2 e^{-\mu \xi}}{\sqrt{c^2+ 4 D \alpha}}\left( e^{\sigma_1 \xi}\int_{- \infty}^{0}e^{-(\sigma_1+ \mu)\eta} d \eta 
 + e^{\sigma_1 \xi}\int_{0}^{\xi }e^{(\mu-\sigma_1)+ \eta} d \eta + e^{\sigma_2 \xi}\int_{\xi}^{\infty }e^{(\mu-\sigma_2)\eta} d \eta\right)\|\Phi-\psi\|_\mu    \\ &= \dfrac{\widehat{C}_2}{\sqrt{c^2+4 D \alpha}} \left( - \dfrac{e^{(\sigma_1-\mu)\xi}}{\mu + \sigma_1}+ \dfrac{1-e^{{(\sigma_1-\mu)\xi}}}{\mu- \sigma_1}+ \dfrac{1}{\sigma_2- \mu} \right) \|\Phi-\Psi\|_\mu \\ &\le   \dfrac{\widehat{C}_2}{\sqrt{c^2+4 D \alpha}} \left(\underbrace{ - \dfrac{1}{\mu + \sigma_1}+ \dfrac{1}{\mu- \sigma_1}+ \dfrac{1}{\sigma_2- \mu}}_{\tilde C_{\xi >0}} \right) \|\Phi-\Psi\|_\mu.
\end{align}
Let $C_2=\dfrac{\widehat{C}_2}{\sqrt{c^2+4 D \alpha}}\max\{\tilde C_{\xi>0}, \tilde C_{\xi<0}\}$, thus $|T_2[\Phi]-T_2[\Psi]||e^{- \mu |\xi|}| \le C_2 \|\Phi-\Psi\|_\mu$.

Again, we derive a corresponding estimate for the third component $T_{3}$.  
We begin by examining the integrand
\begin{align}
   \nonumber& |F_3(\Phi)- F_3(\Psi)|=|(\alpha+\theta)(\phi_2-\psi_2)-\gamma(\phi_3-\psi_3)|\le |\phi_1-\psi_1||0|+|\phi_2-\psi_2||\theta|+|\phi_3-\psi_3||\gamma|
\end{align}
        Thus, there exists a constant $\widehat C_3$ such that $|F_3(\Phi)-F_3(\Psi)|e^{-\mu|\xi|}\le \widehat{C}_3\|\Phi-\Psi\|_\mu$. Moving to the operator $T_3$ itself, we write
\begin{align}
    |T_3[\Phi]-T_3[\Psi]||e^{- \mu |\xi|}|= &\dfrac{e^{-\mu |\xi|}}{\sqrt{c^2+ 4  \alpha}}\left|\int_{- \infty}^{\xi}e^{\zeta_1(\xi-\eta)}\left(F_3(\Phi)- F_3(\Psi)\right) d \eta \right. \\   & \left.
 + \int_{\xi}^{\infty }e^{\zeta_2(\xi-\eta)}\left(F_3(\Phi)- F_3(\Psi)\right) d \eta  \right| \\ =& \dfrac{e^{-\mu |\xi|}}{\sqrt{c^2+ 4 \alpha}}\left(\int_{- \infty}^{\xi}e^{\zeta_1(\xi-\eta) + \mu|\eta|}\left|F_3(\Phi)- F_3(\Psi)\right|e^{-\mu |\eta|} d \eta 
  \right. + \\ &\left. \int_{\xi}^{\infty }e^{\zeta_2(\xi-\eta)+ \mu|\eta|}\left|F_3(\Phi)- F_3(\Psi)\right|e^{-\mu |\eta|} d \eta  \right) \\ \leq &\dfrac{{\widehat C_3} e^{-\mu |\xi|}}{\sqrt{c^2+ 4 \alpha}}\left(\int_{- \infty}^{\xi}e^{\zeta_1(\xi-\eta)+ \mu|\eta|} d \eta 
 + \int_{\xi}^{\infty }e^{\zeta_2(\xi-\eta)+ \mu|\eta|} d \eta  \right)\|\Phi-\psi\|_\mu.
\end{align}

The integrals are bounded for $0< \mu< \min\{-\zeta_1, \zeta_2\}$ which we assume forthwith. We estimate the integrals separately for $\xi<0$ and $\xi \ge 0$. If $\xi <0$, it follows 
\begin{align}
     &  |T_3[\Phi]-T_3[\Psi]||e^{- \mu |\xi|}| \\ &\le\dfrac{\widehat{C}_3 e^{\mu \xi}}{\sqrt{c^2+ 4  \alpha}}\left( e^{\zeta_1 \xi}\int_{- \infty}^{\xi}e^{-(\zeta_1+ \mu)\eta} d \eta 
 + e^{\zeta_2 \xi}\int_{\xi}^{0 }e^{-(\zeta_2+ \mu)\eta|} d \eta + e^{\zeta_2 \xi}\int_{0}^{\infty }e^{(\mu-\zeta_2)\eta} d \eta\right)\|\Phi-\psi\|_\mu    \\ &= \dfrac{\widehat{C}_3}{\sqrt{c^2+4  \alpha}} \left( - \dfrac{1}{\mu + \zeta_1}+ \dfrac{1-e^{{(\zeta_2+\mu)\xi}}}{\mu+ \zeta_2}+ \dfrac{e^{(\zeta_2+\mu)\xi}}{\zeta_2- \mu} \right) \|\Phi-\Psi\|_\mu \\ &\le   \dfrac{\widehat{C}_3}{\sqrt{c^2+4 \alpha}} \left(\underbrace{ - \dfrac{1}{\mu + \zeta_1}+ \dfrac{1}{\mu+ \zeta_2}+ \dfrac{1}{\zeta_2- \mu}}_{\tilde C_{\xi <0}} \right) \|\Phi-\Psi\|_\mu.
\end{align}
Now, for the other case $\xi \geq 0$, we obtain \begin{align}
     &  |T_3[\Phi]-T_3[\Psi]||e^{- \mu |\xi|}| \\ &\le\dfrac{\hat{C_3} e^{-\mu \xi}}{\sqrt{c^2+ 4  \alpha}}\left( e^{\zeta_1 \xi}\int_{- \infty}^{0}e^{-(\zeta_1+ \mu)\eta} d \eta 
 + e^{\zeta_1 \xi}\int_{0}^{\xi }e^{(\mu-\zeta_1)+ \eta} d \eta + e^{\zeta_2 \xi}\int_{\xi}^{\infty }e^{(\mu-\zeta_2)\eta} d \eta\right)\|\Phi-\psi\|_\mu    \\ &= \dfrac{\hat{C_3}}{\sqrt{c^2+4  \alpha}} \left( - \dfrac{e^{(\zeta_1-\mu)\xi}}{\mu + \zeta_1}+ \dfrac{1-e^{{(\zeta_1-\mu)\xi}}}{\mu- \zeta_1}+ \dfrac{1}{\zeta_2- \mu} \right) \|\Phi-\Psi\|_\mu \\ &\le   \dfrac{\hat{C_3}}{\sqrt{c^2+4  \alpha}} \left(\underbrace{ - \dfrac{1}{\mu + \zeta_1}+ \dfrac{1}{\mu- \zeta_1}+ \dfrac{1}{\zeta_2- \mu}}_{\tilde C_{\xi >0}} \right) \|\Phi-\Psi\|_\mu
\end{align}

Let $C_3=\dfrac{\widehat{C}_3}{\sqrt{c^2+4  \alpha}}\max\{\tilde C_{\xi>0}, \tilde C_{\xi<0}\}$, thus $|T_3[\Phi]-T_3[\Psi]||e^{- \mu |\xi|}| \le C_3 \|\Phi-\Psi\|_\mu$.

Therefore, there exists a constant $C^{*}>0$ such that
\[
\|T[\Phi]-T[\Psi]\|_{\mu}
   \;\le\; C^{*}\,\|\Phi-\Psi\|_{\mu},
\]
and hence $T$ is continuous with respect to the norm $\|\cdot\|_{\mu}$.
    \end{proof}
\begin{lemma}
If $c > \bar{c}$, then $T : \Gamma \to \Gamma$ is compact with respect to the norm $\|\cdot\|_{\mu}$.
\end{lemma}
\begin{proof}
    Let $\Phi \in \Gamma$, then consider 
    \begin{align}
        |F_1(\Phi)|&= |(\alpha-1)\phi_1+\phi_1^2-\phi_1\phi_2-\phi_1\phi_3+\phi_2+\phi_3| \\& =|\phi_1(\alpha-1+\phi_1)+ (1-\phi_1)\phi_2+(1-\phi_1)\phi_3|\\& \le |\alpha+1-\phi_1||\phi_1|+|\phi_2||`1-\phi_1|+|\phi_3||1-\phi_1| \\& \le  |\alpha+1-c_1||\phi_1|+|\phi_2||c_1|+|\phi_3||c_1|  \le M_1,
    \end{align}
since $\phi_1, \phi_2$ and $\phi_3$ are bounded in $\Gamma$. Now consider \begin{align}
    |T_1[\Phi]|&=\dfrac{1}{\sqrt{c^2+ 4 D \alpha}}\left|\int_{- \infty}^{\xi}e^{\delta_1(\xi-\eta)}F_1[\Phi] d \eta 
 + \int_{\xi}^{\infty }e^{\delta_2(\xi-\eta)}F_1[\Phi] d \eta  \right|\\&  \le \dfrac{M_1}{\sqrt{c^2+ 4 D \alpha}}\left|\int_{- \infty}^{\xi}e^{\delta_1(\xi-\eta)} d \eta 
 + \int_{\xi}^{\infty }e^{\delta_2(\xi-\eta)} d \eta  \right|\\& =\dfrac{M_1}{\sqrt{c^2+ 4D \alpha}}\left|\dfrac{-1}{\delta_1}+\dfrac{1}{\delta_2}\right|= \dfrac{M_1}{\sqrt{c^2+ 4D \alpha}}\left|\dfrac{\sqrt{c^2+ 4D \alpha}}{\alpha}\right|= \frac{M_1}{\alpha},
\end{align} where we used the definition of $\delta_1$ and $\delta_2$ from \eqref{lambda_opperator_def}.

Now for the other compartment \begin{align}|F_2(\Phi)|&= |(\alpha-a)\phi_2+(1-\phi_1)\phi_3| \le |\alpha-a||\phi_2|+|1-\phi_1||\phi_3| \le  M_2,
    \end{align}
since $\phi_1, \phi_2$ and $\phi_3$ are bounded in $\Gamma$. Now consider \begin{align}
    |T_2[\Phi]|&=\dfrac{1}{\sqrt{c^2+ 4 D \alpha}}\left|\int_{- \infty}^{\xi}e^{\sigma_1(\xi-\eta)}F_2[\Phi] d \eta 
 + \int_{\xi}^{\infty }e^{\sigma_2(\xi-\eta)}F_2[\Phi] d \eta  \right|\\&  \le \dfrac{M_2}{\sqrt{c^2+ 4 D \alpha}}\left|\int_{- \infty}^{\xi}e^{\sigma_1(\xi-\eta)} d \eta 
 + \int_{\xi}^{\infty }e^{\sigma_2(\xi-\eta)} d \eta  \right|=  \frac{M_2}{\alpha}.
\end{align} 

Similarly, noting that $\phi_2$ and $\phi_3$ are bounded, for the third component we have \begin{align}|F_3(\Phi)|&= |(\alpha-\gamma)\phi_3+\theta\phi_2| \le  M_3.
    \end{align}
Thus, \begin{align}
    |T_3[\Phi]|&=\dfrac{1}{\sqrt{c^2+ 4  \alpha}}\left|\int_{- \infty}^{\xi}e^{\zeta_1(\xi-\eta)}F_3[\Phi] d \eta 
 + \int_{\xi}^{\infty }e^{\zeta_2(\xi-\eta)}F_3[\Phi] d \eta  \right|\\&  \le \dfrac{M_3}{\sqrt{c^2+ 4 \alpha}}\left|\int_{- \infty}^{\xi}e^{\zeta_1(\xi-\eta)} d \eta 
 + \int_{\xi}^{\infty }e^{\zeta_2(\xi-\eta)} d \eta  \right|=  \frac{M_3}{\alpha}.
\end{align}

Now consider 
\begin{align}
\left|\frac{d}{d \xi}T_1[\Phi]\right|&=\left|\dfrac{1}{\sqrt{c^2+ 4 D \alpha}}\left(\delta_1 \int_{- \infty}^{\xi}e^{\delta_1(\xi-\eta)}F_1[\Phi] d \eta 
 + \delta_2\int_{\xi}^{\infty }e^{\delta_2(\xi-\eta)}F_1[\Phi] d \eta \right) \right| \\ & \le \dfrac{M_1}{\sqrt{c^2+ 4 D \alpha}}\left(|\delta_1|\int_{- \infty}^{\xi}e^{\delta_1(\xi-\eta)} d \eta 
 + \delta_2\int_{\xi}^{\infty }e^{\delta_2(\xi-\eta)}  d \eta \right)  \\ & \le \dfrac{M_1}{\sqrt{c^2+ 4 D \alpha}} \left( \frac{|\delta_1|}{-\delta_1}+1\right)=  \dfrac{2M_1}{\sqrt{c^2+ 4 D \alpha}}.
\end{align}
Thus, $\left|\frac{d}{d \xi}T_1[\Phi]\right|\le \frac{2M_1}{\sqrt{c^2+ 4 D \alpha}}$. Similarly, we can find that $\left|\frac{d}{d \xi}T_2[\Phi]\right|\le \frac{2M_2}{\sqrt{c^2+ 4 D \alpha}}$, and $\left|\frac{d}{d \xi}T_3[\Phi]\right|\le \frac{2M_3}{\sqrt{c^2+ 4 \alpha}}$, hence
\begin{align}
    \|\frac{d}{d \xi}T[\phi]\|\le 2 \max \left\{\frac{2M_1}{\sqrt{c^2+ 4 D \alpha}}, \frac{2M_2}{\sqrt{c^2+ 4 D \alpha}}, \frac{2M_3}{\sqrt{c^2+ 4 \alpha}}\right\}.
\end{align}

To obtain a convergent subsequence, we introduce a continuous cut–off function as follows.  
We define the operator
 \begin{align}
    T^n [\Phi](\xi) =  \begin{cases}
       T[\Phi](\xi), & \quad \xi \in [-n,n] \\ 
       T[\Phi](-n), & \quad \xi \in [-\infty,-n] \\T[\Phi](n), & \quad \xi \in [n,\infty] 
    \end{cases}
\end{align}
    which equals with $T$ on the interval $[-n,n]$ and is constant outside this region.

Let $\{\Phi_i\}_{i \ge 0} \subset \Gamma$.  
Since $T$ has uniformly bounded derivatives, the family
\[
T^{n}[\Phi_i](\xi)
\]
is, for each fixed $n$, uniformly bounded and equicontinuous with respect to the norm $\|\cdot\|_{\mu}$.  
By the Arzelà--Ascoli theorem, the sequence $T^{n}[\Phi_i](\xi)$ therefore admits a convergent subsequence.  
Consequently, for each $n$, the operator $T^{n}$ is compact with respect to the norm $\|\cdot\|_{\mu}$.

What remains is to show that, as $n \to \infty$,
$
T^n \longrightarrow T
$
in the operator norm. We have
    \begin{align}
        \|T^n-T\|_{op}&= \underset{\Phi\in \Gamma;\|\Phi\|=1}{\sup} \underset{\xi \in \mathbb{R}}{ \sup } |T^n[\Phi]- T[\Phi]|e^{- \mu |\xi|} \\ &= \underset{\Phi\in \Gamma;\|\Phi\|=1}{\sup} \underset{\xi \in [-\infty,-n]\cup[n, \infty]}{ \sup } |T^n[\Phi]- T[\Phi]|e^{- \mu |\xi|}  \\ &\le \frac{2}{\alpha} \max \{ M_1,M_2,M_3\}e^{-\mu n} \to 0 \quad\text{as} \quad n \to \infty.
    \end{align}
    Thus, by Theorem 3.4.2 in \cite{Hillen_functional}, $T$ is also a compact operator.
\end{proof}
Combining the results of the previous lemmas, we apply Schauder's Fixed Point Theorem \cite{Hillen_functional} to establish the existence of a travelling wave solution for our system of equations \eqref{ODE_B}.
\begin{thm}
    If $c \ge \bar{c}$, then system \eqref{ODE_B} has a non-negative, non-trivial solution $(B(\xi),I(\xi),V(\xi))\in \Gamma $ with $\lim_{\xi \to - \infty}(B(\xi),I(\xi),V(\xi))=(0,0,0)$ and stays bounded.   
\end{thm}
We chose  $\alpha$, defined in \eqref{alpha_condition}, large enough and $\mu
    $, from the definition of the norm \eqref{norm_definition}, small enough. We will prove the theorem in the following steps.
\begin{proof}
    \textbf{Step 1:} For $c > \bar{c}$ defined in \eqref{barc definiton}, the previous lemmas satisfy the requirements of Schauder's fixed point theorem; thus, the operator $T$ has at least one fixed point which we will denote as $(B,I,V) \in \Gamma$ and it satisfies the ODE system \eqref{ODE_B}. 

    \noindent \textbf{Step 2:}For $c=\bar{c}$, we employ the limiting argument of Fang and Zhao \cite{FangJian2009MWfP} (Theorem~3.1). 
Choose a sequence $\{c_n\}_{n\ge1}\subset(\bar{c},\infty)$ such that
\[
\lim_{n\to\infty} c_n=\bar{c}.
\]
For each $c_n$, there exists a travelling–wave solution $(B^n,I^n,V^n)$ of the ODE system \eqref{ODE_B}. 
From the previous lemma, these solutions satisfy the uniform derivative bound
\begin{align*}
\left\|\frac{d}{d\xi}(B^n,I^n,V^n)\right\|
&= \left\|\frac{d}{d\xi}T[(B^n,I^n,V^n)]\right\| \\
&\le 2 \max \left\{
\frac{2M_1}{\sqrt{c_n^2+4D\alpha}},
\frac{2M_2}{\sqrt{c_n^2+4D\alpha}},
\frac{2M_3}{\sqrt{c_n^2+4\alpha}}
\right\} \\
&\le 2 \max \left\{
\frac{2M_1}{\sqrt{\bar c^2+4D\alpha}},
\frac{2M_2}{\sqrt{\bar c^2+4D\alpha}},
\frac{2M_3}{\sqrt{\bar c^2+4\alpha}}
\right\}.
\end{align*}
Moreover, the bounds
\[
0\le B^n\le1,\qquad
0\le I^n\le1,\qquad
0\le V^n\le\frac{\theta}{\gamma}
\]
imply that the sequence $\{(B^n,I^n,V^n)\}$ is uniformly bounded and equicontinuous.
By the Arzelà--Ascoli theorem, there exists a subsequence
$\{(B^{n_k},I^{n_k},V^{n_k})\}$ that converges pointwise to some
$(\tilde B,\tilde I,\tilde V)\in\chi$.
 Consider 
\begin{align}\label{B_k_convergence}
    \lim_{k \to \infty} B^{nk}& = \lim_{k \to \infty} T_1[(B^{nk},I^{nk},V^{nk})] \\ &= \lim_{k \to \infty}\dfrac{1}{\sqrt{c_{nk}^2+ 4 D \alpha}}\left( \int_{- \infty}^{\xi}e^{\sigma_1^{nk}(\xi-\eta)}F_1[(B^{nk},I^{nk},V^{nk})] d \eta \right. \\ & \quad \left.
 + \int_{\xi}^{\infty }e^{\sigma_2^{nk}(\xi-\eta)}F_1[(B^{nk},I^{nk},V^{nk})] d \eta \right) 
\end{align}
where \begin{align}
    \sigma_1^{nk}=\dfrac{c_nk- \sqrt{c_{nk}+ 4D\alpha}}{2 D}, \quad \sigma_2^{nk}=\dfrac{c_nk+ \sqrt{c_{nk}+ 4D\alpha}}{2 D}.
\end{align}
Applying the Lebesgue Dominated Convergence Theorem to the above limit, we obtain
\[
\tilde B = T_{1}\bigl[(\tilde B,\tilde I,\tilde V)\bigr].
\]
Analogous arguments yield
\[
\tilde I = T_{2}\bigl[(\tilde B,\tilde I,\tilde V)\bigr],
\qquad
\tilde V = T_{3}\bigl[(\tilde B,\tilde I,\tilde V)\bigr].
\]
Consequently,
\[
(\tilde B,\tilde I,\tilde V)
= T\bigl[(\tilde B,\tilde I,\tilde V)\bigr],
\]
and hence $(\tilde B,\tilde I,\tilde V)$ is a nonnegative solution of the ODE
system~\eqref{ODE_B}.

\noindent \textbf{Step 3:}
Note that for every $c>\bar c$ $(B,I,V)$ are bounded above by their corresponding upper
solutions. 
\begin{align}
    0 &\le B(\xi) \le \bar{B}(\xi)
      = e^{\xi \rho}\,
        \frac{c\rho + a + 1 - D\rho^{2}}
        {(c\rho + a - D\rho^{2})(c\rho + 1 - D\rho^{2})}, \\
    0 &\le I(\xi) \le \bar{I}(\xi)
      = e^{\xi \rho}\,
        \frac{1}{c\rho + a - D\rho^{2}}, \\
    0 &\le V(\xi) \le \bar{V}(\xi)
      = e^{\xi \rho}.
\end{align}
In particular when $\xi \to - \infty$  we have 
\begin{align}
    0 \le \lim_{\xi\to -\infty} B(\xi)
      &\le \lim_{\xi\to -\infty} \bar{B}(\xi) = 0,\\
    0 \le \lim_{\xi\to -\infty} I(\xi)
      &\le \lim_{\xi\to -\infty} \bar{I}(\xi) = 0,\\
    0 \le \lim_{\xi\to -\infty} V(\xi)
      &\le \lim_{\xi\to -\infty} \bar{V}(\xi) = 0.
\end{align}

\medskip

Now consider the critical case $c = \bar c$.  
Let $\{c_n\}_{n\ge1 } \subset (\bar{c}, c_N]$ for some fixed $N$ be any decreasing sequence with
\[
c_n > \bar c, \qquad \lim_{n\to\infty} c_n = \bar c.
\]
For each $c_n$, let $(B^n, I^n, V^n)$ denote the corresponding travelling–wave
solution. Each component is uniformly bounded above by its upper solution:
\[
0 \le B^n(\xi) \le \bar B^n(\xi), \qquad
0 \le I^n(\xi) \le \bar I^n(\xi), \qquad
0 \le V^n(\xi) \le \bar V^n(\xi).
\]

Since $\bar B^n(\xi) \to 0$ as $\xi\to -\infty$ for each fixed $n$. Now we consider the
pointwise supremum
\[
\overline{\overline{B}}(\xi)
    := \limsup_{n\to\infty} \bar B^n(\xi).
\]
If 
\[
\lim_{\xi\to -\infty} \overline{\overline{B}}(\xi) > 0,
\]
then for each index $n>0$ there would be a $k>n$ such that 
\[
\lim_{\xi\to -\infty} \bar B^{k}(\xi)> 0,
\]
contradicting the fact that $e^{\xi\rho}\to 0$ as $\xi\to -\infty$.  
Hence,
\[
\lim_{\xi\to -\infty} \overline{\overline{B}}(\xi)=0.
\]

Therefore,
\[
0 \le \lim_{\xi\to -\infty} B^{n}(\xi) = 0.
\]
An identical argument applied to $I^n$ and $V^n$ shows that
\[
\lim_{\xi\to -\infty} (B^{n}(\xi), I^{n}(\xi), V^{n}(\xi))
    = (0,0,0).
\]
Thus the limiting travelling wave at $c=\bar c$ connects to $(0,0,0)$ as
$\xi\to -\infty$.

\noindent {\bf Step 4:} Now we consider the other end of the interval i.e. $\xi \to \infty$. For all $c>\bar{c}$,
the corresponding solutions are uniformly bounded between their upper and lower
solutions. To treat the critical case $c=\bar{c}$, let
$\{c_n\}_{n\ge 1}\subset(\bar{c},c_N]$, for some fixed $N$, be a decreasing sequence
such that
\[
c_n>\bar{c}, \qquad \lim_{n\to\infty} c_n=\bar{c}.
\]
For each $c_n$, let $(B^n,I^n,V^n)$ denote the associated travelling–wave solution.
Each component is uniformly bounded by its corresponding upper and lower
solutions, namely
\[
0 \le B^n(\xi) \le \bar B^n(\xi) \le 1, \quad
0 \le \underline{I}_n(\xi) \le I^n(\xi) \le \bar I^n(\xi) \le 1, \quad
0 \le \underline{V}_n(\xi) \le V^n(\xi) \le \bar V^n(\xi) \le \frac{\theta}{\gamma}.
\]

By the convergence result in~\eqref{B_k_convergence}, the sequence
$\{(B^n,I^n,V^n)\}_{n\ge1}$ admits a limit
\[
\lim_{n\to\infty}(B^n,I^n,V^n)=(\tilde B,\tilde I,\tilde V),
\]
and this limit inherits the same bounds. Consequently,
$(\tilde B,\tilde I,\tilde V)$ is a travelling–wave solution of~\eqref{ODE_B} with
wave speed $\bar{c}$.




\noindent {\bf Step 5:} In the last step, we need to show there are no nontrivial solutions with \(B\equiv 0\) and \((I,V)\not\equiv(0,0)\).
Assume, for contradiction, that \(B(\xi)\equiv 0\) for all \(\xi\) while \((I,V)\) is not identically zero.  
Substituting \(B\equiv 0\) into the first equation of \eqref{eqn:mainPDE_B} yields
\[
0 = I(\xi)+V(\xi)\qquad\text{for all }\xi.
\]
In our biological model \(I\) and \(V\) are nonnegative functions. Hence \(I(\xi)\ge 0\) and \(V(\xi)\ge 0\) for all \(\xi\), and the identity \(I+V\equiv 0\) forces \(I\equiv 0\) and \(V\equiv 0\). This contradicts the assumption that \((I,V)\) is not identically zero. Therefore \(B\) cannot be identically zero.

\end{proof}
\section{Numerics}\label{sec:numerics}
In this section we show how our theoretical minimum wave speeds compare to the numerically observed wave speeds. To address this, in Figure \ref{thetawaves_steady}, we present travelling–wave simulations of the model, carried out in \textsc{Matlab} using the \texttt{pdepe} solver.
On the left column of Figure \eqref{thetawaves_steady}, we run the simulations using the parameter values used previously by \cite{arwa}. In that case the parameter $a$ was greater than one and for $\gamma$ they used the value $\frac{40}{3}$. We compare the difference in wave formation for smaller $a$ value in the right column of Figure \ref{thetawaves_steady}. 
\begin{figure}
    \centering
    \begin{subfigure}{0.49\textwidth}
        \centering
        \includegraphics[width=\textwidth]{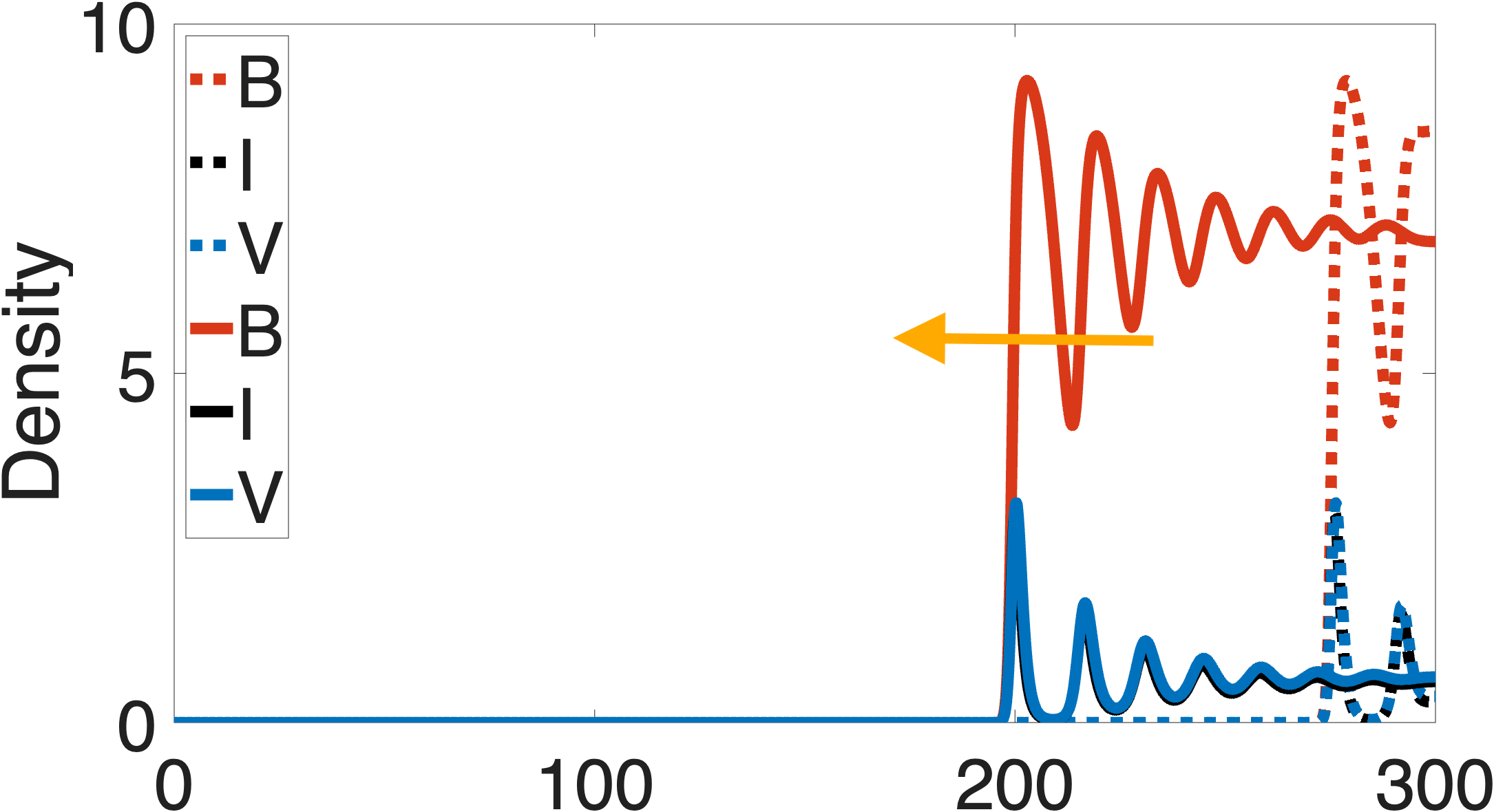}
        \caption{}
        \label{theta25}
    \end{subfigure}
    \begin{subfigure}{0.48\textwidth}
        \centering
\includegraphics[width=\textwidth]{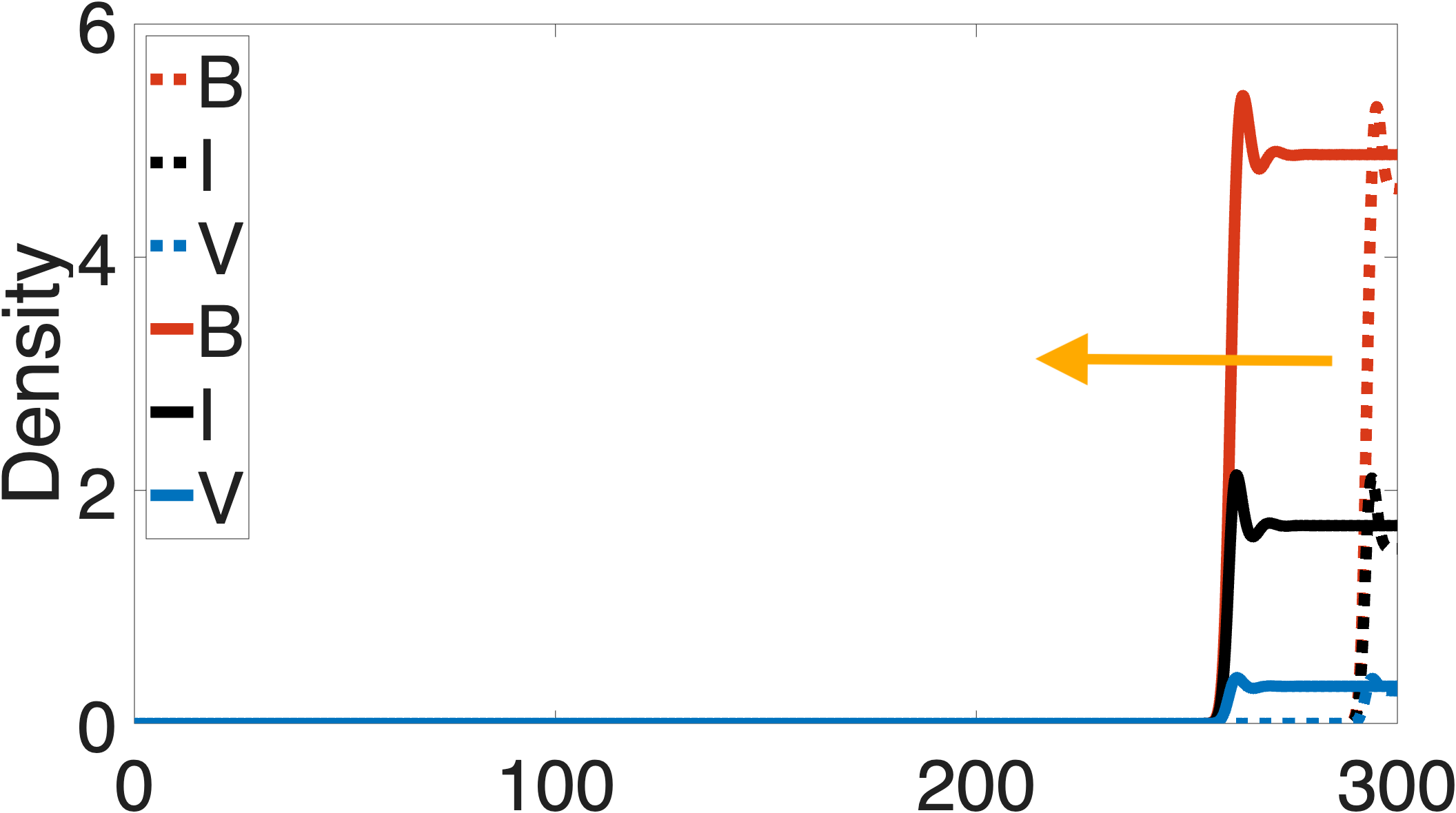}
        \caption{}
        \label{theta25large}
    \end{subfigure}
     \begin{subfigure}{0.49\textwidth}
        \centering
        \includegraphics[width=\textwidth]{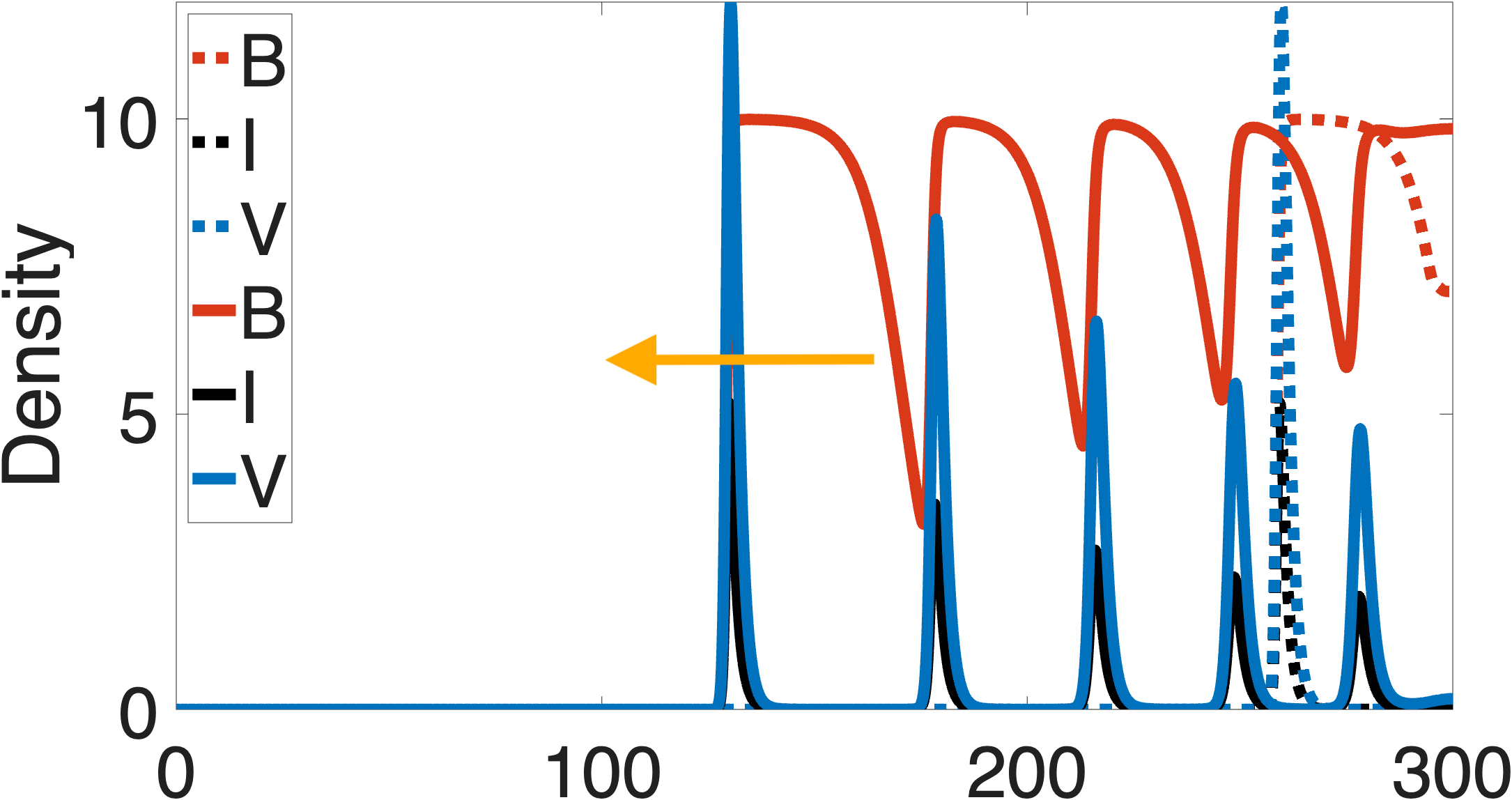}
        \caption{}
        \label{theta150}
    \end{subfigure}
    \begin{subfigure}{0.49\textwidth}
        \centering
        \includegraphics[width=\textwidth]{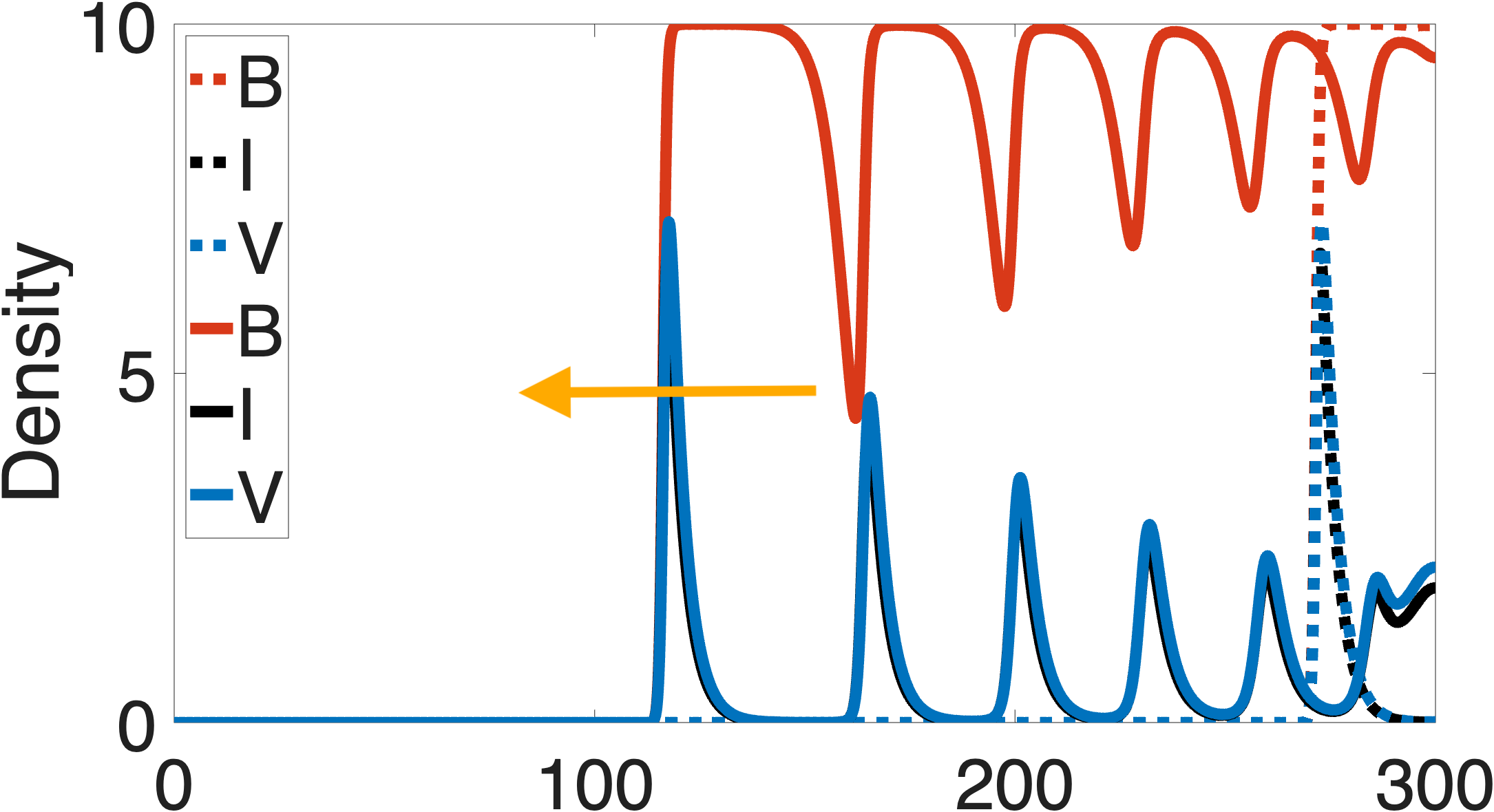}
        \caption{}
        \label{theta150large}
    \end{subfigure}
     \begin{subfigure}{0.49\textwidth}
        \centering
        \includegraphics[width=\textwidth]{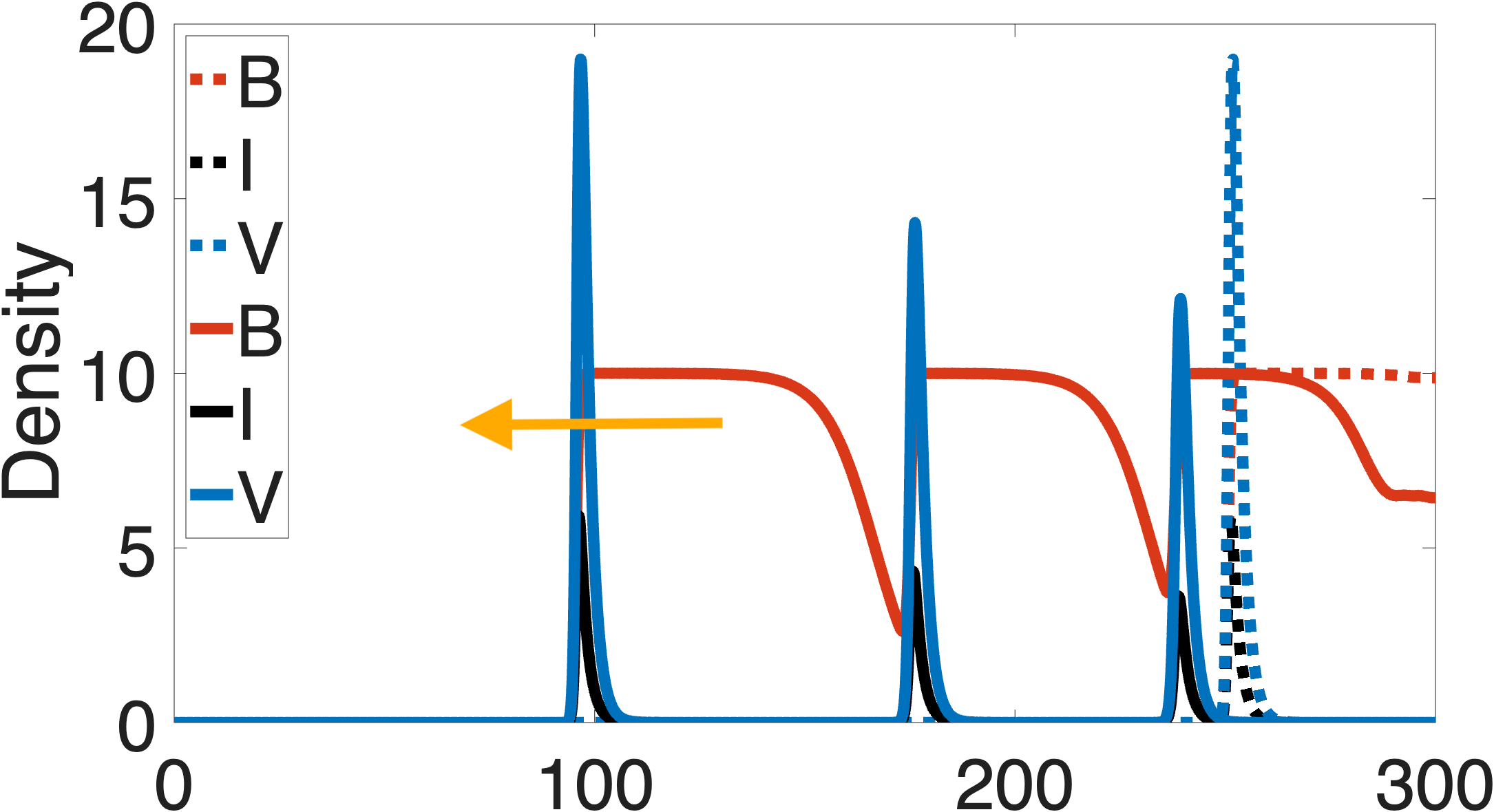}
        \caption{}
        \label{theta250}
    \end{subfigure}
    \begin{subfigure}{0.49\textwidth}
        \centering
        \includegraphics[width=\textwidth]{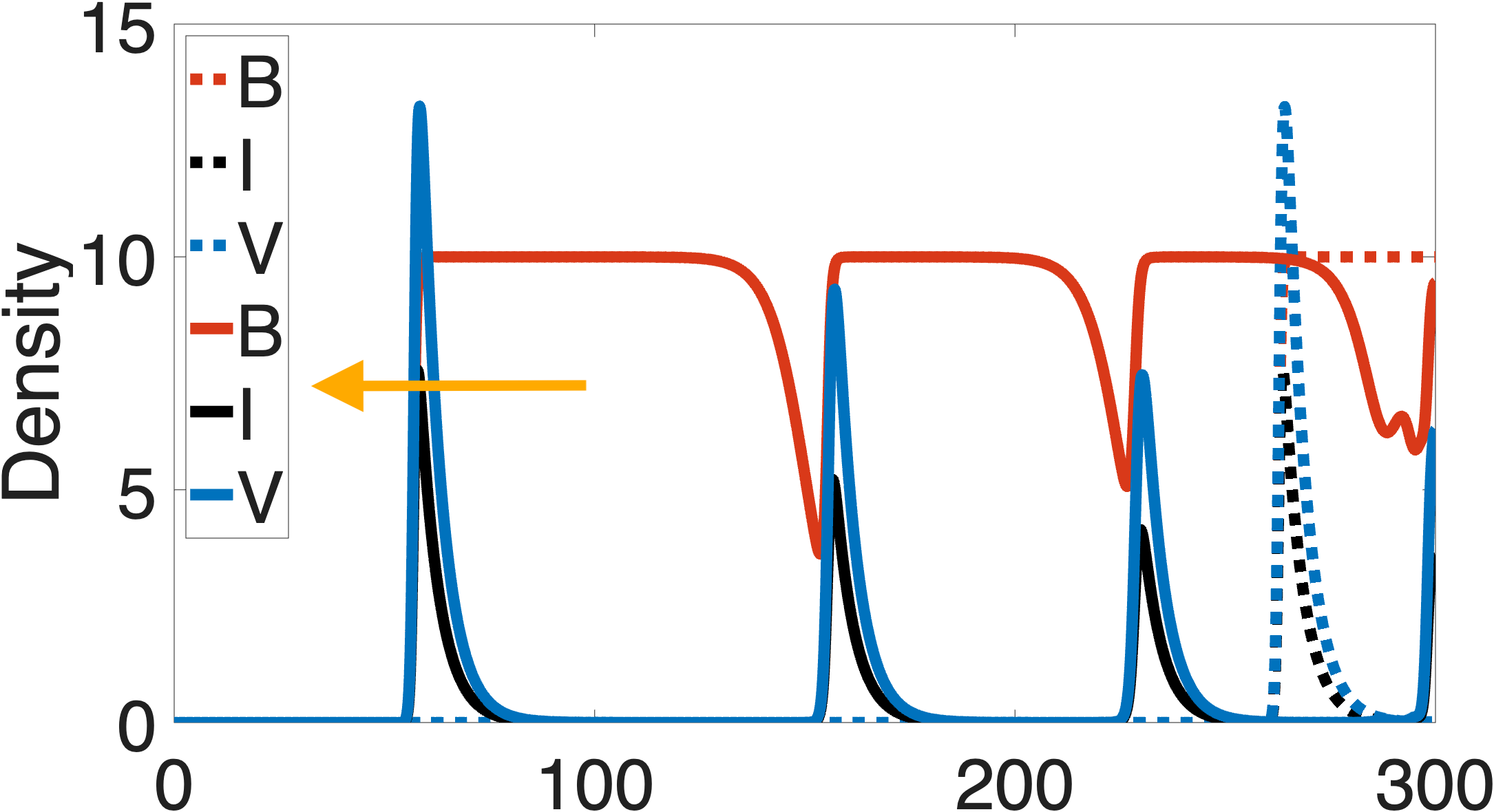}
        \caption{}
        \label{theta250large}
    \end{subfigure}
  \caption{ 
Solutions of the full model \eqref{eqn:main_PDE} for two different values of the parameter $a$. We are showing the solutions as $(B,I,V)$ where $B=1-C$.    
The left column displays simulations for $a=\frac{10}{3}>1$ at times $t = 8$ (dotted lines) and $t = 34$ (solid lines).  
Panel~(a) ($\theta = 150$) exhibits no oscillations, while Panel~(e) ($\theta = 350$) shows oscillatory behaviour has started  
Panel~(c) ($\theta = 500$) illustrates the dynamics in a oscillatory regime with longer oscillations. The right column shows numerical solutions for the case $a<1$ at times $t = 8$ (dotted lines) and $t = 54$ (solid lines), illustrating the formation of oscillatory dynamics behind the travelling wave.  
For visual clarity, the variables $B$ and $I$ have been multiplied by a factor of $10$. 
Panel~(b) ($\theta = 25$) exhibits no oscillations, whereas Panel~(d) ($\theta = 150$) shows oscillations converging toward the coexistence equilibrium.  
Panel~(f) ($\theta = 250$) corresponds to a regime beyond the Hopf bifurcation, where oscillatory patterns arise.
}
    \label{thetawaves_steady}
\end{figure}

We compare the numerical wave speeds $c_N$ with the theoretical one $\bar{c}$. As shown in Table~\ref{tab:speed_comparison}, the numerical speed agrees more closely with the minimum $c_m$  of the theoretical characteristic curve $S(\rho)$ in most cases, for both $a>1$ and $a<1$. In some parameter regimes, the agreement is less precise, but this is expected: the assumptions required for the theorem and for constructing the upper and lower solutions prevent us from fully resolving the interval between $c_m$ and $\tilde{c}$ and a foggy regime exists between the two. This uncertainty contributes to the discrepancy between the numerically observed wave speeds and the theoretical minimum speed $\bar{c}$.

\begin{table}[h!]
\footnotesize
\centering
\begin{tabular}{|c|c|c|c|c|}
\hline
\textbf{Case} & $\theta$ & \textbf{Numerical $c$} & minimum of $S(\rho)$ \textbf{$c_m$} & theoretical speed \textbf{$\bar{c}$} \\
\hline

\multirow{3}{*}{$a<1$}
    & $\theta=25$  & $c_N=0.705$  & $c_m=0.72$   & $\bar c=0.9687$ \\ 
    & $\theta=150$ & $c_N=3.362$  & $c_m=3.2457$ & $\bar c=3.225$  \\
    & $\theta=250$ & $c_N=4.4422$ & $c_m=4.2324$ & $\bar c=4.2324$ \\
\hline

\multirow{3}{*}{$a>1$}
    & $\theta=150$ & $c_N=2.851$  & $c_m=2.8472$ & $\bar c=2.853$  \\
    & $\theta=350$ & $c_N=4.982$  & $c_m=4.6961$ & $\bar c=4.6961$ \\
    & $\theta=500$ & $c_N=4.4422$ & $c_m=5.5403$ & $\bar c=5.5403$ \\
\hline

\end{tabular}
\caption{Comparison of theoretical and numerical wave speeds for $a>1$, and $a<1$.}
\label{tab:speed_comparison}
\end{table}

\section{Conclusion}
In this paper, we study a reaction–diffusion model designed to capture the spatiotemporal dynamics of oncolytic virotherapy (OVT) and tumor cell populations. The model, originally proposed by Baabdulla et al.~\cite{arwa}, describes the interactions between uninfected cancer cells, infected cancer cells, and free virions. A central open question in OVT is whether the virus, when administered as a monotherapy, can spread through and eradicate the tumor before the tumor itself invades the surrounding tissue. Mathematical modelling provides a natural framework to explore this question by comparing the invasion speed of the virus with that of the cancer cells, and by identifying which viral parameters could be enhanced to tip the competition in favour of successful tumor clearance.
In this work, we rigorously establish the existence of traveling–wave solutions for this non-cooperative three-component system, proving that such waves exist for all wave speeds $c\ge \bar{c}$
 . While traveling waves have been extensively studied for both cooperative and non-cooperative systems, the present model poses additional challenges due to its three interacting variables and non-monotone structure, which impose technical restrictions in the construction of upper and lower solutions. As a result, the proof naturally leads to parameter regimes where existence can be established and others where the status of travelling waves remains unclear. These ‘‘foggy’’ regions highlight interesting mathematical questions for future investigation, particularly regarding whether the remaining gaps can be closed by refining the analytical framework.
Finally, to complement the theoretical analysis, we conduct numerical simulations that illustrate the full PDE dynamics, confirm the formation of travelling invasion fronts, and demonstrate how wave propagation depends on the viral parameters. Together, the analytical and numerical results provide a deeper understanding of the mechanisms that influence the success of OVT as a monotherapy and point toward avenues for optimizing viral design.}

\vspace{1.3cm}
\noindent {\bf Acknowledgements.} NM acknowledges the funding support from the University of Alberta. TH is supported through a Discovery Grant of the Natural Science and Engineering Research Council of Canada (NSERC), RGPIN-2023-04269.\\
\\
\noindent
\textbf{Data Availability.} No data were generated or analyzed during the course of this study, so data sharing is not
applicable.
\\
\\
\noindent {\bf Conflict of Interest.}
The authors declare that they have no competing interests.

\bibliographystyle{plainnat}  
\bibliography{ref}  

\end{document}